\documentclass[10pt,letterpaper]{amsart}
\usepackage[utf8]{inputenc}

\usepackage{amsmath}
\usepackage{amsfonts}
\usepackage{amssymb}
\usepackage{graphicx}
\usepackage[left=2.4cm,right=2.4cm,top=2.5cm,bottom=2.5cm]{geometry}
\numberwithin{equation}{section}
\graphicspath{ {images/} }
\usepackage{amsmath,amssymb, amsthm} 
\usepackage{enumerate}
\usepackage{lipsum}

\makeatletter
\def\@settitle{\begin{center}%
  \baselineskip14\p@\relax
  \bfseries
  \uppercasenonmath\@title
  \@title
  \ifx\@subtitle\@empty\else
     \\[1ex]\uppercasenonmath\@subtitle
     \footnotesize\mdseries\@subtitle
  \fi
  \end{center}%
}

\newcommand\blfootnote[1]{%
  \begingroup
  \renewcommand\thefootnote{}\footnote{#1}%
  \addtocounter{footnote}{-1}%
  \endgroup
}

\usepackage{tikz-cd}
\usepackage{hyperref}
\usepackage{bm}
\hypersetup{colorlinks, linkcolor=blue, citecolor=black, urlcolor=blue}

\usepackage[english]{babel}
\usepackage[colorinlistoftodos]{todonotes}

\theoremstyle{plain}
\newtheorem{thm}{Theorem}[subsection] 

\usepackage{mathrsfs}

\theoremstyle{definition}
\newtheorem{defi}[thm]{Definition}

\newtheorem{rmk}[thm]{Remark}

\theoremstyle{definition}

\theoremstyle{plain}
\newtheorem{prop}[thm]{Proposition}
\theoremstyle{plain}
\newtheorem{lemma}[thm]{Lemma}
\theoremstyle{plain}
\newtheorem{cor}[thm]{Corollary}
\theoremstyle{plain}
\newtheorem*{thmintro}{Theorem}
\newtheorem*{rmkintro}{Remark}


\newcounter{parentnumber}

\DeclareMathOperator{\ord}{ord}

\newcommand{\Hom}{\operatorname{Hom}}

\newcommand{\res}{\operatorname{Res}}

\newcommand\SmallMatrix[1]{{%
  \tiny\arraycolsep=0.3\arraycolsep\ensuremath{\begin{pmatrix}#1\end{pmatrix}}}}

\makeatletter  
\newcommand{\colim@}[2]{%
  \vtop{\m@th\ialign{##\cr
    \hfil$#1\operator@font colim$\hfil\cr
    \noalign{\nointerlineskip\kern1.5\ex@}#2\cr
    \noalign{\nointerlineskip\kern-\ex@}\cr}}%
}
\newcommand{\colim}{%
  \mathop{\mathpalette\colim@{\rightarrowfill@\scriptscriptstyle}}\nmlimits@
}
\renewcommand{\varprojlim}{%
  \mathop{\mathpalette\varlim@{\leftarrowfill@\scriptscriptstyle}}\nmlimits@
}
\renewcommand{\varinjlim}{%
  \mathop{\mathpalette\varlim@{\rightarrowfill@\scriptscriptstyle}}\nmlimits@
}
\makeatletter  
\newcommand{\limm@}[2]{%
  \vtop{\m@th\ialign{##\cr
    \hfil$#1\operator@font Rlim$\hfil\cr
    \noalign{\nointerlineskip\kern1.5\ex@}#2\cr
    \noalign{\nointerlineskip\kern-\ex@}\cr}}%
}
\newcommand{\limm}{%
  \mathop{\mathpalette\limm@{\leftarrowfill@\scriptscriptstyle}}\nmlimits@
}
\makeatletter  
\newcommand{\ccolim@}[2]{%
  \vtop{\m@th\ialign{##\cr
    \hfil$#1\operator@font colim$\hfil\cr
    \noalign{\nointerlineskip\kern1.5\ex@}#2\cr
    \noalign{\nointerlineskip\kern-\ex@}\cr}}%
}
\newcommand{\ccolim}{%
  \mathop{\mathpalette\ccolim@{\rightarrowfill@\scriptscriptstyle}}\nmlimits@
}
\makeatletter  
\newcommand{\llimm@}[2]{%
  \vtop{\m@th\ialign{##\cr
    \hfil$#1\operator@font lim$\hfil\cr
    \noalign{\nointerlineskip\kern1.5\ex@}#2\cr
    \noalign{\nointerlineskip\kern-\ex@}\cr}}%
}
\newcommand{\llimm}{%
  \mathop{\mathpalette\llimm@{\leftarrowfill@\scriptscriptstyle}}\nmlimits@
}
\makeatother

\newcommand{\N}{\mathfrak{N}}
\newcommand{\M}{\mathfrak{M}}
\newcommand{\MM}{\mathbb{M}}

\newcommand{\Z}{\mathbb{Z}}
\newcommand{\Q}{\mathbb{Q}}

\newcommand{\Ff}{\mathbb{F}}

\newcommand{\R}{\mathbb{R}}
\newcommand{\C}{\mathbb{C}}
\newcommand{\A}{\mathbb{A}}

\newcommand{\Oo}{\mathcal{O}}

\newcommand{\I}{\mathcal{I}}

\newcommand{\hh}{\mathcal{H}}

\newcommand{\F}{\mathcal{F}}

\newcommand{\p}{\mathfrak{p}}
\newcommand{\cc}{\mathfrak{c}}
\newcommand{\q}{\mathfrak{q}}

\newcommand{\m}{\mathfrak{m}}
\newcommand{\X}{\mathfrak{X}}

\newcommand{\OG}{\Omega^{(\kappa_1,\kappa_2)}}

\DeclareMathOperator{\GL}{GL}

\DeclareMathOperator{\RG}{R\Gamma}
\DeclareMathOperator{\tor}{tor}

\DeclareMathOperator{\End}{End}
\DeclareMathOperator{\et}{\text{\'{e}}t}

\DeclareMathOperator{\dr}{dR}

\DeclareMathOperator{\tr}{Tr}

\DeclareMathOperator{\id}{Id}
\DeclareMathOperator{\HT}{HT}

\DeclareMathOperator{\myom}{\underline{\omega}}
\newcommand{\adjunction}[4]{\xymatrix@1{#1{\ } \ar@<-0.3ex>[r]_{ {\scriptstyle #2}} & {\ } #3 \ar@<-0.3ex>[l]_{ {\scriptstyle #4}}}}

\makeatletter
\def\@settitle{\begin{center}%
  \baselineskip14\p@\relax
  \bfseries
  \uppercasenonmath\@title
  \@title
  \ifx\@subtitle\@empty\else
     \\[1ex]\uppercasenonmath\@subtitle
     \footnotesize\mdseries\@subtitle
  \fi
  \end{center}%
}
\def\subtitle#1{\gdef\@subtitle{#1}}
\def\@subtitle{}
\makeatother

\begin{document}

\title{Higher Hida theory for Hilbert modular varieties in the totally split case}\blfootnote{\date{\today}}
\author{Giada Grossi}

  \newcommand{\Addresses}{{
  \bigskip
  \footnotesize

\textsc{CNRS, Institut Galilée, Université Sorbonne Paris Nord,
99, avenue Jean-Baptiste Clément
93430 - Villetaneuse}\par\nopagebreak
  \textit{E-mail address}, G.~Grossi: \texttt{grossi@math.univ-paris13.fr}

}}

\begin{abstract}
We study $p$-adic properties of the coherent cohomology of some automorphic sheaves on the Hilbert modular variety $X$ for a totally real field $F$ in the case where the prime $p$ is totally split in $F$. More precisely, we develop higher Hida theory \`{a} la Pilloni, constructing, for $0\leq q\leq [F:\Q]$, some modules $M^q$ which $p$-adically interpolate the ordinary part of the cohomology groups $H^q(X, \underline{\omega}^{\kappa})$, varying the weight $\kappa$ of the automorphic sheaf.
\end{abstract}

\maketitle
\tableofcontents
\addtocontents{toc}{\protect\setcounter{tocdepth}{1}}
\section{Introduction}
The theory of $p$-adic families of ordinary modular forms was introduced by Hida in the '80s and has been proved to be fruitful in many aspects of number theory, such as the construction of $p$-adic $L$-functions or, together with the corresponding theory of Galois deformations, modularity-type results. This theory, later generalised to more general automorphic forms, provides a $p$-adic variation of the degree zero coherent cohomology groups of suitable Shimura varieties: the idea is to use the additional structure of the geometry of the ordinary locus of the Shimura variety to $p$-adically interpolate the automorphic sheaves (whose global sections are automorphic forms). By applying a projector with respect to certain Hecke operators at $p$, one is then able to determine when sections over the ordinary locus come from a classical automorphic form (see \cite{hidabook,hidajussieu,pillonibull} and more recently \cite{zhang}). The same circle of ideas was extended in the '90s by Coleman \cite{col1,col2}, who developed, working on neighbourhoods of the ordinary locus, the finite slope theory. 

The recent pioneering works \cite{pilloni, modularsurf, boxer2020higher, highersiegel,highercoleman} have developed analogous theories for higher degree coherent cohomology. In \cite{pilloni}, Pilloni introduced higher Hida and Coleman theory for automorphic forms for $\operatorname{GSp}_4/\Q$ and these ideas were later generalised in \cite{modularsurf} for $\res_{F/\Q}\operatorname{GSp}_4$, where $F$ is a totally real field in which the prime $p$ totally splits, and used to prove potential modularity of abelian surfaces over $F$. Boxer and Pilloni conjectured the existence of Hida and Coleman theories in all cohomological degrees for all Shimura varieties, confirming this prediction in the simplest case of $\GL_2$ in the recent work \cite{boxer2020higher} and for Sigel modular varieties in \cite{highersiegel}, as well as developing Coleman theory more generally in \cite{highercoleman}. In \cite{pilloni} and \cite{modularsurf}, the integral control theorem for ordinary families is obtained assuming the weights are \emph{big enough}. The control theorem for more general weights is obtained at the cost of inverting $p$ and losing control of torsion, using Coleman theory: the authors of \emph{op. cit.} construct over $\Q_p$ an overconvergent version of the considered complex and develop a theory of finite slope cohomological families, which they can prove to be classical cohomology classes in the small slope situation.

In this paper, we study the theory of $p$-adic ordinary families of Hilbert modular forms, which are automorphic forms for the group $\res_{F/\Q}\GL_2$, where $F$ is a totally real field and we assume the prime $p\ge 5$ to be totally split in $F$. We develop higher Hida theory in this context and, using the ideas of \cite{boxer2020higher}, we are able to prove integral control theorems without appealing to an overconvergent theory. 

Higher Hida (and Coleman) theory, as developed in \cite{pilloni}, has been applied in \cite{padicLfnct} for the construction of $p$-adic $L$-functions and was one of the main ingredients in the proof in \cite{gsp4blochkato} of new cases of the Bloch--Kato conjecture in rank 0 and one divisibility in the Iwasawa main conjecture for the spin motive of automorphic forms for $\operatorname{GSp}_4$. The work carried out in this paper (and more precisely its generalisation to Iwahori level Hilbert modular forms) has similar applications: in \cite{grossi2023padic}, the results on higher Hida theory are used to construct $p$-adic $L$-functions for the Asai motive of Hilbert modular forms (and in current work in progress with A. Graham we obtain similar applications to the twisted triple product associated to a Hilbert modular form and an elliptic modular form), following the strategy of \cite{padicLfnct}. The integral classicality results simplify the techniques of \emph{op. cit.}, where the authors had to rely on the overconvergent results. In order to obtain the application for the Bloch--Kato conjecture (for $F$ a real quadratic field), we plan to prove an explicit reciprocity law, linking such $p$-adic $L$-functions with the Euler system classes studied in \cite{HMS, giada}. 

\subsection{Main results} We now state our main result more precisely. Let $F$ be a totally real field of degree $n$ and $X$ be a smooth toroidal compactification of the Hilbert modular variety for $F$ of level sufficiently small and coprime to $p$. Let $L$ denote the Galois closure of $F$ containing the square roots of the totally positive units of $F$ and let $\Oo_L$ be its ring of integers. Fix $\wp\mid p$ a prime of $L$ and denote by $R$ the ring of integers of the completion of $L$ at $\wp$. Assume that $p$ is odd and splits completely in $F$. The set $\Sigma_{\infty}$ of embeddings of $F$ in $\R$ is then identified with the set of $p$-adic embeddings $F\hookrightarrow \bar{\Q}_p$ and therefore with $\{\p \subset \Oo_F: \p\mid p\}$. Let \[
\Lambda=R[[(\Z_p^{\times}/\{\pm 1\})^{n+1}]]
\]
Any $\underline{k}\in \Z^n,w\in \Z$ gives an algebra homomorphism $(2\underline{k},2w): \Lambda \to R$, induced by the character on $(\Z_p^\times)^{n+1}$ given by $((x_i)_{i=1,\dots,n},y)\mapsto y^{2w}\cdot \prod x_i^{2k_i}$. The main result of the paper is the following.
\begin{thmintro}[Classicality, Theorem \ref{mainthm}]
For any $J\subset \Sigma_{\infty}$, there exists a perfect complex of $\Lambda$-modules $M_J^\bullet$ satisfying the following property: for any $\underline{k}\in \Z^n, w \in \Z$ such that 
$2k_{\p}\leq -1$ for $\p \in J$, $2k_{\p}\geq 3$ for $\p\not\in J$, we have an isomorphism
\[
M_J^\bullet\otimes_{\Lambda,(2\underline{k},2w)}R \simeq e(T_p)\RG(X,\myom^{(2\underline{k},2w)}),
\]
where $\myom^{(2\underline{k},2w)}$ is the automorphic sheaf on $X$ of weight $(2\underline{k},2w)$ and $e(T_p)$ is the ordinary projector with respect to the Hecke operator $T_p$. Moreover, for any $J\subset \Sigma_{\infty}$, there exists a perfect complex of $\Lambda$-modules $N_J^\bullet$, which, for the same range of weights as above, satisfies 
\[
N_J^\bullet \otimes_{\Lambda,(2\underline{k},2w)}R \simeq e(T_p)\RG (X,\myom^{(2\underline{k},2w)}(-D)),
\]
where $D$ is the cuspidal divisor of $X$.

Moreover $M_J^\bullet$ is concentrated in degrees $[ \# J,n]$ and $N_J^\bullet$ is concentrated in degrees $[0, \# J]$.
\end{thmintro}

\begin{rmkintro}[Odd weights]
The complexes $M_J^\bullet$ and $N_J^\bullet$ are obtained using a sheaf of $\Lambda$-modules over the ordinary locus of the Hilbert modular variety. Performing the same construction but ``twisting'' such sheaves by the automorphic sheaf of parallel weight one (see \eqref{eq:twistIgusasheaf}), one can easily construct perfect complexes $M_{1,J}^\bullet$ and $N_{1,J}^\bullet$ concentrated in degrees $[ \# J,n]$ and $[0, \# J]$ respectively satisfying the following: for any $\underline{k}\in \Z^n, w \in \Z$ such that 
$2k_{\p}+1\leq -1$ for $\p \in J$, $2k_{\p}+1\geq 3$ for $\p\not\in J$, we have isomorphisms
\[
M_{1,J}^\bullet\otimes_{\Lambda,(2\underline{k},2w)}R \simeq e(T_p)\RG(X,\myom^{(2\underline{k}+1,2w+1)}),
\]
\[
N_{1,J}^\bullet \otimes_{\Lambda,(2\underline{k},2w)}R \simeq e(T_p)\RG (X,\myom^{(2\underline{k}+1,2w+1)}(-D)).
\]
In this paper, for the sake of simplicity and ease of notation, we will focus on the even weight complexes $M_J^\bullet$ and $N_J^\bullet$ only.
\end{rmkintro} 

Rephrasing the main result above, $M_J^\bullet$ (and $N_J^\bullet$) $p$-adically interpolate the ordinary part of coherent (cuspidal) cohomology of $X$ in a range of even weights depending on $J$. In particular, if $J=\emptyset$, $H^0(N_{\emptyset}^\bullet)$ interpolates ordinary classical holomorphic Hilbert cuspforms of weight $(2\underline{k},2w)$ for $2k_\p\geq 4$ for every $\p$, i.e. it recovers classical Hida theory for Hilbert modular forms. Such theory was developed by Hida in \cite{hidaannals} and \cite{hidaalgebra} with a very different method: the construction in his work is not geometric but it relies on the duality between cuspforms and Hecke algebras and the Jacquet--Langlands correspondence between Hilbert modular forms and quaternionic modular forms. In \cite{hidaannals}, Hida constructs a Hecke algebra over $R[[W]]$, where $W$ is the torsion-free part of the Galois group of the maximal abelian extension of $F$ unramified outside $p$ and $\infty$, fixing $\underline{n}\in \Z^n$, which interpolates the ordinary Hecke algebra of cuspforms of weight $(\underline{k},w)$ for $\underline{k}=w+2\underline{n}$. Later in \cite{hidaalgebra}, he \emph{unifies} these infinitely many Hecke algebras to obtain a universal one (without the restriction on the weight being parallel to $2\underline{n}$). Assuming Leopoldt's conjecture holds true for $F$, the number of variables in Hida's work in the same as in our theorem and both classicality results are for characters of the torus of $\res_{F/\Q}\GL_2(\Z_p)$ factoring through the quotient by the ($p$-adic closure of the) units of $\Oo_F$. However, Hida considers the diagonal embedding of the units, whereas we need to consider, because of the geometric nature of our construction and its moduli space interpretation, the embedding given by $\epsilon \mapsto (\epsilon,\epsilon^2)$, which results in a slightly different formulation of the classicality result (see Remark \ref{rmkhida} for more details). 

It is also important to mention that the geometric theory of $p$-adic Hilbert modular forms was developed, in the overconvergent setting, in various works (see for example \cite{AIS,AIP,tianxiao,kisin}), where the rational classicality results are obtained for finite slope families of degree zero coherent cohomology classes. We also remark that in the past works where sheaves interpolating the automorphic sheaves of classical Hilbert modular forms (\cite{AIS,AIP}) were constructed, this was done for automorphic forms for the group $G^*=G\times_{\det} \mathbb{G}_m$. Hilbert modular forms for the group $G$ are then obtained as the image of the global sections of the sheaves for $G^*$ under a projector for a suitably defined action of the units of $\Oo_F$. In this work we instead descend the interpolating sheaves to sheaves over the toroidal compactification of the Shimura variety for $G$, exploiting the action of the units encoded in the definitions and, since the novelty of our construction is that it also interpolates the \emph{determinant factors} $\wedge^2\mathcal{H}^1(\mathcal{A})_{\tau}^{(w-k_\tau)/2}$, we do not need to add the twist in the unit action as done for example in \cite{AIP} (see the discussion right before Definition 4.1), where the twist by the power $(w-k_\tau)/2$ is added \emph{artificially} (see $\S$\ref{seccomparison} for more details).

Finally, after showing (see Proposition \ref{propisomodnoneis}) that the natural map of $\Lambda$-complexes $N^\bullet_{J}\to M^\bullet_J$ becomes an isomorphism after localising at a non-Eisenstein maximal ideal $\mathfrak{M}$ of the Hecke algebra, we deduce that the localised complexes are concentrated in exactly one degree (namely $\# J$) and are therefore projective $\Lambda$-modules. 
Using this, we also prove that there is a perfect pairing interpolating in the classical weights Serre's duality pairing. More precisely, let $M_J= H^{\# J}(M^{\bullet}_J)_{\mathfrak{M}}$, we show in $\S$\ref{secduality} that we can define a pairing
\[
\langle-,-\rangle: M_J \times M_{J^c}\to \Lambda 
\]
of $\Lambda$-modules (where, in order to be precise, the structure of $\Lambda$-module of $M_{J^c}$ is actually twisted by a certain automorphism of $\Lambda$) which satisfies the following.
\begin{thmintro}[Theorem \ref{thmduality}] The pairing $\langle-,-\rangle$ is a perfect pairing compatible with Serre duality, i.e. for a classical weight $(\underline{k},w)$ as above, the following diagram is commutative
\[
\begin{tikzcd}
M_J\otimes_{\Lambda, (\underline{k},w)} R\arrow{d}{\simeq} &\times &M_{J^c}\otimes_{\Lambda, (2-\underline{k},-w)} R\arrow{d}{\simeq}\arrow{r} &R \\
e(T_p)H^{\#J}(X,\myom^{(2\underline{k},2w)})_{\mathfrak{M}} &\times &e(T_p)H^{n-\#J}(X,\myom^{(2-2\underline{k},-2w)}(-D))_{\mathfrak{M}}\arrow{ur}
\end{tikzcd}
\]
where the bottom pairing is induced by Serre duality and the vertical maps are the one obtained by the classicality theorem.
\end{thmintro}

\subsection{Strategy} We briefly sketch how the complexes $M_J^\bullet$ (and $N_J^\bullet$) are constructed. As explained above, the idea is to construct a sheaf of $\Lambda$-modules over (the formal completion) of the ordinary locus of a fixed smooth toroidal compactification of the Hilbert modular variety. This sheaf is constructed using Igusa towers, which are torsor over the ordinary locus, and recover the classical automorphic sheaves $\myom^{(\underline{k},w)}$ when specialised at weights $(\underline{k},w)$. The complexes $M_{\emptyset}^\bullet$ and $N_{\emptyset}^\bullet$ are obtained simply as the image of the ordinary projector $e(U_p)$ of the cohomology over the ordinary locus of this sheaf. In order to define $M_J^\bullet$ and $N_J^\bullet$ in general we use the divisors $D_\p$ which are the vanishing locus of (various lifts of powers of) the partial Hasse invariants. Then, very roughly, we consider extensions of the sheaf above to the formal completion of the Hilbert modular variety and take cohomology over the complement of $\cup_{\p\not\in J}D_{\p}$ with compact support towards the divisors $D_\p$ for $\p\in J$. Then the desired complexes are obtained by taking the image of this cohomology under the projector with respect to a certain operator $T_J$ given by the composition of the operators $U_\p$ for $\p\not\in J$ and the partial Frobenii $F_\p$ for $\p\in J$. The way we obtain the classicality result is by first working on the special fibre of the variety and prove the classicality result modulo $\wp$ ($\S$\ref{modpsection}, Theorem \ref{modpclassical}). Since the sheaf of $\Lambda$-modules modulo the maximal ideal of $\Lambda$ is simply $\myom^{(\underline{k},w)}$ (for certain choices of $(\underline{k},w)$), this result can be formulated as follows: for certain choices of $(\underline{k},w)$ depending on $J$, the image under $e(T_p)$ of the cohomology of $\myom^{(\underline{k},w)}$ over the complement of $\cup_{\p\not\in J}D_{\p}$ with compact support towards the divisors $D_\p$ for $\p\in J$ is isomorphic to $e(T_p)\RG (X_1, \myom^{(\underline{k},w)})$, where $X_1$ is the special fibre of the Hilbert modular variety $X/R$. The proof of this result relies on the study of the partial $T_{\p}$ operators on the special fibre, once they have been carefully normalised in order to be optimally integral. The last ingredient we need is then to show that the operator $T_J$ specialised at the desired weight $(\underline{k},w)$ is congruent to the operator $T_p$ modulo $\wp$. 

The vanishing result of Theorem \ref{mainthm} is proved again by reducing it to a vanishing result of the cohomology modulo $\wp$ and using a filtration by $\# J$-strata of the complement of $\cup_{\p\not\in J}D_{\p}$, such that the complement of each stratum in the previous one is affine (in the minimal compactification).

The whole construction has various technical difficulties coming from the Hilbert modular variety not being a Shimura variety of PEL type. It is however a union of moduli spaces of Hilbert-Blumental abelian varieties (with prescribed polarisations) quotiented out by the action of the totally positive units of $\Oo_F$. Hence we often give definitions and constructions for the moduli space and then need to check how they behave with respect to this action in order to show that they descend to the Shimura variety (see for example Definition \ref{defiunits} and $\S$ \ref{secdescentigusa}).

\subsection{Outline of the paper} We recall in $\S$\ref{secprelim} the preliminaries on Hibert modular varieties, their compactifications and the automorphic vector bundles over them.

In $\S$\ref{secheckeandhasse} we define the partial Hecke operators $T_\p$ acting on cohomology of the automorphic vector bundles and normalise them so that they are optimally integral; we also recall the definition of the partial Hasse invariants on the special fibre of the Hilbert modular variety and how (certain powers of) these sections lift modulo powers of $\wp$.

The main constructions are carried out in $\S$\ref{sechida}, where we first work on the special fibre ($\S$\ref{modpsection}) and then move on (in $\S$\ref{char0sec}) to the Igusa tower constructions on the formal completion of the Hilbert modular variety and the proof of the classicality result. 

We finally construct the duality pairing in $\S$\ref{secduality} and prove its compatibility with Serre duality. 

\subsubsection*{Acknowledgements.} 
I would like to thank Vincent Pilloni for his seminal work on higher Hida theory, from which this article originates from. I thank both him and George Boxer for helpful discussions and explanations on their work. I am also grateful to Jacques Tilouine for many useful remarks and conversations. I thank David Loeffler and Sarah Zerbes for their encouragement and valuable comments and discussions. I also thank Ana Caraiani, Mladen Dimitrov and the anonymous referee for pointing out the issue about projectivity. Finally, I would like to express my gratitude to the anonymous referee for their valuable and thorough comments. The author was partially supported by the postdoctoral fellowship of the Fondation Sciences Mathématiques de Paris.
\addtocontents{toc}{\protect\setcounter{tocdepth}{2}}

\section{Preliminaries}\label{secprelim}
\subsection{Hilbert modular varieties and moduli interpretation}
Let $F$ be a totally real field of degree $n$. We consider $G:= \res_{F/\Q}\GL_2$.


Consider $K$ a neat open compact subgroup of $G(\A_f)$ and let
\begin{displaymath}
Y_{G,K}(\C)=G(\Q)\backslash G(\A)/Z_G(\R)^+K_{\infty}^+K,
\end{displaymath}
where $K_{\infty}^+=O(2)^n\cap G(\R)^+$ is the connected component of the maximal compact subgroup of $G(\R)$. We have $G(\R)/Z_G(\R)^+K_{\infty}^+ =(\hh\cup\hh^-)^n$ where $\hh\cup \hh^-=\C\setminus \R$ and $\hh$ is the upper half plane. The $n$-dimensional Shimura variety $Y_{G,K}(\C)$  carries a natural structure of complex quasi-projective variety. 

The determinant map $\det: G \rightarrow \operatorname{Res}_{F / \Q}\left(\mathbb{G}_{m}\right)$ induces a bijection between the set of geometric connected components of $Y_{G,K}(\C)$ and the finite double coset space
\begin{equation}\label{eq:connecomps}
Cl_{F}^{+}(K):=F_{+}^{\times} \backslash (\A_{F,f})^{\times} / \det(K)
\end{equation}
where $F_{+}^{\times}$ denotes the subgroup of $F^{\times}$ of totally positive elements. There is a natural surjective map $Cl_{F}^{+}(K) \rightarrow Cl_{F}^{+},$ where $Cl_{F}^{+}$ is the strict ideal class group of $F$. The preimage of each ideal class $[\mathrm{c}]$ is a torsor under the group $I:=\widehat{\mathcal{O}}_{F}^{\times} / \operatorname{det}(K) \mathcal{O}_{F,+}^{\times},$ where $\mathcal{O}_{F,+}^{\times}$ denotes the group of totally positive units in $\mathcal{O}_{F}$. By strong approximation we can write $G(\A_f)$ as a finite disjoint union over elements $c\in G(\A_f)$ such that $\det(c)$'s form a set of representatives of $Cl_{F}^{+}(K)$
\begin{displaymath}
G(\A_f)=\coprod_{c}G(\Q)^+cK
\end{displaymath}
and we therefore have
\begin{displaymath}
Y_{G,K}(\C)=\coprod_{c}\Gamma(c,K)\backslash \hh^n,
\end{displaymath}
where $\Gamma(c,K)=G(\Q)^+\cap cKc^{-1}$.

This Shimura variety is not of PEL type. However, as explained for example in \cite{tianxiao} (whose exposition we follow closely) it acquires a moduli space interpretation as follows. Firstly, from now on, we assume that $K=K^pK_p$, where $K^p\subset G(\A_f^p)$ is sufficiently small and $K_p=\GL_2(\Oo_F\otimes \Z_p)$. We rewrite the above disjoint union as 
\begin{displaymath}
Y_{G,K}(\C)=\coprod_{[\cc]\in Cl_{F}^{+}} \MM^{\cc}_K(\C), \ \ \ \text{where } \ \ \  \MM^{\cc}_K(\C)=\coprod_{c_i\in [\cc]_K}\Gamma(c_i,K)\backslash \hh^n,
\end{displaymath}
where for every ideal $\cc$ we write $[\cc]$ for its class in $Cl_{F}^{+}$, we select such representatives to be coprime to $p$ and we choose a subset $[\cc]_K=\{c_i, i\in I\}\subset G(\A_f)$ such that the fractional ideal associated to $\det(c_i)$ is $\cc$ and the set $\{\det(c_i)\}_{i\in I}$ is a set of representatives of the preimage of $[\cc]$ in $Cl_{F}^{+}(K)$.

Note that $\MM^{\cc}_K$ does not depend on the choice of $[\cc]_K$ and descends to an algebraic variety defined over $\Q$. Following \cite{tianxiao}, we will realise $\MM^{\cc}_K$ as quotient of some moduli space $\mathcal{M}^{\cc}_K$ by the action of the finite group
\begin{displaymath}
\Delta(K):=\mathcal{O}_{F,+}^{\times}/(K\cap \Oo_F^{\times})^2.
\end{displaymath} 
If $K^p$ is sufficiently small, we denote by $\mathcal{M}^{\cc}_K$ the smooth quasi-projective $\Z_{(p)}$-scheme (see 
\cite{rapo, chai}) representing the moduli problem which associates to a locally noetherian $\Z_{(p)}$-scheme $S$ the quadruple $(A,\iota,\lambda,\alpha_{K^p})$ given as follows
\begin{itemize}
\item $A$ is an $n$-dimensional abelian variety over $S$ with a homomorphism 
\begin{displaymath}
\iota: \Oo_F \to \operatorname{End}_S(A)
\end{displaymath}
such that Lie$(A)$ is a locally free $\Oo_S\otimes_{\Z}\Oo_F$-module of rank one;
\item $\lambda$ is a $\cc$-polarisation on $A$, i.e.
it is a $\Oo_F$-linear isomorphism
\begin{displaymath}
\lambda: A\otimes_{\Oo_F} \cc\xrightarrow{\simeq} A^{\vee},
\end{displaymath}
where $A^{\vee}$ denotes the dual abelian variety of $A$, which has a natural $\Oo_F$-action;
\item $\alpha_{K^p}$ is a $K^p$-level structure on $(A, \iota, \lambda)$, namely, assuming firstly that
\begin{displaymath}
K=\Gamma(N):=\{\gamma \in G(\hat{\Z}): \gamma \equiv 1 \mod N\},
\end{displaymath}
for an integer $N$ coprime to $p$, $\alpha_{K^p}$ is an $\Oo_F$-linear isomorphism of \'{e}tale group schemes over $S$
\begin{displaymath}
\alpha_{K^p}: (\Oo_F/N)^2 \xrightarrow{\simeq} A[N].
\end{displaymath}
The Weil pairing together with the polarisation $\lambda$ gives an $\Oo_F$-linear pairing $A[N]\times A[N]\to \mu_N\otimes_{\Z} \cc^{-1}\mathfrak{d}_F^{-1}$, where $\mathfrak{d}_F$ is the different ideal of $F/\Q$. Hence $\alpha_{K^p}$ determines an isomorphism $\Oo_F/N\to \mu_N \otimes \cc^{-1}\mathfrak{d}_F^{-1}$. One similarly defines a $K^p$-level structure by choosing $N$ such that $K(N)\subset K$, working on fibres $A_s$ over points $s$ of $S$ and using the action of $\GL_2(\Oo_F/N)$ on the $K(N)$-level structures on $A_s$ as above (see \cite[$\S$2.3]{tianxiao} for more details).
\end{itemize} 

We now recall that there is a natural action of $\Delta(K)$ on $\mathcal{M}^{\cc}_K$ given as follows. Firstly if $\epsilon\in \Oo_{F,+}^{\times}$, we can define
\begin{displaymath}
\epsilon\cdot (A,\iota, \lambda, \alpha_{K^p}) = (A,\iota, \iota(\epsilon)\circ\lambda, \alpha_{K^p}). 
\end{displaymath}
Moreover, if $\epsilon=\eta^2$ for some $\eta\in K\cap \Oo_F^{\times}$, then $\epsilon\cdot (A,\iota, \lambda, \alpha_{K^p})= (A,\iota, \lambda, \alpha_{K^p})$. This follows from the fact that any unit $\eta$ defines an isomorphism $\eta:A\simeq A$ such that $\eta^*\lambda =\eta^2\lambda$. Therefore we have defined an action of $\Delta(K)$ on $\mathcal{M}^{\cc}_K$ and the set of equivalent classes of geometric components under such action is in bijection with $\hat{\Oo}_F^{\times}/\det(K) \Oo_{F,+}^{\times}$ and the stabiliser of each component is $\det(K)\cap \Oo_{F,+}^{\times}/(K\cap \Oo_F^{\times})^2$. Following \cite{tianxiao}, we write $(A,\iota, \bar{\lambda}, \bar{\alpha_{K^p}})$ for the $\Oo_{F,+}^{\times}$-orbit of $(A,\iota, \lambda, \alpha_{K^p})$. The following is \cite[Proposition 2.4, Lemma 2.5]{tianxiao}.
\begin{prop}
The quotient of $\mathcal{M}^{\cc}_K(\C)$ by $\Delta(K)$ is isomorphic to $\MM^{\cc}_K(\C)$, which can be identified with the coarse moduli space over $\C$ of the orbits $(A,\iota, \bar{\lambda}, \bar{\alpha_{K^p}})$. Moreover, up to replacing $K^p$ by an open compact normal subgroup of finite index, we can assume $\det(K)\cap \Oo_{F,+}^{\times}=(K\cap \Oo_F^{\times})^2$ and therefore the quotient map
\begin{displaymath}
\mathcal{M}^{\cc}_K(\C) \to \MM^{\cc}_K(\C)
\end{displaymath}
induces an isomorphism between every connected component of $\mathcal{M}^{\cc}_K(\C)$ with its image. 
\end{prop}
We assume from now on that $K$ is sufficiently small and $\det(K)\cap \Oo_{F,+}^{\times}=(K\cap \Oo_F^{\times})^2$. Let
\begin{displaymath}
\mathcal{M}_K:=\coprod_{[\cc]\in Cl_{F}^{+}} \mathcal{M}^{\cc}_K, \ \ \ \MM^{\cc}_K:=\mathcal{M}^{\cc}_K/ \Delta(K) \ \ \ \text{ and } \ \ \ Y_{G,K}=\coprod_{[\cc]\in Cl_{F}^{+}} \MM^{\cc}_K.
\end{displaymath}
The proposition implies that every geometric connected component of $Y_{G,K}$ is identified with a geometric connected component of $\mathcal{M}^{\cc}_K$ for some $\cc$. Hence $Y_{G,K}$ is quasi-projective smooth over $\Z_{(p)}$ and it is the integral model of $Y_{G,K}(\C)$. It also has a universal family of abelian varieties over it, denoted by
\begin{displaymath}
\mathcal{A}\to Y_{G,K}
\end{displaymath} 
built using the universal abelian schemes $\mathcal{A}^{\cc}\to \mathcal{M}^{\cc}_K$.

We will also need the auxiliary variety of Iwahori level at a prime $\p$ above $p$. Assume $p$ is unramified in $F$ and let $K^p$ as above. Consider
\begin{displaymath}
K_0(\p)=\{g\in G(\Z_p): g\equiv \SmallMatrix{*&*\\0&*} \mod \p\}.
\end{displaymath}
Then $K^pK_0(\p)$ is again sufficiently small and we denote by $\mathcal{M}^{\cc}_{K}(\p)$ the smooth quasi-projective $\Z_{(p)}$-scheme (see 
\cite{rapo}) representing the moduli problem which associates to a locally noetherian $\Z_{(p)}$-scheme $S$ an isogeny $\phi:A_1\to A_2$ of degree Norm$(\p)$, where $A_1,A_2$ corresponds to quadruples $(A_i,\iota_i,\lambda_i,\alpha_{i,K^p})$ as above where 
\begin{itemize}
\item the kernel of $\phi$ is annihilated by $\p$;
\item $\lambda_1$ is a $\cc$-polarisation on $A_1$ and $\lambda_2$ is a $\cc\p$-polarisation of $A_2$ and for every $x\in \cc\p\subset \cc$ we have
\begin{displaymath}
\phi^*\circ \lambda_2(x)\circ \phi = \lambda_1(x);
\end{displaymath}
\item the $K^p$-level structures $\alpha_{i,K^p}$ are compatible, i.e. if $K^p$ is the congruence subgroup of level $N$
\begin{displaymath}
\alpha_{2,K^p}=\phi_{|A_1[N]}\circ \alpha_{1,K^p}.
\end{displaymath}
\end{itemize} 
The fibre of $\mathcal{M}^{\cc}_K(\p)$ over $p$ is smooth outside a closed subset of codimension 1. We can define an action of the units on this moduli space and we let as above 
\begin{displaymath}
\mathcal{M}_K(\p):=\coprod_{[\cc]\in Cl_{F}^{+}} \mathcal{M}^{\cc}_K(\p), \ \ \ \MM^{\cc}_K(\p):=\mathcal{M}^{\cc}_K(\p)/ \Delta(K') \ \ \ \text{ and } \ \ \ Y_{G,K}(\p)=\coprod_{[\cc]\in Cl_{F}^{+}} \MM^{\cc}_K(\p).
\end{displaymath}
As above, $Y_{G,K}(\p)$ is quasi-projective over $\Z_{(p)}$ and it is the integral model of $Y_{G,K'}(\C)$, where $K'=K^pK_0(\p)$. Moreover, there is a natural forgetful morphism $\mathcal{M}^{\cc}_{K}(\p)\to \mathcal{M}^{\cc}_{K}$ which is equivariant for the actions of $\Delta(K)$ and hence induces a morphism
\begin{equation}\label{eq:p1}
p_1: Y_{G,K}(\p)\to Y_{G,K}.
\end{equation}
Fix a fractional ideal $\cc$ and an isomorphism $\theta_{\cc}: \cc'\to \cc\p$ for some $[\cc']\in Cl_{F}^{+}$; such isomorphism is unique up to an element of $\Oo_{F,+}^{\times}$. Then one can also consider the forgetful morphism $p_{2,\theta_{\cc}}:\mathcal{M}^{\cc}_{K}(\p)\to \mathcal{M}^{\cc'}_{K}$, which now sends the isogenous pair to the second quadruple with polarisation $A_2\otimes \cc'\xrightarrow{\theta_{\cc}} A_2\otimes \cc\p\xrightarrow{\lambda_2}A^{\vee}$. This map is equivariant under the action of $\Delta(K)$ and $p_{2,\epsilon\cdot\theta_{\cc}}$ is equal to $p_{2,\theta_{\cc}}$ composed with the map induced by the action of $\epsilon$. Therefore we obtain a well defined morphism
\begin{equation}\label{eq:p2}
p_{2}: Y_{G,K}(\p)\to Y_{G,K}
\end{equation} 
independent on the choice of $\theta_{\cc}$.
\subsection{Compactifications} We recall a few facts about toroidal compactifications. Let $K=K^p K_p\subset G(\A_f)$ an open compact as above with $K_p=\prod_{\p\mid p} K_{\p}$ and $K_{\p} \in \{\GL_2(\Oo_{F_{\p}}), K_0(\p)\}$. 

Choosing an admissible rational polyhedral cone decomposition, one constructs a smooth toroidal compactifications of $\mathcal{M}^{\cc}_K$, see for example \cite[$\S$ 5]{dimtil}, \cite{rapo,chai} and more recently \cite{lan13,lan17}. In particular the case $K_{\p}=\GL_2(\Oo_{F_{\p}})$ for every $\p\mid p$ is covered in \cite{lan13} and the case with \emph{some level} at $p$ is covered in \cite{lan17}. More precisely, there exists a scheme $\mathcal{M}^{\cc, \tor}_K$ flat, local complete intersection and normal over $\Z_{(p)}$ containing $\mathcal{M}^{\cc}_K$ as a fiberwise dense open subscheme. This depends a priori on the choice of the cone decomposition, but we will see later that the cohomology groups we work with are independent on this choice. Moreover, there exists a semi-abelian scheme $\mathcal{A}^{\cc,\tor}\to \mathcal{M}^{\cc, \tor}_K$ extending the universal abelian scheme over $\mathcal{M}^{\cc,\tor}$, with $\cc$-polarisation, $\Oo_F$-action and level structure extending the data on $\mathcal{A}^{\cc}$. The boundary divisor $\mathcal{D}=\sqcup_{\cc}\mathcal{M}^{\cc, \tor}_K - \sqcup_{\cc}\mathcal{M}^{\cc}_K$ is a relative simple normal crossing divisor, endowed with a free action of $\Delta(K)$.  Let
\begin{displaymath}
\mathcal{M}_K^{\tor}:=\coprod_{[\cc]\in Cl_{F}^{+}} \mathcal{M}^{\cc,\tor}_K, \ \ \ \MM^{\cc,\tor}_K:=\mathcal{M}^{\cc,\tor}_K/ \Delta(K) \ \ \ \text{ and } \ \ \ X_{G,K}=\coprod_{[\cc]\in Cl_{F}^{+}} \MM^{\cc,\tor}_K
\end{displaymath}
and denote by $D$ the boundary divisor of $X_{G,K}$.

\subsection{Automorphic vector bundles} Let $L$ denote the Galois closure of $F(\sqrt{\epsilon}:\epsilon\in \Oo_{F,+}^{\times})$ and let $\Oo_L$ be its ring of integers. Fix a noetherian $\Oo_{L,(p)}$-algebra $R$. Let us rename for simplicity $\mathcal{M}=(\mathcal{M}_K^{\tor})_R$ and let $\mathcal{A}^{\tor}=\sqcup \mathcal{A}^{\cc, \tor}\to \mathcal{M}$ be the semi-abelian scheme extending the universal abelian variety with real multiplication by $\Oo_F$. Let $e: \mathcal{M}\to \mathcal{A}^{\tor}$ be the unit section and
\begin{displaymath}
\myom:=e^{\star}\Omega^1_{\mathcal{A^{\tor}}/\mathcal{M}};
\end{displaymath}  
this is $(\Oo_\mathcal{M}\otimes_{\Z}\Oo_F)$-module locally free of rank 1. Its restriction to $\mathcal{M}_K$ coincides with the sheaf defined analogously using the unit section of the abelian scheme $\sqcup \mathcal{A}^{\cc}$.  We can write
\begin{displaymath}
\myom=\oplus_{\tau \in \Sigma_{\infty}} \omega_{\tau}
\end{displaymath}
where $\omega_{\tau}$ is the direct summand on which the $\Oo_F$-action is given by the composition of the embedding $\tau$ with the structure morphism $\Oo_{L,(p)}\to R$. Let $\hh^1$ be the canonical extension of is the relative de Rham cohomology $\hh^1_{\dr}(\mathcal{A}/\mathcal{M})$ of $\mathcal{A}$. It is a $(\Oo_{\mathcal{M}}\otimes_{\Z}\Oo_F)$-module locally free of rank 2. We have the Hodge filtration
\begin{equation}\label{hodgeexseq}
\tag{H}
0\to \myom \to \hh^1 \to (\myom_{\mathcal{A}^{\vee}} )^{\vee}\otimes \mathfrak{d}^{-1}\to 0,
\end{equation}
where $\myom_{\mathcal{A}^{\vee}}:=(e')^{\star}\Omega^1_{\mathcal{A}^{\vee}/\mathcal{M}}$, where $e'$ is the unit section of $\mathcal{A}^{\vee}$.

For $(\underline{k}, w) \in \Z^{\#\Sigma_{\infty}}\times \Z$ such that $k_{\tau}\equiv w \mod 2$ and $k_{\tau}\leq w$ for every $\tau\in \Sigma_{\infty}$ let
\begin{displaymath}
\myom^{(\underline{k},w)}:= \bigotimes_{\tau \in \Sigma_{\infty}}\left(\left(\wedge^{2} \mathcal{H}_{\tau}^{1}\right)^{\frac{w-k_{\tau}}{2}} \otimes {\omega}_{\tau}^{k_{\tau}}\right)
\end{displaymath}

\begin{rmk}
One has (see for example \cite[1.0.13-1.0.15]{katzpadic}) that $\bigotimes_{\tau \in \Sigma_{\infty}}\left(\wedge^{2} \mathcal{H}_{\tau}^{1}\right)$ admits a trivialisation on each of the components $\mathcal{M}^{\cc}_K $. Such trivialisation however depends on the $\cc$-polarisation and is not canonical on $X_{G,K}$. 
\end{rmk}

Global sections of this sheaf can be interpreted as Hilbert modular forms \`{a} la Katz. See for example \cite[1.2]{katzpadic} (where however the definition corresponds to sections of the sheaf $\bigotimes_{\tau \in \Sigma_{\infty}}{\omega}_{\tau}^{k_{\tau}}$, in view of the above remark).

\begin{defi}\label{defikatz}
A $\cc$-Hilbert modular form of weight $(\underline{k},w)$ of level $K$ defined over an $\Oo_{L,(p)}$-algebra $R$ is a rule $f$ which assigns to every quadruple $(A,\iota, \lambda, {\alpha_{K}})$ as above defined over $R$, where $\lambda$ is a $\cc$-polarisation, given with a pair $(\omega, \eta)$, where $\omega$ is an $\Oo_F\otimes R$ basis of $\Omega^1_{A/R}$ and $\eta$ is an $\Oo_F\otimes R$-basis of $\wedge^2H^1_{\dr}(A)$ an element $f(A,\iota, \lambda,{\alpha_{K}}, \omega, \eta)\in R$, satisfying the following conditions:
\begin{itemize}
\item[(i)] $f(A,\iota, \lambda, {\alpha_{K}}, \omega, \eta)$ depends only on the $R$-isomorphism class of $(A,\iota, \lambda, {\alpha_{K}}, \omega, \eta)$;
\item[(ii)] $f$ commutes with extension of scalars $R_1\to R_2$ of $R$-algebras;
\item[(iii)] for any $\underline{a},\underline{b}\in (R^{\times})^{\Sigma_{\infty}}\simeq(\Oo_F\otimes R)^{\times}$, we have
\[
f(A,\iota, \lambda, {\alpha_{K}}, \underline{a}\cdot \omega, \underline{b}\cdot \eta)=\prod_{\tau \in \Sigma_{\infty}}a_{\tau}^{-k_{\tau}}b_{\tau}^{-\tfrac{w-k_{\tau}}{2}}f(A,\iota, \lambda, {\alpha_{K}}, \omega, \eta).
\]
\end{itemize}
\end{defi}

Clearly, in order to get sections of the sheaf $\myom^{(\underline{k},w)}$, one needs to admit different polarisation types.

One could more generally define, for $(\underline{k},\underline{n})\in\Z^{\#\Sigma_{\infty}}\times\Z^{\#\Sigma_{\infty}}$, a sheaf
\begin{displaymath}
\myom^{(\underline{k},\underline{n})}:= \bigotimes_{\tau \in \Sigma_{\infty}}\left(\left(\wedge^{2} \mathcal{H}_{\tau}^{1}\right)^{n_{\tau}} \otimes {\omega}_{\tau}^{k_{\tau}}\right).
\end{displaymath}
For $k_{\tau}\equiv w \mod 2$ for every $\tau$, we recover $\myom^{(\underline{k},\tfrac{w-\underline{k}}{2})}=\myom^{(\underline{k},w)}$.

Let us denote by $X$ the compactified Shimura variety $X_{G,K}$. In order to define a sheaf over $X$, we need to give a descent datum for the map $\mathcal{M}\to X$. We will see that this will force $k_\tau + 2n_\tau=w$ for some $w\in \Z$.

\begin{defi}\label{defiunits} The action of $\epsilon\in \Oo_{F,+}^{\times}$ is given on stalks by the isomorphism (see for example \cite[$\S$4]{dimtil})
\begin{displaymath}
\omega_{(A,\iota,\epsilon^{-1}\lambda,\alpha_{K^p})}= \omega_{(A,\iota,\lambda,\alpha_{K^p})} \to  \omega_{(A,\iota,\lambda,\alpha_{K^p})},
\end{displaymath}
where the first equality is given by the fact that the sheaf $\omega$ does not depend on the polarisation and the second map is the multiplication by $\prod_{\tau} \tau(\epsilon)^{-1/2}$. Similarly the action on $\wedge^2\hh^1$ is given by multiplication by $\prod_{\tau} \tau(\epsilon)^{-1}$.
\end{defi}

If $\epsilon^2 \in (K\cap \Oo_F^{\times})^2$, one easily verifies that the action defined above is trivial. More precisely, we have
\begin{displaymath}
\epsilon^*(A,\iota,\lambda,\alpha_{K^p},\omega)=(A,\iota,\epsilon^2 \lambda, \epsilon\omega, \epsilon^{2} \eta)
\end{displaymath}
and any section $f$ of  ${\omega}^{(\underline{k},\underline{n})}$ satisfies
\begin{align*}
f(\epsilon^2\cdot (A,\iota,\lambda,\alpha_{K^p},\omega,\eta))&=\prod_{\tau}\tau(\epsilon)^{-(k_\tau +2n_\tau)}f(A,\iota,\epsilon^{-2}\lambda,\alpha_{K^p},\omega,\eta)\\&=\prod_{\tau}\tau(\epsilon)^{-(k_\tau +2n_\tau)}f((\epsilon^{-1})^*(A,\iota,\lambda,\alpha_{K^p},\epsilon^{-1}\omega,\epsilon^{-2}\eta))=f(A,\iota,\lambda,\alpha_{K^p},\omega,\eta).
\end{align*}

By abuse of notation we will still denote by ${\omega}^{(\underline{k},\underline{n})}$ the descent of the sheaf ${\omega}^{(\underline{k},\underline{n})}$ over $X$. Note that section of this sheaf are rules as in Definition \ref{defikatz} satisfying the additional condition
\begin{displaymath}
f(A,\iota,\epsilon\cdot\lambda,\alpha_{K^p},\omega,\eta)=f(A,\iota,\lambda,\alpha_{K^p},\omega,\eta) \ \ \ \ \forall \epsilon\in \Oo_{F,+}^{\times}.
\end{displaymath}
This implies, if $R$ has characteristic zero, that $\myom^{(\underline{k},\underline{n})}_R$ has non-zero global sections over $X$ if and only if $k_\tau + 2n_\tau=w$ for some $w\in \Z$.

Finally, we recall that the cohomology of this sheaf does not depend on the cone decomposition chosen to define the toroidal compactification $X$.
\begin{lemma}[\cite{lan13}]
The cohomologies $\RG (X,\myom^{(\underline{k},w)})$, $\RG (\mathcal{M},\myom^{(\underline{k},w)})$, $\RG (X,\myom^{(\underline{k},w)}(-D))$ and \\$\RG (\mathcal{M},\myom^{(\underline{k},w)}(-\mathcal{D}))$ are independent on the cone decompositions chosen to define $\mathcal{M}$ and $X$.
\end{lemma}
Moreover, we define the Hodge line bundle 
\[
\det(\myom):=\wedge^n_{\Oo_{X}}\myom.
\] 
One can construct minimal compactifications of $Y_{G,K}$, following \cite{chai} or \cite[$\S$ 7.2]{lan13}, as follows
\[
X_{G,K}^*=\operatorname{Proj}\left( \oplus_{m\geq 0} \Gamma(X_{G,K},\det(\myom)^{\otimes m})\right).
\]
It is a normal projective scheme over $\Z_{(p)}$ and $\det(\myom)$ descends to an ample line bundle on $X^*=X_{G,K}^*$. The inclusion $Y_{G,K}\hookrightarrow X$ induces an inclusion $Y_{G,K}\hookrightarrow X^*$ and $X^*$ is canonically determined by $Y_{G,K}$. The boundary $X^*-Y_{G,K}$ is finite flat over $\Z_{(p)}$. The following is \cite[Theorem 8.2.1.3]{lanbis}.

\begin{lemma}[\cite{lanbis}]\label{lancuspminimal}
Let $\pi:X\to X^*$ be the canonical projection. Then we have $$\operatorname{R}^i\pi_*(\myom^{(\underline{k},w)}(-D))=0 \ \text{ for every } i>0.$$
\end{lemma}
We finally recall (see for example \cite[2.11.2-2.11.3]{tianxiao}) that the Kodaira--Spencer isomorphism gives an isomorphism
\begin{equation}\label{eq:KS}
\tag{KS}
KS: \omega^{(\underline{2},0)}(-D)\to \Omega^n_{{X}/\Z_p}, \ \ \ \ \text{where }\omega^{(\underline{2},0)}=\bigotimes_{\tau \in \Sigma_{\infty}}\left(\left(\wedge^{2} \mathcal{H}_{\tau}^{1}\right)^{-1} \otimes {\omega}_{\tau}^{2}\right).
\end{equation}
\subsubsection{A more general definition}\label{sectiontorsors}
The automorphic vector bundle ${\omega}^{(\underline{k},w)}$ can also be defined using the theory of torsors. Compare for example with \cite[D\'{e}finition 4.4]{dimtil}, where however the twist by $\wedge^{2} \mathcal{H}^1$ is not consider. Let
\begin{displaymath}
\mathcal{T}=\underline{\operatorname{Isom}}_{\Oo_{\mathcal{M}}\otimes \Oo_F}(\Oo_{\mathcal{M}}\otimes \Oo_F, \myom).
\end{displaymath}
It is a $T=\res_{\Oo_F/\Z}\mathbb{G}_m$-torsor over $\mathcal{M}$, representing the functor sending a $\Z_{(p)}$-algebra $R$ to the set of isomorphisms of tuples $(A,\iota, \lambda, {\alpha_{K}}, \omega)$, where $(A,\iota, \lambda, {\alpha_{K}})$ is an abelian variety over $R$ with extra structure as above and $\omega: R\otimes \Oo_F\simeq e^*\Omega^1_{A/R}$ is a trivialisation of the conormal sheaf of $A/R$ with respect to the unit section $e$. We can decompose $\omega$ with respect to the $\Oo_F$-action and write $\omega=(\omega_\tau)_{\tau}$, where $\omega_{\tau}:R\simeq (e^*\Omega^1_{A/R})_{\tau}$. The action of $T$ is then given as follows: if $t=(t_{\tau})\in T(R)$, then
\begin{displaymath}
t\cdot (A,\iota, \lambda, {\alpha_{K}}, \omega)=(A,\iota, \lambda, {\alpha_{K}}, t\cdot \omega),
\end{displaymath}
where $t\cdot \omega=(t_\tau \cdot \omega_\tau)_\tau$. One can similarly define the $T$-torsor
\begin{displaymath}
\mathcal{L}=\underline{\operatorname{Isom}}_{\Oo_{\mathcal{M}}\otimes \Oo_F}(\Oo_{\mathcal{M}}\otimes \Oo_F,\wedge^2\mathcal{H}^1_{\dr}(\mathcal{A})).
\end{displaymath}
It represents the functor sending a $\Z_{(p)}$-algebra $R$ to the set of isomorphisms of tuples $(A,\iota, \lambda, {\alpha_{K}}, \eta)$, where $\eta: R\otimes \Oo_F\simeq \wedge^2H^1_{\dr}(A/R)$. The action of $T$ is given similarly as above. 

Let us write $\pi_{\mathcal{T}}:\mathcal{T} \to \mathcal{M}$, $\pi_{\mathcal{L}}:\mathcal{L} \to \mathcal{M}$ for the natural maps (corresponding to the forgetful functors). Let $(\underline{\kappa},\underline{n})\in \Z[\Sigma_{\infty}]\times \Z[\Sigma_{\infty}]$. We can consider the sheaves $(\pi_{\mathcal{T}})_*(\Oo_{\mathcal{T}}),(\pi_{\mathcal{L}})_*(\Oo_{\mathcal{L}})$. They both have an action of the torus $T$ and we can consider the component on which $T$ acts via the character $t\to t^{-\underline{\kappa}}$, respectively via the character $t\to t^{-\underline{n}}$ and denote the corresponding invertible sheaves over $\mathcal{M}$ by $(\pi_{\mathcal{T}})_*(\Oo_{\mathcal{T}})[-\underline{\kappa}]$ and $(\pi_{\mathcal{L}})_*(\Oo_{\mathcal{L}})[-\underline{n}]$ respectively. We define
\begin{displaymath}
\tilde{\myom}^{(\underline{\kappa},\underline{n})}:=(\pi_{\mathcal{T}})_*(\Oo_{\mathcal{T}})[-\underline{\kappa}]\otimes (\pi_{\mathcal{L}})_*(\Oo_{\mathcal{L}})[-\underline{n}].
\end{displaymath}
Sections of this sheaf are rules as in Definition \ref{defikatz}, where clearly (iii) is replaced by the analogous condition, with the discrepancy factor being $\prod a_{\tau}^{-\kappa_{\tau}}b_{\tau}^{-n_{\tau}}$. We obtain $\myom^{(\underline{k},\underline{n})}\simeq \tilde{\myom}^{(\underline{k},\underline{n})}$ and, in particular,
\begin{equation}\label{eq:isotorsoromega}
\myom^{(\underline{k},w)}\simeq \tilde{\myom}^{(\underline{k},\tfrac{w-\underline{k}}{2})},
\end{equation}
for $\underline{n}=\tfrac{w-\underline{k}}{2}$ in the case where $k_\tau\equiv w \mod 2$ for every $\tau$.

Moreover, we can observe that, using the exact sequence (\ref{hodgeexseq}), there is a natural map
\begin{displaymath}
s: \mathcal{T}\times_{\mathcal{M}} \mathcal{T'}:=\underline{\operatorname{Isom}}_{\Oo_{\mathcal{M}}\otimes \Oo_F}(\Oo_{\mathcal{M}}\otimes \Oo_F, \myom)\times \underline{\operatorname{Isom}}_{\Oo_{\mathcal{M}}\otimes \Oo_F}(\Oo_{\mathcal{M}}\otimes \Oo_F, \myom_{\mathcal{A}^{\vee}}^{\vee} )\to \underline{\operatorname{Isom}}_{\Oo_{\mathcal{M}}\otimes \Oo_F}(\Oo_{\mathcal{M}}\otimes \Oo_F,\wedge^2\mathcal{H}^1_{\dr}(\mathcal{A})).
\end{displaymath}
This map is defined over $R$ after fixing a generator of the principal ideal $\mathfrak{d}^{-1}$, which is coprime to $p$ under our assumptions.

The sheaf
$(\pi_{\mathcal{T}\times\mathcal{T'}})_*(s^*\Oo_{\mathcal{L}})$
has an action of the torus $T$. It is the sheaf $(\pi_{\mathcal{T}\times\mathcal{T'}})_*\Oo_{\mathcal{T}\times\mathcal{T'}}$ which would naturally have an action of $T\times T \ni (t,\epsilon)=\SmallMatrix{t\epsilon &0\\0&t^{-1}}$, but the pullback via $s$ makes the action of the first component trivial. Let
\begin{displaymath}
\hat{\myom}^{(\underline{\kappa},\underline{n})}:=(\pi_{\mathcal{T}})_*(\Oo_{\mathcal{T}})[-\underline{\kappa}]\otimes (\pi_{\mathcal{T}\times\mathcal{T'}})_*(s^*\Oo_{\mathcal{L}})[-\underline{n}].
\end{displaymath}

\begin{lemma}\label{lemmatildehat}
There is an isomorphism $\tilde{\myom}^{(\underline{\kappa},\underline{n})} \simeq \hat{\myom}^{(\underline{\kappa},\underline{n})}$.
\end{lemma}
\begin{proof}
We prove that there is an isomorphism of $\Oo_{\mathcal{M}}$-modules $(\pi_{\mathcal{T}\times\mathcal{T'}})_*(s^*\Oo_{\mathcal{L}})[-\underline{n}]\simeq (\pi_{\mathcal{L}})_*(\Oo_{\mathcal{L}})[-\underline{n}]$. 
To construct such an isomorphism we essentially use the fact that $\wedge^2\mathcal{H}^1_{\dr}(\mathcal{A})$ is an $\Oo_F\otimes \Oo_{\mathcal{M}}$-module of rank one. Local sections of the second sheaf are rules associating to $(x,\eta)$, where $x =(A,\iota, \lambda, {\alpha_{K}})\in \mathcal{M}(R)$ and $\eta =(\eta_{\tau})_\tau$, where $\eta_\tau : R \simeq \wedge^2 H^1_{\dr}(A_x/R)_\tau$, such that 
 such that
\begin{equation}\label{eq:rule1}
f(x,t\eta)=t^{-\underline{n}}f(x,\eta), \ \ \ \ \text{for every } t\in (R\otimes \Oo_F)^{\times}.
\end{equation}
Local sections of the first sheaf are rules associating to $(x,\alpha\otimes \beta)$, where $x =(A,\iota, \lambda, {\alpha_{K}})\in X(R)$ and $\alpha=(\alpha_{\tau})_\tau, \beta =(\beta_{\tau})_\tau$ and $\alpha_\tau : R \simeq \Omega^1(A_x/R)_{\tau}, \beta_\tau : R \simeq ((\Omega^1_{A_x^{\vee}/R})_{\tau})^{\vee}, \alpha_\tau \otimes \beta_{\tau} : R \simeq \wedge^2 H^1_{\dr}(A_x/R)_\tau$, satisfying 
\begin{equation}\label{eq:rule2}
f(x,t(\alpha\otimes \beta))=t^{-\underline{n}}f(x,\alpha\otimes \beta), \ \ \ \ \text{for every } t\in (R\otimes \Oo_F)^{\times}.
\end{equation}
Given $f$ as above, we define a local section $\tilde{f}$ of $(\pi_{\mathcal{L}})_*(\Oo_{\mathcal{L}})[-\underline{n}]$ as follows. Since $\wedge^2 H^1_{\dr}(A_x/R)_\tau$ is of rank one over $R$, given $\eta_\tau : R \simeq \wedge^2 H^1_{\dr}(A_x/R)_\tau$, we choose arbitrarily $\alpha, \beta$ and we find that there must exists $\lambda_\tau\in R^{\times}$ such that the following diagrams commute
\[
\begin{tikzcd}[column sep=2em, row sep=1em]
R \arrow[equal]{d} \arrow{r}{\eta_\tau} & \wedge^2 H^1_{\dr}(A_x/R)_\tau\arrow{d}[right]{\cdot \lambda_\tau} \\
R \arrow{r}{\alpha_\tau\otimes \beta_\tau} & \wedge^2 H^1_{\dr}(A_x/R)_\tau.
\end{tikzcd}
\] 
We then let $\tilde{f}(x,\eta):=\lambda^{-\underline{n}}f(x, \alpha\otimes \beta)$. We need to check this definition does not depend on the choice of $\alpha, \beta$. Since $\Omega^1(A_x/R)_{\tau}$ has rank one over $R$, any other trivialisation $\alpha'_{\tau}$ is of the form $\alpha'_{\tau}=\lambda_{1,\tau}\alpha_\tau$ for some $\lambda_{1,\tau}\in R^{\times}$ and similarly any other $\beta'$ is of the form $\lambda_{2,\tau}\beta_\tau$ for some $\lambda_{2,\tau}\in R^{\times}$. Therefore $\alpha'\otimes \beta'=\lambda_1\lambda_2(\alpha\otimes \beta)$ and $\eta=\lambda\lambda_1^{-1}\lambda_2^{-1}(\alpha'\otimes \beta')$. Thanks to (\ref{eq:rule2}) we have
\begin{displaymath}
\lambda^{-\underline{n}}f(x, \alpha\otimes \beta)= (\lambda\lambda_1^{-1}\lambda_2^{-1})^{-\underline{n}}f(x, \alpha'\otimes \beta')
\end{displaymath}
and hence $\tilde{f}$ is well defined and it satisfies (\ref{eq:rule1}) since $f$ satisfies (\ref{eq:rule2}). The natural \emph{restriction} map from $(\pi_{\mathcal{T}\times\mathcal{T'}})_*(s^*\Oo_{\mathcal{L}})[-\underline{n}]$ to $ (\pi_{\mathcal{L}})_*(\Oo_{\mathcal{L}})[-\underline{n}]$ is the inverse on the map we have just defined.
\end{proof}

Moreover, similarly as above, we can define a natural action of the units on $\mathcal{T}, \mathcal{T}', \mathcal{L}$ and descend these sheaves to the Shimura variety $X$ compatibly with the previous definitions. 
 
\subsubsection{Comparison with other works}\label{seccomparison}
We clarify the choices we made with respect to other works on the subject to help the reader who may want to compare this definition with the ones of \cite{AIP,tianxiao,redu,kisin,dimtil,katzpadic}. The definition of $\myom^{(\underline{k},w)}$ corresponds to the one of $\dot{{\omega}}^{\kappa}$, ${\omega}^{\kappa}$ for $\kappa=(\underline{k},w)$ in \cite[$\S$ 2.2]{redu} for the sheaf over $\mathcal{M}$ and $X$ respectively and to $\underline{\omega}^{(\underline{k},w+2)}$ in \cite[$\S$ 2.12]{tianxiao}. Similarly as in \emph{op. cit.}, our definition of the action of the units $(\Oo_F^{\times})^+$ is both on the polarisation and on the sheaves of differentials. 

In \cite{kisin,dimtil,katzpadic} the sheaf considered is $\otimes_{\tau}\omega_{\tau}^{k_\tau}$. In fact, in \cite[$\S$ 1.11]{kisin} and \cite[$\S$ 1.2]{katzpadic}, the authors only work with the moduli space $\mathcal{M}$ and do not consider the Shimura variety for $G$. This however brings some complications when defining Hecke operators for the ideals $\mathfrak{p}\mid p$. For example in \cite[(1.11.6)]{kisin} the Hecke operator is defined by carefully considering a trivialisation of the sheaf $\bigotimes_{\tau \in \Sigma_{\infty}}\left(\wedge^{2} \mathcal{H}_{\tau}^{1}\right)$.

In \cite{AIP}, the authors also consider the sheaf $\omega^{\underline{k}}=\otimes_{\tau}\omega_{\tau}^{k_\tau}$, but they work both with the moduli space and the Shimura variety for $G$, however their \emph{descent} is different from the one considered here. The action of the units on sections $f$ in $H^0(X,\omega^{\underline{k}})$ is given by
\begin{displaymath}
\epsilon\cdot f(A, \iota,\lambda,\alpha_{K^p}, \omega):= \prod_\tau \epsilon_\tau^{-(w-k_{\tau})/2}f(A,\iota,\epsilon^{-1}\lambda,\alpha_{K^p},\omega).
\end{displaymath}
If we fix a polarisation class $\cc$, this definition can be thought as follows: the sheaf $\bigotimes_{\tau \in \Sigma_{\infty}}\left(\wedge^{2} \mathcal{H}_{\tau}^{1}\right)^{(w-k_{\tau})/2}$ is trivial but it carries a non-trivial action of the units. The action of the units of \cite{AIP} is therefore given both on the polarisation and on the $\wedge^{2} \mathcal{H}^1$ factor, but not on the sheaf $\omega^{\underline{k}}$ itself. This choice results, when defining Hecke operators, in a normalisation differing from the classical one (the one we define in \ref{secHecke}) by a power  $-(w-k_{\tau})/2$ factor, as explained in \cite[Remark 4.7]{AIP}.
 
In some sense, here and in \cite{tianxiao,redu}, $w$ is encoded in the definition of the sheaf, whereas in the other mentioned works it comes in only when defining the action of the units. 
\subsubsection{BGG decomposition (and higher coherent cohomology)}
For a weight $(\underline{k}, w)$ as above let
\begin{displaymath}
\mathcal{F}_\tau^{(\underline{k}, w)}:=\operatorname{Sym}^{k_\tau -2}\hh^1_\tau \otimes (\wedge^2\hh^1_{\tau})^{\tfrac{w-k_{\tau}}{2}}, \ \ \ \mathcal{F}^{(\underline{k}, w)}:=\otimes_{\tau}\mathcal{F}_\tau^{(\underline{k}, w)}.
\end{displaymath}
The extended Gauss-Manin connection on $\mathcal{H}^1$ induces an integrable connection
\begin{displaymath}
\nabla:  \mathcal{F}^{(\underline{k}, w)}\to  \mathcal{F}^{(\underline{k}, w)} \otimes \Omega^1_{X_{G,K}}(\log D).
\end{displaymath}
One can show (see \cite[$\S$ 2.12]{tianxiao}) that $( \mathcal{F}^{(\underline{k}, w)}, \nabla)$ gives an integral model of the automorphic bundle on $X_{G,K}(\C)$ given by the representation of $G_{\C}$
\begin{displaymath}
\rho^{(\underline{k},w)}:=\bigotimes_\tau \left( \operatorname{Sym}^{k_{\tau}-2}(\check{\operatorname{St}}_\tau)\otimes \operatorname{det}^{-\tfrac{w-k_{\tau}}{2}}_{\tau}\right),
\end{displaymath}
where $\check{\operatorname{St}}_\tau$ is the $\tau$-projection of the dual of the standard representation of $G_{\C}=(\GL_{2,\C})^{\Sigma_{\infty}}$ and $\det_{\tau}$ is the $\tau$-projection of the determinant map. 

In the next chapters, we will study the cohomology of the sheaves $\myom^{(\underline{k}, w)}$ in degree zero (hence classical \emph{holomorphic} Hilbert modular forms) and in positive degrees. One reason for which the cohomology of higher degree is also interesting is the fact that it contributes to the middle degree de Rham cohomology of the Hilbert modular variety. More precisely, the de Rham complex of $\mathcal{F}^{(\underline{k}, w)}$ is quasi-isomorphic to a simpler complex, called the dual BGG-complex (see \cite[$\S$3,7]{faltings82}, \cite[$\S$ 5]{faltchai} and for an overview with examples \cite[$\S$ 2.3]{lanauto}). In this context, the Weyl group $W_G$ is isomorphic to $\{\pm\}^{\Sigma_{\infty}}$. For $J\subset \Sigma_{\infty}$ we denote by $s_J$ the Weyl element whose $\tau$-component is $-1$ if $\tau\not\in J$ and is equal to $1$ is $\tau \in J$. If we work over $\C$, we find the following decomposition
\begin{displaymath}
H^n_{\dr}(Y_{G,K}(\C),\mathcal{F}^{(\underline{k}, w+2)}_{\C})\simeq \bigoplus_{J\subset \Sigma_{\infty}}H^{\# J}(X_{G,K}(\C), \myom^{(s_J\cdot \underline{k}, w)}_{\C}),
\end{displaymath}
where the action of the Weyl group on $\underline{k}$ is given by $(s_J\cdot \underline{k})_\tau = 2-k_{\tau}$ if $\tau\not\in J$ and $(s_J\cdot \underline{k})_\tau = k_{\tau}$ if $\tau\in J$.
A more detailed discussion can be found in \cite[$\S$ 2.15]{tianxiao}. 
 
\section{Hecke operators and Hasse invariants}\label{secheckeandhasse}
From now on we assume that $p\geq 5$ splits completely in $F$ and we write 
\begin{displaymath}
(p)=\p_1 \cdots \p_n.
\end{displaymath}
Recall that $L$ is a Galois extension of $\Q$ containing the totally real field $F$. We fix once for all an embedding $\iota: L\hookrightarrow \bar{\Q}_p$. This fixes a prime $\wp$ of $L$ above $p$ and we consider the ring of integers $R$ of the completion of $L$ at $\wp$ and the residue field $\Ff$. The set $\Sigma_{\infty}$ of embeddings of $F$ in $\R$ is then identified with the set of $p$-adic embeddings $F\hookrightarrow \bar{\Q}_p$ and therefore with 
\begin{equation}\label{eq:sigma}
\Sigma_{\infty}=\Hom_{\Z}(\Oo_F, \Ff)=\{\p \subset \Oo_F: \p\mid p\}.
\end{equation}
From now on we will denote by $\tau_{\p}$ the element in $\Sigma_{\infty}$ such that $\iota\circ \tau_{\p}$ induces the place $\p$. Notice that from the above identification there is a natural action of the Frobenius automorphism $\sigma$ of $\Ff$ on $\Sigma_{\infty}$ (given by the composition $\tau\circ \sigma$); in this particular setting, where $p$ splits completely in $F$, this action is trivial.

We also fix the choice of a sufficiently small neat open compact subgroup $K=K^pK_p$ of $G(\A_f)$ such that $K_p=\prod_{\p\mid p}\GL_2(\Z_p)$. We let $\mathcal{M}=\mathcal{M}^{\tor}_K\times_{\Z_{(p)}}R$ and $X=X_{G,K}\times_{\Z_{(p)}}R$. We still denote by $D$ the boundary divisor in $X$.

\subsection{Hecke operators}\label{secHecke}  
Consider $\p$ one of the primes above $p$. Let $\mathcal{M}_0(\p)=\mathcal{M}(\p)^{\tor}_K\times_{\Z_{(p)}}R$ and $X_0(\p)=X_{G,K}(\p)\times_{\Z_{(p)}}R$. We want to study the cohomolgical correspondence obtained by the maps in (\ref{eq:p1}) and (\ref{eq:p2})
\[
\begin{tikzcd}[column sep=1em]
 & X_0(\p) \arrow{dr}[right]{p_2} \arrow{dl}[left]{p_1}\\
X && X.
\end{tikzcd}
\]
For general background and notation on correspondences and coherent cohomology, we refer for example to \cite[$\S$ 4]{pilloni}. We denote by $\mathcal{A}$ the universal semi-abelian variety $\mathcal{A}^{\tor}\to \mathcal{M}$. The maps $p_1,p_2$ parametrise an isogeny $p_1^{\star}\mathcal{A}\to p_2^{\star}\mathcal{A}$ of degree $p$ and with kernel annihilated by $\p$. From this isogeny we get a rational map $p_2^{\star}{\omega}^{(\underline{k},w)} \dashrightarrow p_1^{\star}{\omega}^{(\underline{k},w)}$ of sheaves over $\mathcal{M}_0(\p)$. Since this map is equivariant under the action of $\Delta(K)$, we get an analogous map of sheaves over $X_0(\p)$.
We now consider the map, called ''fundamental class'', constructed more generally in \cite[$\S$4.2]{pilloni} (see also \cite[$\S$3.8.11-3.8.16]{modularsurf}). It is a map 
$$\Theta:  p_1^\star\Oo_{X}\to p_1^!\Oo_X.$$
Tensoring $p_2^{\star}{\omega}^{(\underline{k},w)} \dashrightarrow p_1^{\star}{\omega}^{(\underline{k},w)}$ with $\Theta$, we get the \emph{naive} cohomological correspondence
\begin{displaymath}
\tilde{T}_{\p, (\underline{k},w)}: p_2^{\star}{\myom}^{(\underline{k},w)} \dashrightarrow p_1^{!}{\myom}^{(\underline{k},w)}.
\end{displaymath}
We finally normalise it letting
\begin{displaymath}
T_{\p, (\underline{k},w)}:=p^{-\inf\lbrace\tfrac{w-k_{\p}}{2}+1,\tfrac{w+k_{\p}}{2}\rbrace}\tilde{T}_{\p, (\underline{k},w)}=\begin{cases}
p^{-\tfrac{w-k_{\p}}{2}-1}\tilde{T}_{\p, (\underline{k},w)} &\text{if }k_{\p}\geq 1 \\
p^{-\tfrac{w+k_{\p}}{2}}\tilde{T}_{\p, (\underline{k},w)} &\text{if }k_{\p}< 1,
\end{cases}
\end{displaymath}
where, if $p=\p_1\cdots \p_n$ and $\p=\p_i$, we write $k_{\p}=k_i$. To simplify the notation we will often denote simply by $T_\p$ the operator $T_{\p, (\underline{k},w)}$ for the automorphic sheaf ${\myom}^{(\underline{k},w)}$.
\begin{prop}\label{propTpnorm}
$T_{\p}$ is a cohomological correspondence $p_2^{\star}{\omega}^{(\underline{k},w)} \to p_1^{!}{\omega}^{(\underline{k},w)}$ (i.e. well defined and optimally integral). Moreover it is supported on the \'{e}tale locus (respectively multiplicative locus) if $k_{\p}> 1$ (resp. $k_{\p}< 1$).
\begin{proof}
We only need to check that $T_\p$ is well defined on the complement of a codimension 2 locus. Since it is well defined over $\Q_p$, we only need to verify the statement locally at generic points of the special fibre of the open variety $Y_0(\p)$. Such locus $V$ is smooth and the fundamental class is uniquely determined by its definition on $V$ and there it is defined to be the determinant of the map d$p_1 : \Omega^1_U \otimes \Oo_V \to \Omega^1_{V}$, where $U$ is a smooth open subset of $X$ such that $p_1(V)\subset U$. Therefore, by \cite[Lemma 4.2.3.1]{pilloni}, the fundamental class is given on $V$ by the trace map $\tr_{p_1}: \Oo_{V}\to p_1^!\Oo_U$.

Since the isogeny $p_1^{\star}\mathcal{A}\to p_2^{\star}\mathcal{A}$ is a $\p$-isogeny, for every generic point of the special fibre the map $p_2^{\star}(\wedge^2\hh^1_{\tau})\to p_1^{\star}(\wedge^2\hh^1_{\tau})$ is an isomorphism for $\tau$ not corresponding to $\p$ and factors through an isomorphism $p_2^{\star}(\wedge^2\hh^1_{\p})\to p(p_1^{\star}(\wedge^2\hh^1_{\p}))$ otherwise. 

We say that the generic point $\xi$ is multiplicative (respectively \'{e}tale) if the kernel of the isogeny $p_1^{\star}\mathcal{A}\to p_2^{\star}\mathcal{A}$ is the multiplicative (respectively constant) $p$-group scheme. We have the following characterisation of the maps $p_1,p_2$ on the open varieties
\begin{center}
\begin{tabular}{ c|c|c } 
 & \'{e}tale & multiplicative\\ 
 \hline
 $p_1$ & totally ramified of degree $p$ & isomorphism \\ 
 \hline
 $p_2$ & isomorphism & totally ramified of degree $p$ \\ 
\end{tabular}
\end{center}
In particular $(\tr_{p_1})_{\xi}: (\Oo_{X_0(\p)})_{\xi}\to (p_1^!\Oo_X)_{\xi}$ is an isomorphism at points in the multiplicative locus and factors through an isomorphism $(\tr_{p_1})_{\xi}: (\Oo_{X_0(\p)})_{\xi}\to p(p_1^!\Oo_X)_{\xi}$ at points of the \'{e}tale locus.

Moreover, the isogeny $p_1^{\star}\mathcal{A}\to p_2^{\star}\mathcal{A}$ is separable if and only if the kernel is \'{e}tale. Hence we have that for $\xi$ in the \'{e}tale locus, the differential map $(p_2^{\star}\omega)_\xi \to (p_1^{\star}\omega)_\xi$ is an isomorphism. And overall we find that for such $\xi$, the correspondence factors as in isomorphism
\begin{equation}
\tag{$\xi$ \'{e}tale}
\tilde{T}_{\p}: (p_2^{\star}{\omega}^{(\underline{k},w)})_{\xi} \xrightarrow{\simeq} p\cdot p^{\tfrac{w-k_{\p}}{2}}(p_1^{!}{\omega}^{(\underline{k},w)})_{\xi}.
\end{equation}

On the other hand, if $\xi$ is in the multiplicative locus, the differential map is an isomorphism of the components of the differential sheaves different from the one corresponding to $\p$ and factors through an isomorphism $(p_2^{\star}\omega_{\p})_\xi \to p(p_1^{\star}\omega_{\p})_\xi$ at the component corresponding to $\p$. Hence we get an isomorphism
\begin{displaymath}
(p_2^{\star}\omega^{(\underline{k},w)})_\xi \xrightarrow{\simeq} p^{k_{\p}+\tfrac{w-k_{\p}}{2}}(p_1^{\star}\omega^{(\underline{k},w)})_\xi.
\end{displaymath}
Overall we find that for $\xi$ in the multiplicative locus, the correspondence factors as in isomorphism
\begin{equation}
\tag{$\xi$ multiplicative}
\tilde{T}_{\p}: (p_2^{\star}{\omega}^{(\underline{k},w)})_{\xi} \xrightarrow{\simeq} p^{\tfrac{w+k_{\p}}{2}}(p_1^{!}{\omega}^{(\underline{k},w)})_{\xi}.
\end{equation}
Hence we have shown that multiplying $T_\p^{naive}$ by $p^{-\inf\lbrace\tfrac{w-k_{\p}}{2}+1,\tfrac{w+k_{\p}}{2}\rbrace}$ gives a well defined correspondence, optimally integral and that this vanishes on the multiplicative locus if $\tfrac{w-k_{\p}}{2}<\tfrac{w+k_{\p}}{2}-1$ and on the \'{e}tale locus if $\tfrac{w-k_{\p}}{2}>\tfrac{w+k_{\p}}{2}-1$. 
\end{proof}
\end{prop}

We obtain that the operator $T_\p$ induces by adujunction a map on cohomology obtained by the following composition
$$\RG(X, \myom^{(\underline{k},w)}) \xrightarrow{p_2^{\star}} \RG(X_0(\p), p_2^{\star}\myom^{(\underline{k},w)}) \xrightarrow{T_\p} \RG(X, \myom^{(\underline{k},w)}). $$
Hence we view $T_\p\in\operatorname{End}(\RG(X, \myom^{(\underline{k},w)}))$. We can similarly obtain a map on cuspidal cohomology and see the operator as $T_\p\in\operatorname{End}(\RG(X, \myom^{(\underline{k},w)}(-D)))$ (see \cite[Lemma 4.2.4.1]{pilloni}). 

We now want to understand the behaviour of this correspondence with respect to duality. We follow \cite[$\S$ 3.2]{boxer2020higher}, to which we refer for notation and a recap on duality. First, we need to fix some notation regarding dual isogenies. We consider the $\p$-isogeny $\pi:p_1^{\star}\mathcal{A}\to p_2^{\star}\mathcal{A}$. To be more precise, we have isogenies $\pi: p_1^{\star}\mathcal{A}^{\cc}\to p_2^{\star}\mathcal{A}^{\cc'}$ for our fixed choice of representatives $\cc,\cc'\in Cl_F^+$, where $\cc'$ is in the same class as $\cc\p$ but coprime to $\p$. We can consider the dual isogeny $\pi^{\vee}:(p_2^{\star}\mathcal{A^{\cc\p}})^{\vee}\to (p_1^{\star}\mathcal{A}^{\cc})^{\vee}$. Using the universal polarisations we obtain a map $p_2^{\star}\mathcal{A^{\cc'}}\otimes \cc' \to p_1^{\star}\mathcal{A}^{\cc}\otimes \cc$. Since we are working over $R$ and we have taken every $\cc,\cc'$ to be coprime to $p$, quotienting out by the units, we obtain a well-defined isogeny $\pi^D:p_2^{\star}\mathcal{A}\to p_1^{\star}\mathcal{A}$ and for any $(\underline{k},w)$, the pullback map on differentials 
\[
\pi^{D}_{(\underline{k},w)}:p_1^{\star}{\omega}^{(\underline{k},w)} \dashrightarrow p_2^{\star}{\omega}^{(\underline{k},w)}.
\]
We can then define $T_{\p,(\underline{k},w)}^t$ as $T_{\p,(\underline{k},w)}$ above, but replacing the pullback by $\pi$ with $\pi^{D}_{(\underline{k},w)}$.

We also recall that the composition $\pi\circ \pi^{D}$, seen as an isogeny $p_1^{\star}\mathcal{A}\to p_1^{\star}\mathcal{A}$ can be identified with the isogeny of kernel $p_1^{\star}\mathcal{A}[\p]$, depending on the choice of a uniformiser of $\Oo_{F,\p}$ (cfr \cite[3.6.3]{modularsurf} and \cite[3.10]{tianxiao}). This is done again using the universal polarisations and the fact that their types are coprime to $p$ to decompose such isogeny (which must be of degree $p^2$ and kernel contained in $p_1^{\star}\mathcal{A}[\p]$) as:
\[
\begin{tikzcd}
    p_1^{\star}\mathcal{A}^{\cc} 
    \arrow[r, "\pi"] 
    & p_2^{\star}\mathcal{A}^{\cc'} 
    \arrow[r, "\simeq"]\arrow[rrrr, bend right=10, "\pi^D"]
    & p_2^{\star}\mathcal{A}^{\cc\p} 
    \arrow[r, "\simeq"]
    & p_2^{\star}\mathcal{A}^{\cc\p,\vee} \otimes (\cc\p)^{-1} 
    \arrow[r, "\pi^\vee"] 
    & p_1^{\star}\mathcal{A}^{\cc,\vee} \otimes (\cc\p)^{-1} 
    \arrow[r, "\simeq"] 
    & p_1^{\star}\mathcal{A}^{\cc} \otimes \p^{-1} =  p_1^{\star}\mathcal{A}^{\cc} / p_1^{\star}\mathcal{A}^{\cc}[\p].
\end{tikzcd}
\]
Hence we find that  
\begin{equation}\label{eqdual}
\pi_{(\underline{k},w)}\circ \pi^{D}_{(\underline{k},w)}=p^{k_{\p}+2\tfrac{w-k_{\p}}{2}}\id=p^w\id.
\end{equation}

{\begin{prop}\label{dualtp}
The dual of the Hecke operator $T_{\p,(\underline{k},w)}$ acting on the cohomology of $\myom^{(\underline{k},w)}$ is equal to the Hecke operator $T^t_{\p,(2-\underline{k},-w)}$ acting on the cohomology of $\myom^{(2-\underline{k},-w)}(-D)$, where $(2-\underline{k})_{\tau}=2-k_\tau$.
\begin{proof}
We proceed similarly as in the proof of \cite[Proposition 3.6]{boxer2020higher}. The Kodaira--Spencer isomorphism \eqref{eq:KS} induces isomorphisms
\begin{displaymath}
p_i^\star KS: p_i^\star \myom^{(\underline{2},0)}\to p_i^\star \Omega^n_{X/\Z_p}(D), \ \ \ \ \text{where }\myom^{(\underline{2},0)}=\bigotimes_{\tau \in \Sigma_{\infty}}\left(\left(\wedge^{2} \mathcal{H}_{\tau}^{1}\right)^{-1} \otimes {\omega}_{\tau}^{2}\right).
\end{displaymath}
$\text{where }\myom^{(\underline{2},0)}=\bigotimes_{\tau \in \Sigma_{\infty}}\left(\left(\wedge^{2} \mathcal{H}_{\tau}^{1}\right)^{-1} \otimes {\omega}_{\tau}^{2}\right).$ Then it can be observed as in \cite[Lemma 3.7]{boxer2020higher} that the composition
\[
\operatorname{Tr}_{p_2}\operatorname{Tr}_{p_1}^{-1}\otimes (\pi^{D})^{\star}: p_1^!\Oo_X\otimes p_1^\star\myom^{(\underline{2},0)}(-D)\longrightarrow p_2^!\Oo_X\otimes p_2^\star \myom^{(\underline{2},0)}(-D)
\]
is the identity on the sheaf of differentials of $X_0(\p)$. 
Here, by abuse of notation, we write $\operatorname{Tr}_{p_1}$ for the fundamental class $\Theta$ and, similarly $\operatorname{Tr}_{p_2}$ for the fundamental class associated to the morphism $p_2$. Moreover note that 
\begin{equation}\label{eq:twistKS}
\myom^{(\underline{2},2)}=\myom^{(\underline{2},0)}\otimes \left(\bigotimes_{\tau \in \Sigma_{\infty}}\wedge^{2} \mathcal{H}_{\tau}^{1}\right)=\bigotimes_{\tau \in \Sigma_{\infty}}{\omega}_{\tau}^{2}.
\end{equation}
In particular the complexes $\RG(X,\myom^{(\underline{2},2)})$ and $\RG(X,\myom^{(\underline{2},0)})$ are isomorphic after twisting the $G$-action on the first one by the inverse adelic norm of the determinant. The map $\pi^{D,\star}$ is the one induced by the double coset of a matrix 
of determinant $\lambda_\p$, where $\lambda_{\p}$ is a local uniformiser of $\Oo_{F,\p}$ and $\|\lambda_\p\|=p^{-1}$. This is why in this setting we don't have the multiplication by $p$ appear as in \cite[Lemma 3.7]{boxer2020higher}, where in the case of the modular curve the sheaf considered is $\myom^2$ and not $\myom^2\otimes (\wedge^{2} \mathcal{H}^{1})^{-1}$.

We need to compare $\tilde{T}_\p^t=\tilde{T}^t_{\p,(2-\underline{k},-w)}$ with
$D(\tilde{T}_{\p})$ where $\tilde{T}_{\p}$ is the operator in weight $(\underline{k},w)$ acting on cohomology via
$$\RG(X, \myom^{(\underline{k},w)}) \xrightarrow{p_2^{\star}} \RG(X_0(\p), p_2^{\star}\myom^{(\underline{k},w)}) \xrightarrow{\tilde{T}_{\p}} \RG(X_0(\p), p_1^{!}\myom^{(\underline{k},w)}) \xrightarrow{\operatorname{Tr}_{p_1}} \RG(X, \myom^{(\underline{k},w)}).$$
It dualises to an operator
$$D(\RG(X, \myom^{(\underline{k},w)})) \xrightarrow{p_1^{\star}} \RG(X_0(\p), p_1^{\star}D(\myom^{(\underline{k},w)})) \xrightarrow{D(\tilde{T}_{\p})} \RG(X_0(\p), p_2^{!}D(\myom^{(\underline{k},w)})) \xrightarrow{\operatorname{Tr}_{p_2}} D(\RG(X, \myom^{(\underline{k},w)})), $$
where $D(\tilde{T}_{\p})$ can be written as the composition $$D(\tilde{T}_{\p}): p_1^{\star}({\omega}^{(-\underline{k},-w)}\otimes \Omega^n_{X/\Z_p}) \xrightarrow{\operatorname{id}\otimes \operatorname{Tr}_{p_1}}p_1^{!}({\omega}^{(-\underline{k},-w)}\otimes \Omega^n_{X/\Z_p})\to p_2^{!}({\omega}^{(-\underline{k},-w)}\otimes \Omega^n_{X/\Z_p}),$$
where the second map is $(\pi^\star)^{-1}\otimes id: p_1^\star\myom^{(-\underline{k},-w)}\otimes p_1^!\Omega^n_{X/\Z_p}\to p_2^\star\myom^{(-\underline{k},-w)}\otimes p_2^!\Omega^n_{X/\Z_p}$.
Applying the Kodaira-Spencer isomorphism as above we identify
\[
\omega_{X_0(\p)}=p_i^! \Omega^n_{X/\Z_p}= p_i^!\Oo_X\otimes p_i^\star \myom^{(\underline{2},0)} =  p_i^!\Oo_X\otimes p_i^\star \myom^{(\underline{2},2)}\otimes p_i^\star(\otimes_{\tau \in \Sigma_{\infty}}(\wedge^{2} \mathcal{H}_{\tau}^{1})^{-1})
\]
and the identity map as $\operatorname{Tr}_{p_2}\operatorname{Tr}_{p_1}^{-1}\otimes (\pi^{D})^{\star}$. 
As explained in \eqref{eqdual} above, we have
\begin{displaymath}
\pi_{(-\underline{k},-w)}\circ \pi_{(-\underline{k},-w)}^{D}= p^{-w}.
\end{displaymath}
Putting everything together the operator $D(\tilde{T}_{\p})$ acts as
$$p^{w}\tilde{T}_{\p}^t = \tr_{p_2}\otimes (\pi_{(\underline{2},0)}^{D}\otimes p^{w}\pi_{(-\underline{k},-w)}^{D})\in \End(\RG(X,\myom^{(2-\underline{k},-w)}(-D) ))$$

From the equalities
\[
\begin{cases}
-\tfrac{w+k}{2} + w  \\
-\tfrac{w-k}{2}-1 +w\end{cases}
= \begin{cases}
-\tfrac{k-w}{2}=-\tfrac{-w-(2-k)}{2}-1\\
-\tfrac{-k-w}{2}+1=-\tfrac{-w+(2-k)}{2}
\end{cases}
\]
and the definitions of the normalisation factors in the different weights, we deduce the final result.
\end{proof}
\end{prop}

\begin{rmk}\label{rmkallp}
In particular, letting $T_p=\prod_{\p\mid p}T_{\p}$, we can describe its transpose using the Atkin--Lehner operator and the diamond operator $\langle p\rangle$ (the latter multiplying the level structure by $p$). We have $T_p^t=\prod_{\p\mid p}T_{\p}^t=\langle p \rangle^{-1} T_p$ and therefore we obtain from the previous proposition that the dual of $T_{p,(\underline{k},w)}$ is $\langle p \rangle^{-1} T_{p,(2-\underline{k},w)}$. In order to write this for the single operators $T_{\p}$ without the assumption that $\p$ is principal, one would have to more carefully define the diamond operator for $\p$.
\end{rmk}

\subsection{Partial Hasse invariants and the Goren--Oort stratification}\label{partialhassesec}
We recall the definition of partial Hasse invariants of \cite{gorenhasse, gorenoort, AG}. We follow the exposition of \cite[$\S$ 3.1]{redu}.

Let $\Ff$ be the residue field considered above and let subscript $\Ff$ denote the base change to $\operatorname{Spec}\Ff$. Consider the Verschiebung isogeny
\[
V: \mathcal{A}_{\Ff}^{(p)}\to \mathcal{A}_{\Ff}.
\]
It induces a morphism of $\Oo_F\otimes \Oo_{\mathcal{M}_{\Ff}}$-modules $\myom_{\Ff}\to \myom_{\Ff}^{(p)}$. For every $\tau\in \Sigma_{\infty}$ this gives a map $\omega_{\Ff,\tau}\to \omega_{\Ff, \tau}^{\otimes p}$. Note that in general the identification $\omega_{\mathcal{A}^{(p)}_{\Ff}}\simeq (\omega_{\mathcal{A}_{\Ff}})^{\otimes p}$ is not $\Oo_F$-linear, but induces $\omega_{\mathcal{A}^{(p)}_{\Ff},\tau}\simeq (\omega_{\mathcal{A}_{\Ff},\sigma^{-1}\circ \tau})^{\otimes p}$, where $\sigma$ is the Frobenius automorphism of $\Ff$. However, as recalled above, in this particular setting where $p$ splits completely, we have $\sigma^{-1}\circ \tau = \tau$. Therefore we get a section in $H^0(\mathcal{M}_{\Ff},\omega_{\Ff,\tau}^{\otimes (p-1)})$ which is invariant under the action of $\Delta(K)$ and hence descends to a section
\begin{displaymath}
h_{\tau}\in H^0(X_{\Ff},\omega_{\Ff,\tau}^{\otimes (p-1)}),
\end{displaymath}
which is called \emph{partial Hasse invariant at $\tau$} (or at $\p$ when identifying $\Sigma_{\infty}$ with $\Hom_{\Z}(\Oo_F, \Ff)$). The product of all the partial Hasse invariants (which is induced by the differential of the Vershiebung) is the usual \emph{total} Hasse invariant $h\in H^0(X_{\Ff},\myom_{\Ff}^{(p-1, \dots, p-1)})$.

\begin{rmk}\label{rmkhassehecke}
Note that, when looking at the special fibre points of the isogeny defining the cohomological correspondence $\tilde{T}_{\p}$ (as in the proof of Proposition \ref{propTpnorm}), in the étale locus the $\p$-component of the pullback of the isogeny $p_1^{\star}\mathcal{A}\to p_2^{\star}\mathcal{A}$ identifies with the partial Hasse invariant $h_{\p}$.
\end{rmk}

Finally let us recall the following results about the vanishing loci of the partial Hasse invariants.

\begin{prop}[\cite{gorenoort, AG}]
Let $D_{\tau}=Va(h_\tau)$ be the vanishing locus of $h_\tau$. $Va(h)=\cup_{\tau\in \Sigma_{\infty}}D_{\tau}$ is a proper, reduced divisor on $X_{\Ff}$ with simple normal crossing. For any $S\subset \Sigma_{\infty}$, $\cap_{\tau\in S} D_{\tau}$ is a regular subvariety of codimension $\# S$. Moreover, $Va(h)$ does not intersect the toroidal boundary $D$.
\end{prop}

The analogous results holds for the vanishing loci over the moduli space $\mathcal{M}_{\Ff}$. Moreover, $\mathcal{M}_{\Ff} - Va(h)$ is the open subscheme of $\mathcal{M}_{\Ff}$ where the universal abelian variety $\mathcal{A}_{\Ff}$ is ordinary and $D_{\tau}$ is the closed subscheme where $\mathcal{A}_{\Ff}$ is supersingular at $\p$ (the prime above $p$ corresponding to the embedding $\tau$), or, in the language of \cite{gorenhasse}, the type of $\mathcal{A}_{\Ff}$ contains the vector $(x_{\p'})_{\p'\mid p}$, where $x_{\p'}=\emptyset$ if $\p'\neq \p$ and $x_{\p'}=\{1\}$ if $\p'=\p$. This amounts to saying that it is the locus where the rank 2 group scheme given by the $\p$ torsion is of multiplicative rank equal to zero. In particular, the multiplicative rank of $\mathcal{A}_{\Ff}$ is $\leq (n-1)$ over $D_{\p}$ for every $\p$ and is maximal, i.e. equal to $n$ over the ordinary locus. We will write
\begin{displaymath}
\mathcal{M}_{\Ff}^{ord}=\mathcal{M}_{\Ff}-Va(h); \ \ \ X_{\Ff}^{\ord}=X_{\Ff} -Va(h).
\end{displaymath} 
\subsubsection{Lifts of the Hasse invariants}\label{liftshasse} We now recall that suitable powers of the partial Hasse invariants lift to the reduction of $X$ modulo $\wp^n$. More precisely, we denote by $X_n$ the base change $X\times_{\operatorname{Spec}R}\operatorname{Spec}(R/\wp^n)$. In particular for $n=1$ we have $X_{\Ff}=X_1$. 

Let $\mathcal{U}\subset X_n$ be an open affine subscheme. As explained in \cite[$\S$ 3.3.1]{redu}, the restriction of the partial Hasse invariant $h_{\tau}$ to $\mathcal{U}\times_{\operatorname{Spec}(R/\wp^n)}\operatorname{Spec}\Ff$ can be lifted to an element $\tilde{h}_{\tau,\mathcal{U}}$ in $H^0(\mathcal{U},\omega_{\tau}^{\otimes (p-1)})$, where by abuse of notation we are still denoting by $\omega_{\tau}$ its base change over $R/\wp^n$. We find that $(\tilde{h}_{\tau,\mathcal{U}})^{p^{n-1}}$ is independent on the choice of the lift $\tilde{h}_{\tau,\mathcal{U}}$. One then deduces that the sections $\{ (\tilde{h}_{\tau,\mathcal{U}})^{p^{n-1}}\}_{\mathcal{U}}$, for $\mathcal{U}$ varying over an open affine covering of $X_n$, glue into a global section
\begin{displaymath}
\tilde{h}_{\tau,n} \in H^0(X_n, \omega_{\tau}^{\otimes p^{n-1}(p-1)})
\end{displaymath}
which is a lift of the $p^{n-1}$-th power of $h_{\tau}\in H^0(X_{1},\omega_{\Ff,\tau}^{\otimes (p-1)})$ (unique up to units). We let
\begin{displaymath}
D_{\tau,n}:=Va(\tilde{h}_{\tau,n})
\end{displaymath}
be the divisor on $X_n$ given by the vanishing locus of $\tilde{h}_{\tau,n}$. Under the natural map $X_1\to X_n$, the divisor $p^{n-1}\cdot D_\tau$ on $X_1$ is mapped to $D_{\tau,n}$. Moreover, since $(\tilde{h}_{\tau,n})^p$ is the unique lift of $(h_{\tau})^{p^n}$ over $X_n$, we have that for every $n$, $\tilde{h}_{\tau,n+1} \in H^0(X_{n+1}, \omega_{\tau}^{\otimes p^{n}(p-1)})$ is a lift of $(\tilde{h}_{\tau,n})^p$. In particular, the divisor $p\cdot D_{\tau,n}$ is mapped to $D_{\tau,n+1}$ under the map $X_{n}\to X_{n+1}$. We let as above $X_n^{\ord}= X_n-Va(\prod_{\tau}\tilde{h}_{\tau,n})$.

\section{Higher Hida theory}\label{sechida}
In this section, we finally move to the construction of the higher Hida theory $\Lambda$-modules. Under the identification (\ref{eq:sigma}), we will denote by $k_{\p}$ the $\p$-component of a vector $\underline{k}\in \Z^{\Sigma_{\infty}}$.
\subsection{Mod $p$ theory}\label{modpsection} Consider $X_1$ the special fibre of the compact Hilbert modular variety $X$. In this section we want to prove a mod $\wp$ control theorem (Theorem \ref{modpclassical}), which will be used crucially to prove the classicality results of the next section. 

Let us denote by $X_0(\p)_1^{\et}$ and $X_0(\p)_1^{m}$ the \'{e}tale and multiplicative locus of $X_0(\p)_1$ and with $i^{\et},i^m$ the inclusions into $X_0(\p)_1$. We denote by $p_i^{\et},p_i^{m}$ the restriction of the projections $p_i$ to $X_0(\p)_1^{\et}$ and $X_0(\p)_1^{m}$ respectively.

\begin{lemma} If $k_\p>1$ we have the following factorisation
\[
\begin{tikzcd}
p_2^{\star}{\omega}^{(\underline{k},w)} \arrow{r}{T_\p} \arrow{d} &p_1^{!}{\omega}^{(\underline{k},w)}\\
i^{\et}_{\star}(p_2^{\et})^{\star} {\omega}^{(\underline{k},w)}\arrow{r} & i^{\et}_{\star}(p_1^{\et})^{!}{\omega}^{(\underline{k},w)}.\arrow{u}
\end{tikzcd}
\]
If $k_\p<1$ we have the following factorisation
\[
\begin{tikzcd}
p_2^{\star}{\omega}^{(\underline{k},w)} \arrow{r}{T_\p} \arrow{d} &p_1^{!}{\omega}^{(\underline{k},w)}\\
i^{m}_{\star}(p_2^{m})^{\star} {\omega}^{(\underline{k},w)}\arrow{r} & i^{m}_{\star}(p_1^{m})^{!}{\omega}^{(\underline{k},w)}.\arrow{u}
\end{tikzcd}
\]
\begin{proof}
The result follows from the study of the correspondence $T_\p$ on the special fibre carried out in the proof of Proposition \ref{propTpnorm}.
More precisely, since $i^{\et}$ and $i^{m}$ are closed immersions, there are natural isomorphisms of functors $ i^{\et}_{\star}(p_1^{\et})^{!} = \underline{\Gamma}_{X_0(\p)_1^{\et}}p_1^!$ and $ i^{m}_{\star}(p_1^{m})^{!} = \underline{\Gamma}_{X_0(\p)_1^{m}}p_1^!$, where $\underline{\Gamma}_{Y}$ denotes the subsheaf of sections with support on a closed susbscheme $Y$ (see \cite[Proposition 2.3]{boxer2020higher}). The proposition then follows from the fact that the image of $T_\p$ vanishes at maximal points of $X_0(\p)_1^{m}$ if $k_\p>1$ and of $X_0(\p)_1^{\et}$ if $k_\p<1$ as shown in Proposition \ref{propTpnorm}.
\end{proof}
\end{lemma} 

\begin{prop}\label{prop:shiftTpmodp}
If $k_\p>1$, $T_\p$ induces a map
\begin{displaymath}
p_2^{\star}{\omega}^{(\underline{k},w)}((np+k_\p -2)D_\p) \to p_1^{!}{\omega}^{(\underline{k},w)}(nD_\p).
\end{displaymath}
If $k_\p<1$, $T_\p$ induces a map
\begin{displaymath}
p_2^{\star}{\omega}^{(\underline{k},w)}(-nD_\p) \to p_1^{!}{\omega}^{(\underline{k},w)}((-np+k_\p)D_\p).
\end{displaymath}
\begin{proof}
Assume that $k_\p>1$. The correspondence $T_{\p}$ is supported on $X_0(\p)_1^{\et}$ and, restricting to the intersection of this locus with the open Shimura variety, we know that $p_2^{\et}$ is an isomorphism and $p_1^{\et}$ is totally ramified of degree $p$. The divisor $D_\p$ does not intersect the toroidal boundary and we therefore obtain that $(p_1^{\et})^{\star}(D_\p)=p(p_2^{\et})^{\star}(D_\p)$. By a slight abuse of notation we still denote by $D_\p$ the divisor $(p_2^{\et})^{\star}(D_\p)$. Twisting by $\Oo_{X_0(\p)_1^{\et}}(npD_\p)$, if $k_\p=2$ the cohomological correspondence $p_2^{\star}{\omega}^{(\underline{k},w)}\to p_1^{!}{\omega}^{(\underline{k},w)}$ induces a morphism $p_2^{\star}{\omega}^{(\underline{k},w)}((np)D_\p) \to p_1^{!}{\omega}^{(\underline{k},w)}(nD_\p)$. If $k_\p\gneq 2$ and $w$ is even, the cohomological correspondence can be written as the tensor product of
\begin{displaymath}
(p_2^{\et})^{\star}{\omega}^{(\underline{k}',w)} \to (p_1^{\et})^{!}{\omega}^{(\underline{k}',w)} \ \ \ \ \text{ and } \ \ \  (p_2^{\et})^{\star}\left(\omega_\p^{k_\p-2}\otimes (\wedge^2\hh^1_\p)^{\tfrac{2-k_\p}{2}}\right) \to (p_1^{\et})^{\star}\left(\omega_\p^{k_\p-2}\otimes (\wedge^2\hh^1_\p)^{\tfrac{2-k_\p}{2}}\right),
\end{displaymath}
where $\underline{k}'_\p=2$, $\underline{k}'_\q=\underline{k}_\q$ if $\q\neq \p$. The $\p$-component of the differential of the isogeny $(p_1^{\et})^{\star}\mathcal{A}\to (p_2^{\et})^{\star}\mathcal{A}$ identifies with the partial Hasse invariant $h_\p$ and induces a map $(p_2^{\et})^{\star}\omega_\p(D_\p)\to (p_1^{\et})^{\star}\omega_\p$. Combining this with the result for $k_\p=2$, we obtain a map $p_2^{\star}{\omega}^{(\underline{k},w)}((np+k_\p -2)D_\p) \to p_1^{!}{\omega}^{(\underline{k},w)}(nD_\p).$ It remains to discuss the case of $w$ odd, which can be treated similarly, writing the correspondence as tensor product of 
\begin{displaymath}
(p_2^{\et})^{\star}{\omega}^{(\underline{k}'',w-1)} \to (p_1^{\et})^{!}{\omega}^{(\underline{k}'',w-1)} \  \text{ and } \  (p_2^{\et})^{\star}(\omega_\p^{k_\p}\otimes (\wedge^2\hh^1_\p)^{3-k_\p /2}\otimes \bigotimes_{\q\neq \p} \omega_{\q}) \to (p_1^{\et})^{\star}(\omega_\p^{k_\p}\otimes (\wedge^2\hh^1_\p)^{3-k_\p /2}\otimes \bigotimes_{\q\neq \p} \omega_{\q}),
\end{displaymath}
where $\underline{k}''_\p=2$, $\underline{k}''_\q=\underline{k}_\q-1 $ if $\q\neq \p$.

Now assume $k_\p<1$. In this case the correspondence is supported on $X_0(\p)_1^{m}$ and, since the role of $p_1^{m}$ and $p_2^{m}$ is interchanged, we find $(p_2^{m})^{\star}(D_\p)=p(p_1^{m})^{\star}(D_\p)$. Denoting by $D_\p$ the divisor $(p_1^{m})^{\star}(D_\p)$, we obtain that $T_\p$ induces a map $(p_2^m)^{\star}(\Oo_{X_1})_{\p}(-nD_\p) \to (p_1^m)^{!}(\Oo_{X_1})_\p(-npD_\p)$, which yields the case $k_\p =0$. For the case $k_\p\lneq 0$, we proceed as above and decompose the correspondence, reducing to study the map 
\begin{displaymath}
(p_2^{m})^{\star}{\omega}^{(\underline{k},w)} \to (p_1^{m})^{\star}{\omega}^{(\underline{k},w)}
\end{displaymath}
in the case where $k_\p \leq -1$. The $\p$-component of the differential of the isogeny dual to $\pi:(p_1^{m})^{\star}\mathcal{A}\to (p_2^{m})^{\star}\mathcal{A}$ can be identified with the partial Hasse invariant $h_\p$. More precisely recall that the composition
\begin{displaymath}
(p_2^{m})^{\star}\omega_{\p}^{-1}\xrightarrow{\pi^{\star}}(p_1^{m})^{\star}\omega_\p^{-1}\xrightarrow{(\pi^D)^{\star}} (p_2^{m})^{\star}\omega_\p^{-1}
\end{displaymath}
is given by multiplication by $p^{-1}$ (and multiplication by $p^{k_\p}$ when taking the $k_{\p}$-th power). Hence the correspondence $(p_2^{m})^{\star}\omega_{\p}^{k_\p}\to(p_1^{m})^{\star}\omega_\p^{k_\p}$ (which by our normalisation carries a multiplication by $p^{-k_\p}$), is given by $(h_\p)^{\otimes k_\p}$ and therefore induces a map $(p_2^{m})^{\star}\omega_{\p}^{k_\p}\to(p_1^{m})^{\star}\omega_\p^{k_\p}(-k_\p D_\p)$.
\end{proof}
\end{prop}
From the previous proposition we deduce the following two corollaries.
\begin{cor}\label{corjumps}
For all $k_\p>1$ and $n\geq 0$, $T_\p$ acts on $\RG(X_1,{\omega}^{(\underline{k},w)}(nD_\p))$ and for $n'\geq n$ the natural maps $\RG(X_1,{\omega}^{(\underline{k},w)}(nD_\p))\to \RG(X_1,{\omega}^{(\underline{k},w)}(n'D_\p))$ are equivariant for this action. Moreover, we have a commutative diagram
\[
\begin{tikzcd}
\RG(X_1,{\omega}^{(\underline{k},w)}(nD_\p))\arrow{r}\arrow{d}{T_\p} &\RG(X_1,{\omega}^{(\underline{k},w)}((np+k_\p-2)D_\p))\arrow{d}{T_\p}\arrow[dashed]{dl}{T_\p}\\
\RG(X_1,{\omega}^{(\underline{k},w)}(nD_\p))\arrow{r}{} &\RG(X_1,{\omega}^{(\underline{k},w)}((np+k_\p-2)D_\p)).
\end{tikzcd}
\]
For all $k_\p<1$ and $n\geq 0$, $T_\p$ acts on $\RG(X_1,{\omega}^{(\underline{k},w)}(-nD_\p))$ and for $n'\geq n$ the maps $\RG(X_1,{\omega}^{(\underline{k},w)}(-n'D_\p))\to \RG(X_1,{\omega}^{(\underline{k},w)}(-nD_\p))$ are equivariant for this action. Moreover, we have a commutative diagram
\[
\begin{tikzcd}
\RG(X_1,{\omega}^{(\underline{k},w)}((-np+k_\p)D_\p))\arrow{r}\arrow{d}{T_\p} &\RG(X_1,{\omega}^{(\underline{k},w)}(-nD_\p))\arrow{d}{T_\p}\arrow[dashed]{dl}{T_\p}\\
\RG(X_1,{\omega}^{(\underline{k},w)}((-np+k_\p)D_\p))\arrow{r}{} &\RG(X_1,{\omega}^{(\underline{k},w)}(-nD_\p)).
\end{tikzcd}
\]
\end{cor}
{Since the complexes considered above are bounded complexes of finite $R$-modules, their endomorphism ring is finite. Therefore for $n$ large enough, we have $T_{\p}^{n!}=T_{\p}^{(n+1)!}$ and $T^{n!}_{\p}$ is idempotent. Hence we can define the limit $e(T_\p)=\lim_{n\to \infty}T_\p^{n!}$ which is an idempotent endomorphism of such complexes.}
The above corollary hence implies the following crucial result.
\begin{cor}\label{corprojector}
After applying the idempotent $e(T_\p)$ the horizontal maps considered above become quasi-isomorphisms, namely if $n\geq 0$ then we have quasi-isomorphisms
\begin{displaymath}
e(T_\p)\RG(X_1,{\omega}^{(\underline{k},w)}(nD_\p)) \xrightarrow{\simeq} e(T_\p)\RG(X_1,{\omega}^{(\underline{k},w)}((np+k_\p-2)D_\p)) \ \ \ \ \text{if }k_\p>1;
\end{displaymath} 
\begin{displaymath}
e(T_\p)\RG(X_1,{\omega}^{(\underline{k},w)}((-np+k_\p)D_\p)) \xrightarrow{\simeq} e(T_\p)\RG(X_1,{\omega}^{(\underline{k},w)}(-nD_\p)) \ \ \ \ \text{if }k_\p<1.
\end{displaymath}
\end{cor}

We are finally ready to define the objects for which we can prove classicality results. 

\begin{defi}
Consider $(\underline{k},w)$ a cohomological weight as above such that $k_\p\neq 1$ for every $\p\mid p$. Let $T_p=\prod_{\p\mid p}T_\p$, $J:=\{\p: k_\p<1\}$ and $i_J:=\# J$. We then let
\begin{displaymath}
M(\underline{k},w):=\left(\limm_{n_\p}\right)_{\p\in J} \left(\colim_{n_\p}\right)_{\p\not\in J} e(T_p)\RG(X_1, {\myom}^{(\underline{k},w)}(\sum_{\p\not\in J}n_{\p}D_\p -\sum_{\p\in J}n_{\p}D_\p));
\end{displaymath}
\begin{displaymath}
M(\underline{k},w)(-D):=\left(\limm_{n_\p}\right)_{\p\in J} \left(\colim_{n_\p}\right)_{\p\not\in J} e(T_p)\RG(X_1, {\myom}^{(\underline{k},w)}(-D+\sum_{\p\not\in J}n_{\p}D_\p -\sum_{\p\in J}n_{\p}D_\p)).
\end{displaymath}
\end{defi}

Notice that this is well defined thanks to Corollary \ref{corprojector}, which tells you that there is no ambiguity when ``deciding in which order taking the limits and colimits''. 

\begin{thm}\label{modpclassical}
If for all $k_\p$, we have $k_\p\leq -1$ when $\p\in J$ and $k_\p\geq 3$ when $\p\not\in J$, then there are quasi-isomorphisms
\begin{displaymath}
M(\underline{k},w) \simeq e(T_p)\RG(X_1,{\omega}^{(\underline{k},w)}) \ \ \ \text{and} \ \ \ M(\underline{k},w)(-D) \simeq e(T_p)\RG(X_1,{\omega}^{(\underline{k},w)}(-D)). 
\end{displaymath}
Moreover $M(\underline{k},w)(-D)$ is a perfect complex concentrated in degrees $[0,i_J]$ and $M(\underline{k},w)$ is a perfect complex concentrated in degrees $[i_J,n]$.
\begin{proof}
The classicality isomorphisms follow from Corollary \ref{corprojector}, which tells us that, under the above conditions on the weights, the transition maps in the limits are isomorphisms. Since the cohomology of the complexes $\RG(X_1,{\omega}^{(\underline{k},w)})$ and $\RG(X_1,{\omega}^{(\underline{k},w)}(-D))$ is finite and $e(T_p)$ is an idempotent, so is the cohomology of $e(T_p)\RG(X_1,{\omega}^{(\underline{k},w)})$ and of $e(T_p)\RG(X_1,{\omega}^{(\underline{k},w)}(-D))$, and, by the classicality isomorphisms, the one of $M(\underline{k},w)$ and $M(\underline{k},w)(-D)$.

We now prove the statement about the vanishing of the cohomology for $M(\underline{k},w)(-D)$. The vanishing for $M(\underline{k},w)$ will follow by this, combined with Proposition \ref{dualtp}, the classicality isomorphism, and duality. Let 
\begin{displaymath}
\RG(\underline{k},w)(-D):=\left(\colim_{n_\p}\right)_{\p\not\in J} e(T_{p})\RG(X_1, {\myom}^{(\underline{k},w)}(-D+\sum_{\p\not\in J}n_{\p}D_\p ).
\end{displaymath}
First we notice that $\RG(\underline{k},w)(-D)= e(T_{p}) \RG(X_1\setminus (\cup_{\p\not\in J}D_{\p}), {\myom}^{(\underline{k},w)}(-D)$. Moreover, by Corollary \ref{corprojector}, we have $\RG(\underline{k},w)(-D) \simeq e(T_{p})\RG(X_1, {\myom}^{(\underline{k},w)}(-D)) \simeq M(\underline{k},w)(-D)$. 
Therefore to prove the claim for the latter complex, we will prove that $\RG(\underline{k},w)(-D)$ is concentrated in degrees $[0, i_J]$.
We use the stratification
\begin{displaymath}
Z_0=X_1\setminus (\cup_{\p\not\in J}D_{\p})\supset Z_{1}=Z'_{n-1}\setminus (\cup_{\p\not\in J}D_{\p})\supset \dots \supset Z_n = Z'_0 \setminus (\cup_{\p\not\in J}D_{\p}) \supset Z_{n+1}=\emptyset,
\end{displaymath}
where $Z'_i$ is the union of the Ekedahl-Oort stratum of $X_i$ of dimension $i$, which in our case is given by the closed subspace of $X_1$ where the multiplicative rank of the universal $p$-divisible group is $\leq i$, which in other words is $\cup_{\p_1 \neq \p_2 \dots \neq \p_{n-i}}(D_{\p_1}\cap D_{\p_2}\dots \cap D_{\p_{n-i}})$. By the theory of generalised Hasse invariants of \cite{boxerthesis, strata}, one has that $Z'_i\setminus Z'_{i-1}$ is affine and we hence have that $Z_i\setminus Z_{i+1}$ is affine (where for $i=0$ the statement holds only inside the minimal compactification). Now we prove that the cohomology of $\RG(X_1\setminus (\cup_{\p\not\in J}D_{\p}), {\myom}^{(\underline{k},w)}(-D))$ vanishes for $i>i_J$. Let $\myom:={\myom}^{(\underline{k},w)}(-D)$ and $X_{1,J}:=X_1\setminus (\cup_{\p\not\in J}D_{\p})$. It follows from \cite[Theorem 3.9.6]{modularsurf} that $\myom \to Cous_Z(\myom)$ is a quasi-isomorphism, where $Cous_Z(\myom)$ is the Cousin complex associated with the stratification $Z=(Z_i)$ given above. We claim that the cohomology of $\operatorname{\RG}(X_{1,J}, \myom)$ is computed by $\Gamma(X_{1,J}, Cous_Z(\myom))$. To see this, we write explicitly the complex $Cous_Z(\myom)$ and show that it is a complex of acyclic sheaves. Since it is a complex of length $i_J$, this concludes the proof. Let $\mathcal{L}:= \bigotimes_{\p\in J} \omega_{\p}^{(p-1)}$. By \cite[Remark 4.2.32]{modularsurf}, $Cous_Z(\myom)$ is given by
\begin{align*}
0 \to &\ccolim_{k} \myom\otimes \mathcal{L}^k\to \ccolim_{k} \bigoplus_{\p\in J}\left(\myom\otimes (\mathcal{L}^k/h_{\p}^k)\right) \\&\to \ccolim_{k} \bigoplus_{\p_i\neq \p_j\in J}\left(\myom\otimes (\mathcal{L}^k/(h_{\p_i}^k,h_{\p_j}^k )\right)\to \dots \to \ccolim_{k} \myom\otimes (\mathcal{L}^k/(\sum_{\p\in J} h_{\p}^k)\to 0
\end{align*}
Note that $\myom\otimes \mathcal{L}^k= (\bigotimes_{\p\not\in J}(\omega_{\p}^{k_{\p}}\otimes \wedge^2(\mathcal{H}^1_{\p})^{(w-k_\p)/2})\otimes (\bigotimes_{\p\in J}(\omega_{\p}^{k_{\p}+k(p-1)}\otimes \wedge^2(\mathcal{H}^1_{\p})^{(w-k_\p)/2})(-D)$. We can replace all the colimits over $k\geq 0$ by the same colimits over $k\geq \max_{\p\in J}n_{\p}$. Moreover, \cite[Lemma 4.2.31]{modularsurf} tells us that
\begin{displaymath}
\ccolim_{k} \mathcal{L}^{k}/(h_{p_1}^k,\dots,h_{p_i}^k)\simeq \ccolim_{k} \mathcal{L}^{k}/(h_{p_1}^k,\dots,h_{p_i}^k)_{|X_{1,J}\setminus D_{\p_1,\dots,\p_i}},
\end{displaymath}
where $D_{\p_1,\dots,\p_i}=$Va$((\prod_{\p\in J}h_{\p})\cdot h_{\p_1}^{-1}\cdots h_{\p_i}^{-1})$.\\
Combining these observations, we obtain that every summand appearing in the $(i+1)-$th term of the above sequence is supported on Va$((h_{p_1}^k,\dots,h_{p_i}^k))\cap (X_{1,J}\setminus D_{\p_1,\dots,\p_i})$ for $\p_1,\dots,\p_i\in J$ and some $k\geq 1$. This support is equal to $(D_{\p_1}\cap \dots \cap D_{\p_i})\setminus (\cup_{\mathfrak{q}\nmid \p_1\cdots\p_i}D_{\q})$, which is affine in the minimal compactification again by the theory of generalised Hasse invariants of \cite{boxerthesis} (we can restrict to $(D_{\p_1}\cap \dots \cap D_{\p_i})$ the Hasse invariant defined on $Z'_{n-i}$ and vanishing on the $\leq n-i-1$ locus, then our support is the complement of the vanishing locus of this restriction). Since the sheaves appearing in the exact sequence are acyclic with respect to the minimal compactification by Lemma \ref{lancuspminimal} and their support is affine in the minimal compactification, we have shown $Cous_{Z}(\myom)$ is a complex of acyclic sheaves. Then $\operatorname{\RG}(X_{1,J}, \myom)$ is computed by $\Gamma(X_{1,J}, Cous_Z(\myom))$ and the latter is precisely of length $i_J$. 
\end{proof}
\end{thm}

\begin{rmk}\label{rmknaturalmapsmodp}
The quasi-isomorphisms in the statement of the above theorem can be obtained by the fact that the modules $M^{*}(\underline{k},w),M^{*}(\underline{k},w)(-D)$ are constant limits-colimits of $e(T_p)H^{*}(X_1,{\omega}^{(\underline{k},w)})$ and, respectively, of $e(T_p)H^{*}(X_1,{\omega}^{(\underline{k},w)}(-D))$. However we notice that we have natural maps 
\begin{align}\label{eq:nat}
M(\underline{k},w) \to (\colim_{n_\p})_{\p\not\in J} e(T_p)\RG(X_1, {\myom}^{(\underline{k},w)}(\sum_{\p\not\in J}n_{\p}D_\p)) \leftarrow e(T_p)\RG(X_1,{\myom}^{(\underline{k},w)}),\\
\label{eq:nat2}
M(\underline{k},w) \leftarrow (\limm_{n_\p})_{\p\in J} e(T_p)\RG(X_1, {\myom}^{(\underline{k},w)}(-\sum_{\p\in J}n_{\p}D_\p)) \to e(T_p)\RG(X_1,{\myom}^{(\underline{k},w)}),
\end{align}
where the first ones are given by the properties of limits and colimits respectively and the second ones by the ones of colimits and limits respectively. Hence, again by applying Corollary \ref{corprojector}, the statement of the theorem can be made more precise saying that all these four natural maps are quasi-isomorphisms if $k_\p\leq -1$ when $k_\p\in J$ and $k_\p\geq 3$ when $k_\p\not\in J$ (and the analogous statement for the cuspidal version).
\end{rmk}

\subsection{Characteristic zero theory} \label{char0sec}
Let $\mathfrak{X}$ be the formal completion of $X$ along its special fibre. It is the limit of the schemes $X_n=X\times_{\operatorname{Spec}R}\operatorname{Spec}(R/\wp^n)$. We denote by $\X^{\ord}$ the ordinary locus of $\X$. It is defined by choosing a lift of the total Hasse invariant of $\S$ \ref{partialhassesec} in characteristic zero and taking the formal completion of the subscheme $X^{\ord}$ of $X$ where such lift does not vanish. Even if $X^{\ord}$ does depend on the chosen lift, $X^{\ord}_n$ (and hence $\X^{\ord}$) does not. We use analogous notation for the formal completion $\M$ of $\mathcal{M}$ along $\mathcal{M}_1$ and its ordinary locus $\mathfrak{M}^{\ord}$.

In order to define the $p$-adic module we will introduce Igusa towers.
\subsubsection{Igusa tower sheaves}
We can consider the $\Z_p^{\times}$-torsor over $\M^{\ord}$ (actually defined over the larger locus corresponding to the universal abelian variety being $\p$-ordinary)
\begin{displaymath}
\pi_\p: \mathfrak{IG}_\p = \text{Isom}(\Z_p,T_\p(\mathcal{A})^{\et})\to \mathfrak{M}^{\ord}.
\end{displaymath}
More precisely one defines the scheme
\begin{displaymath}
\mathfrak{IG}_\p^{m,n}=\text{Isom}_{\mathcal{M}_m}(\Z/p^n, \mathcal{A}[\p^n]^{\et}),
\end{displaymath}
with the obvious action of $(\Z/p^n)^{\times}$ on the right. The natural morphism $\mathfrak{IG}_\p^{m,n}\to \mathcal{M}_m$ makes $\mathfrak{IG}_\p^{m,n}$ an étale cover of $\mathcal{M}_m$ with group $(\Z/p^n)^{\times}$. Letting $ \mathfrak{IG}_\p =\ccolim_m \llimm_n \mathfrak{IG}_\p^{m,n}$ we obtain a $\Z_p^{\times}$-torsor over the ordinary locus of the formal completion of $\mathcal{M}$ along $\mathcal{M}_1$.

Let $\Lambda_{\p}=R[[\Z_p^{\times}]]$ and denote by $\kappa_\p:\Z_p^{\times}\to \Lambda_{\p}^{\times}$ the universal character. We consider the invertible sheaf of $\Lambda_{\p}\hat{\otimes}\Oo_{\mathfrak{M}^{\ord}}$-modules
\begin{displaymath}
\Omega_{\p}^{\kappa}:= (\pi_{\p,\star}\Oo_{\mathfrak{IG}_\p} \otimes \Lambda_{\p})^{\Z_p^{\times}}.
\end{displaymath}
We can now consider the full Igusa tower, i.e. the $(\Z_p^{\times})^{\Sigma_{\infty}}$-torsor given as follows
\begin{displaymath}
\pi: \mathfrak{IG}=\text{Isom}_{\M^{\ord}\otimes \Oo_F}(\Z_p\otimes \Oo_F,T_p(\mathcal{A})^{\et})=\prod_{\p\mid p}\text{Isom}_{\M^{\ord}}(\Z_p,T_\p(\mathcal{A})^{\et}) \xrightarrow{\prod \pi_\p} \mathfrak{M}^{\ord}.
\end{displaymath}
We then define the sheaf
\begin{displaymath}
\Omega^{\kappa}:=(\pi_{\star}\Oo_{\mathfrak{IG}} \otimes R[[T_1\times \dots\times T_{n}]])^{\Z_p^{\times}\times \dots \times \Z_p^{\times}}
\end{displaymath}
where $(\Z_p^{\times})^{\Sigma_{\infty}}$ acts on $R[[T_1\times\dots\times T_{n}]]$, where $T_i=\Z_p^\times$, via the universal character $\kappa: (\Z_p^{\times})^{\Sigma_{\infty}}\to R[[T_1\times\dots\times T_{n}]]$ and on $\pi_{\star}\Oo_{\mathfrak{IG}}$ via the action on $\mathfrak{IG}$. 

For $\underline{k}=(k_i)_i\in \Z^{n}$ we write $\underline{k}$ for the $R$-valued homomorphism of $R[[T_1,\dots,T_{n}]]$ induced by the characters $x_i\mapsto x_i^{k_i}$, where $x_i\in T_i$. We hence can consider the sheaf
\[
\Omega^{\underline{k}}:=\Omega^{\kappa}\otimes_{\underline{k}}R.
\]
We now recall the construction of the Hodge-Tate map, which provides an isomorphism between $\Omega^{\underline{k}}$ and the restriction over the ordinary locus of the automorphic bundle $\tilde{\myom}^{(\underline{k},0)}$, constructed in $\S$ \ref{sectiontorsors} using the torsor $\mathcal{T}$.

Let $B$ be a $R/\wp^m$-algebra and $\mathcal{A}$ an ordinary semi-abelian scheme over $\operatorname{Spec}B$ of dimension $n$ with $\Oo_F$-multiplication and with polarization coprime to $p$. Let $e$ be the unit section and assume we are given a $\Oo_F$-linear trivialisation $\phi_n: \Z/p^n\Z \otimes \Oo_F\to \mathcal{A}[p^n]^{\et}$. The dual of this map (using the prime to $p$ polarisation) gives a trivialisation $\phi_n^D: \mathcal{A}[p^n]^{\circ}\to \mu_{p^n}\otimes \Oo_F$. For $n\geq m$, we obtain a $(\Z/p^n)^{\times}$-equivariant isomorphism
\begin{equation}\label{HTmn}
\HT_{m,n}(\phi_n): B\otimes \Oo_F \to e^* \Omega^1_{\mu_{p^n}\otimes \Oo_F} \xrightarrow{(\phi_n^D)^*} e^* \Omega^1_{\mathcal{A}/B},
\end{equation}
where the first map is given by sending an element $t_i$ of the basis of $B^n\simeq B\otimes \Oo_F$ to $dt_i/t$.
We then have a map $\HT_{m,n}:\mathfrak{IG}^{m,n}\to \mathcal{T}_{| \M^{\ord}_m}$ for $m\leq n$. Passing to the limits, we obtain a commutative diagram
\[
\begin{tikzcd}
\mathfrak{IG} \arrow{r}{\HT}\arrow{d} &\mathcal{T}\arrow{dl} \\
\M^{\ord}.
\end{tikzcd}
\]
Exploiting the commutativity of the Hodge-Tate map with respect to the $(\Z_p^{\times})^n$-action we obtain
\begin{lemma}\label{lemmaclassical1}
Let $\underline{k}=(k_i)_i\in \Z^{n}$. The Hodge-Tate map above gives a canonical isomorphism of $\Oo_{\M^{\ord}}$-modules
$\Omega^{\underline{k}} \simeq \tilde{\myom}^{(\underline{k},0)}$.
\end{lemma}
\begin{proof}
Sections of the sheaf $\Omega^{\underline{k}} $ are rules associating to $(x, \phi:\Z_p\otimes \Oo_F\simeq T_p(A_x)^{\et})\in \mathfrak{IG}(R)$ an element $f(x,\phi)\in R$ such that for every $\lambda=(\lambda_\p)_\p\in (\Z_p^{\times})^n$, 
\[
f(x, \phi\circ \lambda^{-1})= \prod_{\p\mid p}\lambda_{\p}^{k_{\p}}f(x,\phi).
\]
On the other hand, sections of the sheaf $\tilde{\myom}^{(\underline{k},0)}=(\pi_{\mathcal{T}})_*\Oo_{\mathcal{T}}[-\underline{k}]$ are rules associating to $(x, \omega:R\otimes \Oo_F \simeq e^*\Omega^1_{A_x/R})\in \mathcal{T}(R)$ an element $g(x,\omega)\in R$ such that for every $\lambda\in (R\otimes \Oo_F)^{\times}$
\[
g(x, \omega\circ \lambda^{-1})= \prod_{\p\mid p}\lambda_{\p}^{k_{\p}}g(x,\omega).
\]
The claimed isomorphism is then explicitly given by sending $g$ to the rule defined by
$(x, \phi)\mapsto g(x, \HT(\phi))$. The fact that this is an isomorphism follows as, for example, in \cite[$\S$4.2.1-4.2.2]{pillonibull}.
\end{proof}

We now want to twist the sheaf $\Omega^{\kappa}$ by a factor that will allow to recover after specialisation the sheaves $\tilde{\myom}^{(\underline{k},\underline{n})}$ and, in particular, the sheaves ${\myom}^{(\underline{k},w)}$ in light of \eqref{eq:isotorsoromega}. We perform a $p$-adic construction analogous to the one employed for constructing $\hat{\myom}^{(\underline{k},\underline{n})}$ and use Lemma \ref{lemmatildehat}.

Let $\mathfrak{IG}^{\vee}$ the torsor defined by $\text{Isom}_{\M^{\ord}\otimes \Oo_F}(T_p(\mathcal{A}^{\vee})^{\circ}, \mu_{p^{\infty}}\otimes \Oo_F)$. Via the Hodge-Tate map, we obtain a map $\mathfrak{IG}^{\vee}\to \mathcal{T}'$. More precisely, we define a map $(\mathfrak{IG}^{\vee})^{m,n}\to \mathcal{T'}_{| \M^{\ord}_m}$ for $m\leq n$. We start with a trivialisation $\phi_n: \mathcal{A}^{\vee}[p^n]^{\circ}\to \mu_{p^n}\otimes \Oo_F$ where $\mathcal{A}$ is an ordinary semi-abelian scheme over $\operatorname{Spec}B$ of dimension $n$ with $\Oo_F$-multiplication and with polarization coprime to $p$ and $B$ is a $\Z/p^m$-algebra. By fixing a canonical basis of $B\otimes \Oo_F$ as above we find the isomorphism
\begin{equation}\notag
\tilde{\HT}_{m,n}(\phi_n): B\otimes \Oo_F \to e^* \Omega^1_{\mu_{p^n}\otimes \Oo_F} \xrightarrow{(\phi_n)^*} e^* \Omega^1_{\mathcal{A}^{\vee}/B}.
\end{equation}
Identifying $B\otimes \Oo_F = \Hom_{B\otimes \Oo_F}(B\otimes \Oo_F,B\otimes \Oo_F)$, we obtain an isomorphism
\begin{equation}\label{HTmndual}
\HT_{m,n}(\phi_n):B\otimes \Oo_F \to \Hom_{B\otimes \Oo_F}(e^* \Omega^1_{\mathcal{A}^{\vee}/B}, B\otimes \Oo_F)
\end{equation}
defined by $g\mapsto g\circ (\tilde{\HT}_{m,n}(\phi_n))^{-1}$. Passing to the limits, we obtain a map $\HT: \mathfrak{IG}^{\vee}\to \mathcal{T}'$ commuting with the projections of the torsors over $\M^{\ord}$. 

We can now define a sheaf of $R[[T_1\times\dots\times T_n\times T_1'\times\dots\times T_n']]$-modules, where $T_i'=\Z_p^\times$.
\begin{defi} Consider $\mathbb{T}$ the maximal torus of $\res_{F/\Q}\GL_2(\Z_p)$ and write $\mathbb{T}=(\Z_p^{\times})^{n}\times (\Z_p^{\times})^{n} \ni (z,t)=\SmallMatrix{zt &0\\0&z^{-1}}$. We write $(\kappa, \kappa'): (\Z_p^{\times})^{n}\times \mathbb{T}\to R[[T_1\times\dots\times T_n\times T_1'\times\dots\times T_n']]$ for the character given by the universal character $\kappa:(\Z_p^{\times})^{n}\to R [[T_1\times\dots\times T_n]]$ and the character $\kappa':T(\Z_p) \to R [[T'_1\times\dots\times T'_n]]$ obtained by the composition of the projection map $\mathbb{T}\to (\Z_p^{\times})^{n}, (z,t)\to t$ with the universal character $(\Z_p^{\times})^{n}\to R [[T'_1\times\dots\times T'_n]]$. We then let
\begin{equation}\label{Migusasheaf}
\Omega^{(\kappa, \kappa')}:= \left(\left(\pi_{\star}\Oo_{\mathfrak{IG}} \otimes_{\Oo_{\M^{\ord}}} (\pi_{\mathfrak{IG}\times \mathfrak{IG}^{\vee}})_{\star}\Oo_{\mathfrak{IG}\times \mathfrak{IG}^{\vee}}\right) \hat{\otimes} R[[T_1\times\dots\times T_n\times T_1'\times\dots\times T_n']]\right)^{(\Z_p^{\times})^{n}\times \mathbb{T}},
\end{equation}
where $(\Z_p^{\times})^{n}\times \mathbb{T}$ acts on the Iwasawa algebra by $(\kappa,\kappa')$ and on the factors on the left as follows: the natural action of the first $n$-copies of $\Z_p^{\times}$ on $\mathfrak{IG}$ defines an action of $(\Z_p^{\times})^n$ on $\pi_{\star}\Oo_{\mathfrak{IG}}$ and $\mathbb{T}$ acts on $(\pi_{\mathfrak{IG}\times \mathfrak{IG}^{\vee}})_{\star}\Oo_{\mathfrak{IG}\times \mathfrak{IG}^{\vee}}$ via the decomposition $\mathbb{T}=(\Z_p^{\times})^{n}\times (\Z_p^{\times})^{n}$ and the natural actions of $(\Z_p^{\times})^{n}$ on each of the two terms. 
\end{defi}
This is, roughly speaking, the $p$-adic analogue of what we obtained via the pullback by $s$ of the sheaf $\Oo_{\mathcal{L}}$ in $\S$ \ref{sectiontorsors}. Similarly as above, for $(\underline{k},\underline{n})\in \Z^n\times \Z^n$, we denote again by $(\underline{k},\underline{n})$ the homomorphism of $\Z_p[[T_1\times\dots\times T_n\times T_1'\times\dots\times T_n']]$ given by the characters $x_i\mapsto x_i^{k_i}, x_i'\mapsto (x_i')^{n_i}$, for $x_i\in T_i, x_i'\in T_i'$. Let $$\Omega^{(\underline{k},\underline{n})}:=\Omega^{(\kappa,\kappa')}\otimes_{(\underline{k},\underline{n})}R.$$ We obtain the following result.
\begin{lemma}\label{lemmaMigusaclass}
Let $(\underline{k},\underline{n})\in \Z^n\times \Z^n$. There is a canonical isomorphism of $\Oo_{\M^{\ord}}$-modules $\Omega^{(\underline{k},\underline{n})}\simeq \tilde{\myom}^{(\underline{k},\underline{n})}$.
\end{lemma}
\begin{proof}
We prove $\Omega^{(\underline{k},\underline{n})}\simeq \hat{\myom}^{(\underline{k},\underline{n})}$ and the result will follow from Lemma \ref{lemmatildehat}. The proof is similar to the one of Lemma \ref{lemmaclassical1}. Sections of the sheaf $\Omega^{(\underline{k},\underline{n})} $ are rules associating to $(x, \phi:\Z_p\otimes \Oo_F\simeq T_p(A_x)^{\et}, \psi_1:\Z_p\otimes \Oo_F\simeq T_p(A_x)^{\et},\psi_2: T_p(A^{\vee}_x)^{\circ}\simeq \mu_{p^{\infty}}\otimes \Oo_F)$ an element $f(x,\phi, \psi_1,\psi_2)\in R$ such that for every $\lambda=(\lambda_\p)_\p, t=(t_\p)_\p, z=(z_\p)_\p \in (\Z_p^{\times})^n$, 
\[
f(x, \phi\circ \lambda^{-1},(\psi_1,\psi_2)\circ \SmallMatrix{zt &0\\0&z^{-1}}^{-1})= \prod_{\p\mid p}\lambda_{\p}^{k_{\p}}t_{\p}^{n_{\p}}f(x,\phi,\psi_1,\psi_2).
\]
On the other hand, sections of the sheaf $\hat{\myom}^{(\underline{k},\underline{n})}$ are rules associating to $(x, \omega:R\otimes \Oo_F \simeq e^*\Omega^1_{A_x/R},\alpha\otimes\beta)$, where $\alpha: R\otimes \Oo_F \simeq e^*\Omega^1_{A_x/R},\beta: R\otimes \Oo_F \simeq (e^*\Omega^1_{A_x^{\vee}/R})^{\vee}$ an element $g(x,\omega,\alpha\otimes\beta)\in R$ such that for every $\lambda,\mu \in (R\otimes \Oo_F)^{\times}$
\[
g(x, \omega\circ \lambda^{-1},(\alpha\otimes \beta)\mu^{-1})= \prod_{\p\mid p}\lambda_{\p}^{k_{\p}}\mu_{\p}^{n_{\p}}g(x,\omega,\alpha\otimes \beta).
\]
We can send $g$ to the rule $(x,\phi,\psi_1,\psi_2)\mapsto g(x,\HT(\phi),\HT(\psi_1)\otimes \HT(\psi_2))$.
\end{proof}
In particular, if there exists $w\in\Z$ such that $k_{\p}\equiv w\mod 2$ for every $\p\mid p$, we have an isomorphism
\[
\Omega^{(\underline{k},\tfrac{w-\underline{k}}{2})}\simeq \myom^{(\underline{k},w)}.
\]
\subsubsection{From the moduli space to the Shimura variety}\label{secdescentigusa} Let $T=\prod_{\p\mid p}(\Z_p)^{\times}$. Recall that the sheaf $\Omega^{(\kappa,\kappa')}$ defined above is a sheaf of $R[[T^2]]\otimes \Oo_{\M^{\ord}}$-modules. 

We now let $T_0=(\Z_p^{\times}/\{\pm 1\})^{n+1}$, which can be identified with a quotient of $T\times \Z_p^{\times}$. Note that for $x\in \Z_p^{\times}/\{\pm 1\}$ a square-root $x^{1/2}\in \Z_p^{\times}/\{\pm 1\}$ exists and is unique. We denote by $x_0$ the projection to $\Z_p^{\times}/\{\pm 1\}$ of an element $x\in \Z_p^{\times}$. Let
\[
\Lambda:=R[[T_0]]=(\hat{\otimes}_{\p\mid p}R[[\Z_p^{\times}/\{\pm 1\}]])\hat{\otimes}R[[\Z_p^{\times}/\{\pm 1\}]].
\]
We have a canonical character $(\kappa_1,\kappa_2): T_0 \to \Lambda^{\times}$, where $\kappa_1: (\Z_p^{\times}/\{\pm 1\})^{n}\to R[[T_0]]$ denotes the character on the first $n$-components and $\kappa_2:(\Z_p^{\times}/\{\pm 1\})\to R[[T_0]]$ the one on the last one. Moreover, composing the canonical projections $\Z_p^{\times}\to\Z_p^{\times}/\{\pm 1\}$ with the map
\begin{align*}
&(\Z_p^{\times}/\{\pm 1\})^{n} \times (\Z_p^{\times}/\{\pm 1\})^{n} \to T_0 \\
&((x_{\p})_{\p},(y_{\p})_{\p}) \mapsto ((x_{\p}y_{\p}^{-1/2})_{\p},\prod_{\p}y_{\p}^{1/2}).
\end{align*}
we obtain a map $R[[T^2]]\to \Lambda$. We can therefore define 
\begin{equation}\label{eq:MIgusasheaf2}
\Omega^{(\kappa_1,\kappa_2)}=\Omega^{(\kappa,\kappa')}\otimes_{R[[T^2]]}\Lambda. 
\end{equation}
If $\underline{k}\in \Z^{n}, w\in \Z$,
the algebra homomorphism $(2\underline{k},\underline{n}): R[[T^2]]\to R$ induced by the character $x_{\p}\mapsto x_{\p}^{2k_{\p}},x_{\p}\mapsto x_{\p}^{n_{\p}}$ with $n_\p=w-k_{\p}$ factors through a morphism $(2\underline{k},2w): \Lambda\to R$, since the character factors through the character of $T_0$ given by $((x_\p)_\p, y)\mapsto y^{2w}\cdot\prod x_\p^{2k_\p}$.\\
We have a diagonal embedding of $\Oo_{F,+}^{\times}\subset (\Oo_{F,(p)})_{+}^{\times}$ in $T\times T$, given by sending $x$ to $((x_{\p})_{\p},(x_{\p})_{\p})$, where $x_{\p}\in \Z_p^{\times}$ is the image of $x$ in the completion of $F$ at $\p$. Now notice that on $T_0$, the map $x\mapsto x^2$ is bijective. We denote by $x^{1/2}$ the preimage of $x\in T_0$ under this map. Finally we let $d$ be the following character
\begin{align*}
d: &\Oo_{F,+}^{\times} \to (R[[T_0^2]])^{\times}\\ &x\to ({((x_{\p})_{\p})_{0}}^{1/2},((x_{\p})_{\p})_0).
\end{align*} 
We define an action of $x\in \Oo_{F,+}^{\times}$ on $\Omega^{(\kappa,\kappa')}$ by
\[
x^*\Omega^{(\kappa,\kappa')}=\Omega^{(\kappa,\kappa')} \to \Omega^{(\kappa,\kappa')},
\]
where first map is the tautological isomorphism (being the construction of $\Omega^{(\kappa,\kappa')}$ independent on the polarisation) and the second one is multiplication by $d(x)$. 
For an algebraic character $(2\underline{k},2w)$ as above, we obtain 
\begin{equation}\label{eq:specialiso}
\Omega^{(\kappa_1,\kappa_2)}\otimes_{\Lambda, (2\underline{k},2w)} R \simeq (\myom^{(2\underline{k},2w)})_{|\X^{\ord}_R}.
\end{equation}
So classical specialisations of the sheaf $\Omega^{(\kappa_1,\kappa_2)}$ recovers automorphic sheaves of even weights\footnote{Note that $2k_i\equiv 2w$ mod $2$ for every $i=1,\dots, n$ and $\tfrac{2w-2k_i}{2}=w-k_i$}.

If we want to get odd weights, it is enough to twist $\Omega^{(\kappa_1,\kappa_2)}$ as follows. Consider the sheaf $\myom^{1}= \otimes_{\tau}\omega_{\tau}$ and let
\begin{equation}\label{eq:twistIgusasheaf}
\Omega^{(\kappa_1+1,\kappa_2)}=\Omega^{(\kappa_1,\kappa_2)} \otimes_{\Oo_{\X^{\ord}}} \myom^{1}.
\end{equation}
This way the isomorphism \eqref{eq:specialiso} induces an isomorphism 
$$\Omega^{(\kappa_1+1,\kappa_2)}\otimes_{\Lambda, (2\underline{k},w)} R \simeq {\myom^{(2\underline{k}+1,2w+1)}}_{|\X^{\ord}}\footnote{Note again that $2k_i+1\equiv 2w+1$ mod $2$ for every $i=1,\dots, n$ and $\tfrac{2w+1-(2k_i+1)}{2}=w-k_i$.}$$ for all $(\underline{k},w)\in \Z^{n+1}$.

\begin{rmk}
For clarity, we write explicitly the action of the units on sections of the sheaf $\Omega^{(\kappa_1,\kappa_2)}$. They are functions $f$ on $((A,\iota, \lambda, \alpha ),\phi:\Z_p\otimes \Oo_F\simeq T_p(A)^{\et}, \psi_1:\Z_p\otimes \Oo_F\simeq T_p(A)^{\et},\psi_2:T_p(A^{\vee}_x)^{\circ}\simeq \mu_{p^{\infty}}\otimes \Oo_F)$, such that for any $\lambda=(\lambda_\p)_\p, t=(t_\p)_\p, z=(z_\p)_\p \in (\Z_p^{\times})^n$, they satisfy
$$f(A,\iota, \lambda, \alpha, \phi\circ \lambda^{-1}, (\psi_1,\psi_2) \circ \SmallMatrix{zt &0\\0&z^{-1}}^{-1})=\kappa_1((\lambda_{0,\p}t_{0,\p}^{-1/2})_{\p}){\kappa_2}(\prod t_{\p}^{1/2}) f(A,\iota, \lambda, \alpha, \phi, \psi_1,\psi_2).$$
Additionally they are invariant by the action of the units $\epsilon$, i.e. $\epsilon\cdot f(A,\iota, \lambda, \alpha, \phi, \psi_1,\psi_2)=f(A,\iota, \lambda, \alpha, \phi, \psi_1,\psi_2)$, where the action of the units is defined by 
$$\epsilon\cdot f(A,\iota, \lambda, \alpha, \phi, \psi_1,\psi_2)=f(A,\iota, \epsilon\cdot\lambda, \alpha, \phi, (\psi_1,\psi_2)).$$
The fact that there exist non-trivial global sections of this sheaf follows from the fact that $\epsilon^*((A,\iota, \lambda, \alpha, \phi, \psi_1,\psi_2))$ is equal to $(A,\iota, \epsilon^2\cdot\lambda, \alpha, \phi\circ \epsilon, (\psi_1,\psi_2)\circ \SmallMatrix{\epsilon &0\\0&\epsilon})$ and $ \kappa_1((\epsilon_{\p}\epsilon_{\p}^{-1})_{\p})\cdot {\kappa_2}(\prod\epsilon_{\p})=1$. 
\end{rmk}
\begin{rmk}\label{rmkhida}
This is a good point where we can briefly draw a comparison with \cite{hidaannals, hidaalgebra}. The construction in \emph{op. cit.} is of course very different, since the author works with Hecke algebras and on quaternionic Shimura varieties (via the Jacquet--Langlands correspondence). The reader may however be confused by the discrepancy on the definition of weights and universal characters of $\Lambda$. Hida considers the map $T^2\to T\times  \Z_p^{\times}$ given by $((x_\p)_\p,(y_\p)_\p)\mapsto((y_\p x_\p^{-1})_\p, \prod x_\p)$. The character $T^2 \to \Z_p^{\times}$, $x_{\p}\mapsto x_{\p}^{k_{\p}},y_{\p}\mapsto x_{\p}^{v_{\p}}$ with $v_\p=\tfrac{w-k_{\p}}{2}$ factors through the character $(\underline{v},w)$ of $T\times  \Z_p^{\times}$ given by $((a_\p)_\p ,z)\mapsto z^w\prod_\p a_\p^{v_{\p}}$. In particular these characters are trivial on the units of $\Oo_F$ embedded diagonally in $T\times T$, whereas the ones we considered above are trivial on the units of $\Oo_F$ embedded in $T\times T$ via $\epsilon\to (\epsilon,\epsilon^2)$. This turns out to be the correct thing to do with our construction in light of the previous remark.
\end{rmk}

From now on we will only work with the sheaf $\Omega^{(\kappa_1,\kappa_2)}$ interpolating even weights. All the constructions can be carried out taking into account the twist in definition \eqref{eq:twistIgusasheaf} to obtain analogous results for odd weights, but, for simplicity, we don't do this explicitly here. 

\subsubsection{$U_\p$ and Frobenius operators} Being the cohomology of the sheaves we have just defined \emph{too big}, we define some operators, whose associated idempotents will cut out a smaller part of the cohomology for which we can prove the classicality result. 

We define the partial Frobenius $F_{\p}:(\M^{\cc})^{\ord}\to (\M^{\cc'})^{\ord}$ to be the morphism sending $(A,\iota,\lambda, \alpha_{K^p})\mapsto (A/H_{\p},\iota',\lambda', \alpha'_{K^p})$, where $H_{\p}\subset A[\p]$ is the multiplicative subgroup of the $\p$-torsion of $A$, $\iota',\alpha'_{K^p}$ are defined by the composition of $\iota, \alpha_{K^p}$ with the isogeny $\pi_{\p}:A\to A/H_{\p}$ and $\lambda'=\theta_{\cc}\circ\tilde{\lambda}$, where $\tilde{\lambda}$ is a $\cc\p$-polarisation of $A/H_{\p}$ determined by the commutative diagram
\begin{equation}\notag
\begin{tikzcd}
A/H_{\p}\otimes_{\Oo_F}\cc\p \arrow{d}{\tilde{\lambda}} \arrow{r}{\tilde{\pi}_{\p}} &A \otimes_{\Oo_F}\cc\arrow{d}{\lambda}\\
(A/H_{\p})^{\vee} \arrow{r}{\pi_{\p}^{\vee}} &A^{\vee}
\end{tikzcd}
\end{equation} 
where $\tilde{\pi}_{\p}$ is the unique map such that the composition 
$A\otimes \cc\p \xrightarrow{\pi_{\p}}  A/H_{\p}\otimes_{\Oo_F}\cc\p \xrightarrow{\tilde{\pi}_{\p}} A \otimes_{\Oo_F}\cc$ 
is the canonical map with kernel $A[\p]\otimes \p$ and $\theta_{\cc}:\cc' \to \cc\p$ is an isomorphism as in the definition of $p_2$ in \eqref{eq:p2}, unique up to an element of $\Oo_{F,+}^{\times}$. We therefore have an isogeny $\pi_{\p}^D$ defined by the commutative diagram
\begin{equation}\label{dualiso}
\begin{tikzcd}
A/H_{\p}\otimes_{\Oo_F}\cc' \arrow{d}{\lambda'} \arrow{r}{\pi_{\p}^D} &A \otimes_{\Oo_F}\cc\arrow{d}{\lambda}\\
(A/H_{\p})^{\vee} \arrow{r}{\pi_{\p}^{\vee}} &A^{\vee}
\end{tikzcd}
\end{equation}
Hence $F_{\p}$ is well defined up to $\Oo_{F,+}^{\times}$ and it is equivariant by the action of $\Delta(K)$. We therefore obtain a well defined morphism
\[
F_{\p}: \X^{\ord}\to\X^{\ord}.
\]
The following result follows from \cite[Lemma 3.14]{tianxiao}.
\begin{lemma}\label{lemmatrace}
The trace map $\tr_{F_{\p}}:(F_{\p})_{\star}\Oo_{\X^{\ord}}\to \Oo_{\X^{\ord}}$ satisfies $\tr_{F_{\p}}((F_{\p})_{\star}\Oo_{\X^{\ord}}) \subset p\Oo_{\X^{\ord}}$.
\end{lemma}
The morphism $F_{\p}:(\M^{\cc})^{\ord}\to (\M^{\cc'})^{\ord}$ extends to a morphism between the partial Igusa towers $\mathfrak{IG}_{\cc}, \mathfrak{IG}_{\cc'}$ over $(\M^{\cc})^{\ord}$ and $ (\M^{\cc'})^{\ord}$ given by $(A,\iota,\lambda, \alpha_{K^p}, \varphi:\Z_p\simeq T_{\p}(A)^{\et})\mapsto (A/H_{\p},\iota',\lambda', \alpha'_{K^p}, \varphi':\Z_p\simeq T_{\p}(A/H_{\p})^{\et})$, where $\varphi':\Z_p\xrightarrow{\varphi}T_{\p}(A)^{\et}\simeq T_{\p}(A/H_{\p})^{\et}$ and the last isomorphism is induced by the isogeny $\pi_{\p}$. We also have a morphism between the partial igusa towers $\mathfrak{IG}^{\vee}_{\cc}, \mathfrak{IG}^{\vee}_{\cc'}$ obtained as follows: the dual isogeny $\pi_{\p}^{\vee}: (A/H_{\p})^{\vee} \to A^{\vee}$ induces an isomorphism $T_{\p}((A/H_{\p})^{\vee})^{\circ} \to T_{\p}(A^{\vee})^{\circ}$. Composing this with the rigidification $\varphi: T_{\p}(A^{\vee})^{\circ}\simeq \mu_{p^{\infty}}$, we obtain an isomorphism $\varphi\circ \pi_{\p}^{\vee}: T_{\p}((A/H_{\p})^{\vee})^{\circ}\to \mu_{p^{\infty}}$. So the morphism induced by $F_{\p}$ on the \emph{dual} Igusa towers is given by $(A,\iota,\lambda, \alpha_{K^p}, \varphi:\Z_p\simeq T_{\p}(A^{\vee})^{\et})\mapsto (A/H_{\p},\iota',\lambda', \alpha'_{K^p}, \varphi\circ\pi_{\p}^{\vee}:\Z_p\simeq T_{\p}((A/H_{\p})^{\vee})^{\et})$. With these constructions in mind, we prove the following.
\begin{lemma}
$F_{\p}$ induces two well defined maps $$F_{\p}:(F_{\p})^{\star}\OG\to \OG \text{  and  }U_{\p}: (F_{\p})_{\star}\OG\to \OG.$$
\end{lemma}
\begin{proof}
We temporarily denote by $\mathfrak{IG}$ the partial Igusa tower (parametrising rigidifications for the étale $\p$-adic Tate module) and we denote with a subscript the tower above the component $\M^{\cc,\ord}$. First we notice that the canonical map $F_{\p}\times \pi_{\mathfrak{IG}_{\cc}}: \mathfrak{IG}_{\cc}\to \mathfrak{IG}_{\cc'} \times_{\M^{\cc',\ord}}\M^{\cc,\ord}$ obtained by the following diagram
\[
\begin{tikzcd}[scale=0.5]
\mathfrak{IG}_{\cc}\arrow[bend left]{drr}{\pi_{\mathfrak{IG}_{\cc}}}
\arrow[bend right]{ddr}[swap]{F_{\p}}
\arrow[dotted]{dr}[description]{F_{\p}\times \pi_{\mathfrak{IG}_{\cc}}} & & \\
& \mathfrak{IG}_{\cc'} \times_{\M^{\cc',\ord}}\M^{\cc,\ord} \arrow{r}{} \arrow{d}{} &\M^{\cc,\ord}\arrow{d}{F_{\p}} \\
& \mathfrak{IG}_{\cc'}  \arrow{r}{\pi_{\mathfrak{IG}_{\cc'}}} & \M^{\cc',\ord}
\end{tikzcd}
\]
is an isomorphism. The inverse morphism is given by sending $$[ (A/H_{\p},\iota',\lambda', \alpha'_{K^p},\varphi:\Z_p\simeq T_{\p}(A/H_{\p})^{\et}),(A,\iota,\lambda, \alpha_{K^p}) ] \mapsto (A,\iota,\lambda, \alpha_{K^p},\tilde{\pi}_{\p}^{-1}\circ\varphi),$$ where $\tilde{\pi}_{\p}:T_{\p}(A)^{\et}\simeq T_{\p}(A/H_{\p})^{\et}$ is the isomorphism induced by the isogeny $\pi_{\p}$.
We apply Lemma \ref{lemmatrace} and the fact that the map $F_{\p}\times \pi_{\mathfrak{IG}_{\cc}}$ is an isomorphism to obtain
$\tr_{F_{\p}}((F_{\p})_{\star}\Oo_{\mathfrak{IG}_{\cc}})\subset p \Oo_{\mathfrak{IG}_{\cc'}}$ and
$$\tr_{F_{\p}}((F_{\p})_{\star}(\pi_{\mathfrak{IG}_{\cc}})_{\star}\Oo_{\mathfrak{IG}_{\cc}})=(\pi_{\mathfrak{IG}_{\cc'}})_{\star}(\tr_{F_{\p}}((F_{\p})_{\star}\Oo_{\mathfrak{IG}_{\cc}}))\subset p (\pi_{\mathfrak{IG}_{\cc'}})_{\star} \Oo_{\mathfrak{IG}_{\cc'}}.$$
Patching together the corresponding maps, we have the natural pullback map $F_{\p}:(F_{\p})^{\star}\pi_{\star}\Oo_{\mathfrak{IG}}\to \pi_{\star}\Oo_{\mathfrak{IG}}$ and $\tfrac{1}{p}\tr_{F_{\p}}:(F_{\p})_{\star}\pi_{\star}\Oo_{\mathfrak{IG}}\to \pi_{\star}\Oo_{\mathfrak{IG}}$, which are well defined up to units, equivariant by the action of $\Delta(K)$ and $(\Z_p)^{\times}$-invariant. 

Similarly, if we still denote by $F_{\p}$ the map induced by the partial Frobenius on the dual Igusa towers $F_{\p}:\mathfrak{IG}^{\vee}_{\cc}\to \mathfrak{IG}^{\vee}_{\cc'}$, we have again that the canonical map $F_{\p}\times \pi_{\mathfrak{IG}^{\vee}_{\cc}}: \mathfrak{IG}^{\vee}_{\cc}\to \mathfrak{IG}^{\vee}_{\cc'} \times_{\M^{\cc',\ord}}\M^{\cc,\ord}$ is an isomorphism. Proceeding as above, we obtain the pullback map $F_{\p}:(F_{\p})^{\star}(\pi_{\mathfrak{IG}\times \mathfrak{IG}^{\vee}})_{\star}\Oo_{\mathfrak{IG}\times \mathfrak{IG}^{\vee}}\to (\pi_{\mathfrak{IG}\times \mathfrak{IG}^{\vee}})_{\star}\Oo_{\mathfrak{IG}\times \mathfrak{IG}^{\vee}}$ and $\tfrac{1}{p}\tr_{F_{\p}}:(F_{\p})_{\star}(\pi_{\mathfrak{IG}\times \mathfrak{IG}^{\vee}})_{\star}\Oo_{\mathfrak{IG}\times \mathfrak{IG}^{\vee}}\to (\pi_{\mathfrak{IG}\times \mathfrak{IG}^{\vee}})_{\star}\Oo_{\mathfrak{IG}\times \mathfrak{IG}^{\vee}}$, which are well defined up to units, equivariant by the action of $\Delta(K)$ and $(\Z_p)^{\times}$-invariant. 

We can then consider the full Igusa towers, where the partial Frobenius on $\M^{\ord}$ lifts to $F_{\p}$ on $\mathfrak{IG},\mathfrak{IG}^{\vee}$ and to isomorphisms on all the factors of the Igusa towers for $\p'\neq \p$. Tensoring with $R[[T^2]]$ and taking $T^2$-invariants we obtain maps of sheaves over $\M^{\ord}$
$$F_{\p}:(F_{\p})^{\star}\Omega^{(\kappa,\kappa')}\to \Omega^{(\kappa,\kappa')} \text{  and  }U_{\p}:=\tfrac{1}{p}\tr_{F_{\p}}:(F_{\p})_{\star}\Omega^{(\kappa,\kappa')}\to \Omega^{(\kappa,\kappa')}$$
again well defined up to units and equivariant by the action of $\Delta(K)$. We therefore obtain maps $F_{\p}, U_{\p}$ for the shaves over $\X^{\ord}$ as claimed.
\end{proof}

We now look more closely to the specialisation of these maps at classical weight $\underline{k},\underline{n}\in \Z^n$, under the isomorphism \eqref{eq:specialiso}. 

Consider the universal isogeny $\pi_{\p}:\mathcal{A^{\cc}}\to\mathcal{A^{\cc'}}/H_{\p}$. The pullback gives an $\Oo_F$-equivariant map $\pi_{\p}^{\star}: (F_{\p})^{\star}\myom \to \myom$ of sheaves over $\M^{\ord}$. Similarly we obtain a map $\pi_{\p}^{\star}: (F_{\p})^{\star}(\wedge^2\hh^1) \to \wedge^2\hh^1$ and we therefore obtain, for any $\underline{k}\in \Z^n,w\in \Z$, a map
\begin{equation}\label{eqdiffFp}
\pi_{\p}^{\star}: (F_{\p})^{\star}\myom^{(2\underline{k},2w)} \to \myom^{(2\underline{k},2w)},
\end{equation}
The dual isogeny induces an isogeny $\pi_{\p}^D: \mathcal{A^{\cc'}}/H_{\p}\otimes \cc'\to \mathcal{A^{\cc}}\otimes \cc$ using the prime to $p$ polaristaions. We therefore find, being $\cc,\cc'$ all coprime to $p$, an $\Oo_F$-equivariant map $(\pi_{\p}^{D})^{\star}: \myom^{(2\underline{k},2w)} \to (F_{\p})^{\star}\myom^{(2\underline{k},2w)}$. We can then construct a map
\begin{equation}\label{eqdiffUp}
U_{\pi_{\p}^D}:(F_{\p})_{\star}\myom^{(2\underline{k},2w)} \xrightarrow{(\pi_{\p}^{D})^{\star}} (F_{\p})_{\star}(F_{\p})^{\star}\myom^{(2\underline{k},2w)}\xrightarrow{\tfrac{1}{p}\tr_{F_{\p}}}\myom^{(2\underline{k},2w))}.
\end{equation}
\begin{lemma}\label{lemmaspecFU}
Let $(\underline{k},w)\in \Z^{n+1}$ and let $n_{\p}=w-k_{\p}$. The map $F_{\p}$ specialised to weight $(\underline{k},w)$ coincides with $p^{-2k_{\p}-n_{\p}}(\pi_{\p}^{\star})^{(\underline{k},\underline{n})}$. The map $U_{\p}$ specialised to weight $(\underline{k},w)$ coincides with $p^{-n_{\p}}U_{\pi_{\p}^D}$.
\end{lemma}
\begin{proof} Let $m\leq n$ and recall $\mathfrak{IG}_{\p}^{m,n}=\text{Isom}_{\mathcal{M}^{\ord}_m}(\Z/p^n, \mathcal{A}[\p^n]^{\et})$, $\mathcal{T}_{\p} =\text{Isom}_{\mathcal{M}_m}(\Oo_{\mathcal{M}_m}, (e^*\Omega^1_{\mathcal{A}})_{\p})$. where $\mathcal{T}_{\p}$ is the $\p$-component of the torsor $\mathcal{T}$. We have the following commutative diagram
\[
\begin{tikzcd}[scale=0.5, every node/.style={scale=0.5}]
\mathfrak{IG}_{\p}^{m,n}\arrow{r}{F_{\p}}\arrow{d}{\HT_{m,n}} &\mathfrak{IG}_{\p}^{m,n}\arrow{d}{\HT_{m,n}}\\
\mathcal{T}_{\p|\mathcal{M}^{\ord}_m}\arrow{r}{(\pi_{\p}^D)^{\star}} &\mathcal{T}_{\p|\mathcal{M}^{\ord}_m}
\end{tikzcd}
\begin{tikzcd}[scale=0.5, every arrow/.style={draw,mapsto}]
(A,\iota,\lambda, \alpha_{K^p}, \varphi_n) \arrow{r}{F_{\p}}\arrow{d}{\HT_{m,n}} &(A/H_{\p},\iota',\lambda', \alpha'_{K^p}, {\pi}_{\p}\circ\varphi_n)\arrow{d}{\HT_{m,n}}\\
(A,\iota,\lambda, \alpha_{K^p}, (\varphi_n^D)^{*}) \arrow{r}{(\pi_{\p}^D)^{\star}} &(A/H_{\p},\iota',\lambda', \alpha'_{K^p}, (\pi_{\p}^D)^{\star}\circ(\varphi_n^D)^{\star}))
\end{tikzcd}
\]
where $(\pi_{\p}^D)^{\star}$ is the map induced by the ($\p$-component) of the pullbacks $\Omega^1_A\to\Omega^1_{A/H_\p}$ of the dual maps defined by the diagram \eqref{dualiso}; note that we use the fact that $\cc,\cc'$ are coprime to $\p$. Moreover, to simplify the notation, we denoted by $(\varphi_n^D)^*$ the image of $\varphi_n:\Z/p^n\simeq A[\p^n]^{\et}$ under the map defined in \eqref{HTmn} (obtained by fixing a basis of $\Omega^1_{\mu_{p^n}}$).
Similarly, recall $(\mathfrak{IG}^{\vee}_{\p})^{m,n}=\text{Isom}_{\mathcal{M}^{\ord}_m}(\mathcal{A}^{\vee}[\p^n]^{\circ},\mu_{p^n})$, $\mathcal{T}'_{\p}=\text{Isom}_{\mathcal{M}_m}(\Oo_{\mathcal{M}_m}, ((e^*\Omega^1_{\mathcal{A}^{\vee}})_{\p})^{\vee})$. We have the commutative diagram
\[
\begin{tikzcd}[scale=0.5]
(\mathfrak{IG}^{\vee}_{\p})^{m,n} \arrow{r}{F_{\p}}\arrow{d}{\HT_{m,n}} &(\mathfrak{IG}^{\vee}_{\p})^{m,n}\arrow{d}{\HT_{m,n}}\\
\mathcal{T}'_{\p|\mathcal{M}^{\ord}_m} \arrow{r}{((\pi_{\p}^{\vee})^{\star})^{-1}} &\mathcal{T}'_{\p|\mathcal{M}^{\ord}_m}
\end{tikzcd}
\begin{tikzcd}[scale=0.5, every arrow/.style={draw,mapsto}]
(A,\iota,\lambda, \alpha_{K^p}, \varphi_n) \arrow{r}{F_{\p}}\arrow{d}{\HT_{m,n}} &(A/H_{\p},\iota',\lambda', \alpha'_{K^p}, \varphi_n\circ \pi_{\p}^{\vee})\arrow{d}{\HT_{m,n}}\\
(A,\iota,\lambda, \alpha_{K^p}, (\varphi_n^*)^{-1}) \arrow{r}{((\pi_{\p}^{\vee})^{\star})^{-1}} &(A/H_{\p},\iota',\lambda', \alpha'_{K^p}, ((\varphi_n\circ \pi_{\p}^{\vee})^*)^{-1})
\end{tikzcd}
\]
where $\varphi_n^*$ denotes the isomorphism $\Oo_{\mathcal{M}^{\ord}_m}\to (e^*\Omega^1_{A^{\vee}})_{\p}$ and $(\varphi_n^*)^{-1}$ denotes the isomorphism $\Oo_{\mathcal{M}^{\ord}_m}\to ((e^*\Omega^1_{A^{\vee}})_{\p})^{\vee}$ obtained as in \eqref{HTmndual}. Moreover the bottom map is given by pre-composing with $((\pi_{\p}^{\vee})^{\star})^{-1}:((e^*\Omega^1_{A^{\vee}})_{\p})^{\vee}\to((e^*\Omega^1_{(A/H_{\p})^{\vee}})_{\p})^{\vee}$. Passing to the limits, we obtain commutative diagrams 
\begin{equation}\label{diagrams}
\begin{tikzcd}
\mathfrak{IG} \arrow{r}{F_{\p}}\arrow{d}{\HT} &\mathfrak{IG}\arrow{d}{\HT}\\
\mathcal{T}_{|\mathcal{M}^{\ord}} \arrow{r}{(\pi_{\p}^{D})^{\star}} &\mathcal{T}_{|\mathcal{M}^{\ord}}
\end{tikzcd}
\begin{tikzcd}
\mathfrak{IG}^{\vee} \arrow{r}{F_{\p}}\arrow{d}{\HT} &\mathfrak{IG}^{\vee}_{\p}\arrow{d}{\HT}\\
\mathcal{T}'_{|\mathcal{M}^{\ord}} \arrow{r}{((\pi_{\p}^{\vee})^{\star})^{-1}} &\mathcal{T}'_{|\mathcal{M}^{\ord}},
\end{tikzcd}
\end{equation}
where we replaced the partial Igusa towers and partial $\mathbb{G}_m$-torsors with the full Igusa towers $\mathfrak{IG},\mathfrak{IG}^{\vee}$ and the $\mathbb{G}_m\otimes \Oo_F$-torsors $\mathcal{T},\mathcal{T}'$, letting $F_{\p}$ act as the identity on the components $\mathfrak{IG}_{\mathfrak{q}},\mathfrak{IG}_{\mathfrak{q}}^{\vee}$ for $\mathfrak{q}\neq \p$ and exploiting the fact that being $\pi_{\p}$ a $\p$-isogeny, the bottom pullback maps are isomorphisms on the components $\mathcal{T}_{\mathfrak{q}},\mathcal{T}'_{\mathfrak{q}}$ for ${\mathfrak{q}}\neq \p$. We find the corresponding commutative diagrams on the structural shaves and taking the $(2\underline{k},\underline{n})$-components we obtain that the pullback map $F_{\p}: (F_{\p})^{\star}\myom^{(\underline{k},\underline{n})} \to \myom^{(\underline{k},\underline{n})}$ is therefore given on the $\p$-component by $((\pi_{\p}^D)^{\star})^{-1} \otimes [((\pi_{\p}^D)^{\star})^{-1}\otimes(\pi_{\p}^{\vee})^{\star}]$, where we are identifying, by (\ref{hodgeexseq}) and the fact that the polarizations are prime to $\p$, $\wedge^2\mathcal{H}^1_{\dr}(\mathcal{A}/H_{\p})_{\p}\simeq (e^*\Omega^1_{\mathcal{A}})_{\p}\otimes (e^*\Omega^1_{\mathcal{A}^{\vee}})_{\p}^{\vee}$. Under this identification the natural pullback map $\pi_{\p}^{\star}$ decomposes as
\[
\begin{tikzcd}[column sep=5em, row sep=1em]
\wedge^2\mathcal{H}^1_{\dr}(\mathcal{A}/H_{\p})_{\p} \arrow{r}{\pi_{\p}^{\star}}\arrow{d}{\simeq}&\wedge^2\mathcal{H}^1_{\dr}(\mathcal{A})_{\p}\arrow{d}{\simeq}\\
(e^*\Omega^1_{\mathcal{A}/H_{\p}})_{\p}\otimes \operatorname{Lie}((\mathcal{A}/H_{\p})^{\vee})_{\p}\arrow{r}{\pi_{\p}^{\star}\otimes \operatorname{Lie}(\pi_{\p}^{\vee})} &(e^*\Omega^1_{\mathcal{A}})_{\p}\otimes \operatorname{Lie}(\mathcal{A}^{\vee})_{\p}
\end{tikzcd}
\]
and, under the isomorphism \cite[(1.0.13)]{katzpadic}, we can write the second component of the map as $\operatorname{Lie}(\pi_{\p}^{\vee})=(\pi_{\p}^{\vee})^{\star}:(e^*\Omega^1_{(\mathcal{A}/H_{\p})^{\vee}})^{\vee}\simeq \operatorname{Lie}((\mathcal{A}/H_{\p})^{\vee})\to (e^*\Omega^1_{\mathcal{A}^{\vee}})^{\vee}\simeq \operatorname{Lie}(\mathcal{A}^{\vee})$, where here $(\pi_{\p}^{\vee})^{\star}$ denotes the map obtained composing with $(\pi_{\p}^{\vee})^{\star}: e^*\Omega^1_{\mathcal{A}^{\vee}}\to e^*\Omega^1_{(\mathcal{A}/H_{\p})^{\vee}}$. So, since the composition $\pi_{\p}^{\star}\circ (\pi_{\p}^D)^{\star}$ is given by multiplication by $p$ on $(e^*\Omega^1_{\mathcal{A}})_{\p}$ and is an isomorphism on the $\mathfrak{q}\neq \p$-components, we obtain

$$F_{\p}=((\pi_{\p}^D)^{\star})^{-1} \otimes [((\pi_{\p}^D)^{\star})^{-1}\otimes (\pi_{\p}^{\vee})^{\star}]=p^{-2k_{\p}-n_{\p}}\pi_{\p}^{\star}.$$

On the other hand, again by the commutativity of the diagrams \eqref{diagrams}, the specialisation of the trace map is given by the composition
$(F_{\p})_{\star}\myom^{(2\underline{k},\underline{n})} \to (F_{\p})_{\star}(F_{\p})^{\star}\myom^{(2\underline{k},\underline{n})}\xrightarrow{1/p\tr_{F_{\p}}}\myom^{(2\underline{k},\underline{n})},$
where the first map is given on the $\p$-component by $(\pi_{\p}^D)^{\star} \otimes [(\pi_{\p}^D)^{\star}\otimes((\pi_{\p}^{\vee})^{\star})^{-1}]$. As before the natural pullback map $(\pi_{\p}^D)^{\star}$ decomposes as
\[
\begin{tikzcd}[column sep=5em, row sep=1em]
\wedge^2\mathcal{H}^1_{\dr}(\mathcal{A})_{\p} \arrow{r}{(\pi_{\p}^D)^{\star}}\arrow{d}{\simeq}&\wedge^2\mathcal{H}^1_{\dr}(\mathcal{A}/H_{\p})_{\p}\arrow{d}{\simeq}\\
(e^*\Omega^1_{\mathcal{A}})_{\p}\otimes \operatorname{Lie}(\mathcal{A}^{\vee})_{\p}\arrow{r}{(\pi_{\p}^D)^{\star}\otimes \operatorname{Lie}((\pi_{\p}^D)^{\vee})} &(e^*\Omega^1_{\mathcal{A}/H_{\p}})_{\p}\otimes \operatorname{Lie}((\mathcal{A}/H_{\p})^{\vee})_{\p},
\end{tikzcd}
\]
where $\operatorname{Lie}((\pi_{\p}^D)^{\vee})=((\pi_{\p}^D)^{\vee})^{\star}:(e^*\Omega^1_{\mathcal{A}^{\vee}})^{\vee}\simeq \operatorname{Lie}(\mathcal{A}^{\vee})\to (e^*\Omega^1_{(\mathcal{A}/H_{\p})^{\vee}})^{\vee}\simeq \operatorname{Lie}((\mathcal{A}/H_{\p})^{\vee})$. The composition $(\pi_{\p}^{\vee})^{\star}\circ ((\pi_{\p}^D)^{\vee})^{\star}$ is given by multiplication by $p$ on $(e^*\Omega^1_{\mathcal{A}^{\vee}})^{\vee}$, so we obtain
$$U_{\p}=\tfrac{1}{p}\tr_{F_{\p}}\circ(\pi_{\p}^D)^{\star} \otimes [(\pi_{\p}^D)^{\star}\otimes((\pi_{\p}^{\vee})^{\star})^{-1}]=p^{-n_{\p}}U_{\pi_{\p}^D}.$$
These maps are all invariant by the action of the units, so we obtain the desired statement for the map between the sheaves over $\X^{\ord}$.
\end{proof}

We now compare these operators with the operator $T_{\p}$ constructed in $\S$ \ref{secHecke}. As we did for $\mathcal{M}$ and $X$, the moduli space and Shimura variety of level $K$ such that $K_p=G(\Z_p)$, we can consider $\M_0(\p)$ and $\mathfrak{X}_0(\p)$ the formal completion of $\mathcal{M}_0(\p)$ and of $X_0(\p)$ along their special fibres. The ordinary locus $\M_0(\p)^{\ord}$ of $\M_0(\p)$ is the disjoint union of the loci $\M_0(\p)^{\et}$ and $\M_0(\p)^{m}$, where the universal $\p$-isogeny has respectively étale and multiplicative kernel. Passing to the quotient by the action of the units, we similarly let $\X_0(\p)^{\et}, \X_0(\p)^m$, noting that the kernel of the isogeny is independent on the polarisations. By construction we have that the two projection maps $p_1,p_2$ are either an isomorphism or can be identified with $F_{\p}$ when restricted to $\X_0(\p)^{\et}, \X_0(\p)^m$, more precisely
\begin{equation}\label{diagramsp1p2}
\begin{tikzcd}[column sep=1.1em, row sep=1em]
&\X_0(\p)^{\et}\arrow{dr}{p_2}[swap]{\simeq} \arrow{dl}[left]{p_1=F_{\p}}\\
\X^{\ord} &&\X^{\ord}
\end{tikzcd}
\begin{tikzcd}[column sep=1.1em, row sep=1em]
&\X_0(\p)^{m}\arrow{dr}[right]{p_2=F_{\p}} \arrow{dl}{\simeq}[swap]{p_1}\\
\X^{\ord} &&\X^{\ord}.
\end{tikzcd}
\end{equation}
Using the previous lemma, we can prove the following
\begin{lemma}\label{lemmareltp}
Let $F_{\p}$ and $U_{\p}$ be the specialisations of the operators $F_{\p}$ and $U_{\p}$ in weight $(\underline{k},w)$. We have the following equalities
\begin{itemize}
\item[(i)] $T_{\p}=p^{2k_{\p}-1}F_{\p}+U_{\p}$ when $2k_{\p}\geq 1$;
\item[(ii)] $T_{\p}=F_{\p}+p^{1-2k_{\p}}U_{\p}$ when $2k_{\p}< 1$.
\end{itemize}
\end{lemma}
\begin{proof}
We denote by $T_{\p}^{naive,\et}, T_{\p}^{naive,m}$ the projection of the restriction of the correspondence $T_{\p}^{naive}$ over $\X_0(\p)^{\ord}$ on the component $\X_0(\p)^{\et}$ and $\X_0(\p)^m$ respectively. They are obtained via the pullback induced by the isogeny $\pi_{\p}^D$ and $\pi_{\p}$ respectively. By the above observations and Lemma \ref{lemmaspecFU}, we have that $F_{\p}=p^{-2k_{\p}-n_{\p}}T_{\p}^{naive,m}$ and $U_{\p}=p^{-n_{\p}-1}T_{\p}^{naive,\et}$. Our definition of the normalised operator $T_{\p}$ gives
\begin{displaymath}
T_{\p}=\begin{cases}
p^{-n_{\p}-1}(p^{2k_{\p}+n_{\p}}F_{\p}+p^{n_{\p}+1}U_{\p})=p^{2k_{\p}-1}F_{\p}+U_{\p} &\text{if }2k_{\p}\geq 1 \\
p^{-n_{\p}-2k_{\p}}(p^{2k_{\p}+n_{\p}} F_{\p}+p^{n_{\p}+1}U_{\p})=F_{\p}+p^{1-2k_{\p}}U_{\p} &\text{if }2k_{\p}< 1.
\end{cases}
\end{displaymath} 
\end{proof}
Note that, in particular, this lemma tells us that we have $T_{\p}\equiv U_{\p}\mod \wp$ when $2k_{\p}> 1$ and $T_{\p}\equiv F_{\p}\mod \wp$ when $2k_{\p}< 1$. This will be crucial in the next section to reduce the classicality result in characteristic zero to the classicality result we proved modulo $\wp$ in $\S$\ref{modpsection}. 

Finally, we have the following analogue of Proposition \ref{dualtp}.
\begin{lemma}\label{lemmaFpUpdual}
For any $J\subset \Sigma_{\infty}$, let $T_J= \prod_{\p\not\in J}U_{\p}\prod_{\p\in J}F_{\p}$. Then we have $D(T_J)=\langle p\rangle^{-1} T_{J^c}$.
\end{lemma}
\begin{proof}
Using the description of the specialisations of $F_{\p}$ and $U_{\p}$ given in Lemma \ref{lemmaspecFU} we can obtain the claimed result similarly as in the proof on Proposition \ref{dualtp} and Remark \ref{rmkallp}. Alternatively, this follows by the same proposition and the equalities of Lemma \ref{lemmareltp}.
\end{proof}
\subsubsection{Construction of $p$-adic families}
Recall that $\mathfrak{X}$ is the formal completion of $X$ along its special fibre and we denoted by $X_i=X\times_{\operatorname{Spec}R}\operatorname{Spec}(R/\wp^i)$ the reduction of $\mathfrak{X}$ modulo $\wp^i$ and by $X_i^{\ord}$ the ordinary locus of $X_i$. In order to define the desired $\Lambda$-modules, we need to extend over $\X$ suitable quotients of the sheaf $\Omega^{(\kappa_1,\kappa_2)}$, which is only defined over $\X^{\ord}$. We start by considering a general framework. 

Let $\F$ be any coherent sheaf defined over $X_i^{\ord}$. Let $\I$ a locally principal sheaf of ideals corresponding to the divisor $D_i=X_i\setminus X_i^{\ord}$, for example we could take $\I=\prod \I_{\tau}$, where $\I_{\tau}$ is a locally principal sheaf of ideals defined by a lift of the partial Hasse invariant $\tilde{h}_{\tau,i}$ (whose divisor is $D_{\tau,i}$). 

We consider an extension $\bar{\F}$ of $\F$ over $X_i$, which we can assume to be $\I$-torsion free, just by replacing $\bar{\F}$ by its quotient by the $\I$-torsion. Thanks to this, the multiplication map $\I^n \otimes_{\Oo_{X_i}} \bar{\F} \to \I^n\bar{\F}$ is an isomorpism for any $n\geq 0$ and we can make sense of the sheaves $\I^n \bar{\F}$ also for $n\leq 0$. A similar reasoning applies for $\I_\tau^n \bar{\F}$. We also have
{\begin{equation}\label{j_*}
j_{\star}\F=\ccolim_{n}\I^{-n}\bar{\F}
\end{equation}}
where $j$ denotes the inclusion $j:X_i^{\ord}\to X_i$. Let $J\subset \Sigma_{\infty}$, consider $\underline{m}=(m_{\tau})_{\tau \not\in J}\in \Z_{\geq 0}^{\#(\Sigma_{\infty}- J)}, \underline{n}=(n_{\tau})_{\tau \in J}\in \Z_{\geq 0}^{\#J}$ and let
\begin{displaymath}
\RG_i(\bar{\F})_{\underline{m},\underline{n}}=\RG(X_i, \prod_{\tau\not\in J}\I_\tau^{-m_{\tau}}\prod_{\tau\in J}\I_{\tau}^{n_{\tau}}\bar{\F}),
\end{displaymath}
\begin{displaymath}
\RG^{J}_i(\bar{\F})= \left(\limm_{n_{\tau}}\right)_{\tau\in J} \left(\colim_{m_{\tau}}\right)_{\tau\not\in J} \RG_i(\bar{\F})_{\underline{m},\underline{n}}.
\end{displaymath}
These complexes may a priori depend on the extension $\bar{\F}$ of $\F$. We now go back to the Igusa sheaf, quotient it by a certain $\Lambda$-ideal so that is a quasi coherent sheaf over $X_i^{\ord}$ and show that it has a locally finite action of the operator 
\begin{equation}\label{eqTJop}
T_J:= \prod_{\p\not\in J}U_{\p}\prod_{\p\in J}F_{\p}.
\end{equation}
We can then apply the idempotent $e(T_J)$ to $ \RG^{J}_i(\bar{\F})$ and we show that the obtained complex is independent on the choice of $\bar{\F}$.

Let ${\Lambda}=R[[(\Z_p^{\times})^{n+1}]]$. Recall that we defined the sheaf of ${\Lambda}\otimes \Oo_{\M^{\ord}}$-modules $\Omega^{(\kappa,\kappa')}$ in \eqref{eq:MIgusasheaf2}. We will now define truncated versions of this sheaf using $\mathfrak{IG}^{m,i},(\mathfrak{IG}^{\vee})^{m,i}$ the level $p^m$ Igusa towers on the reduction modulo $\wp^i$ of $\M^{\ord}$. Let ${\Lambda}_i=R/\wp^i[[((\Z/p^i\Z)^{\times})^{n+1}]]$ and let $\pi_{m,i}:\mathfrak{IG}^{m,i}\to\mathcal{M}_i^{\ord}$, $\pi'_{m,i}:\mathfrak{IG}^{m,i}\times (\mathfrak{IG}^{\vee})^{m,i}\to\mathcal{M}_i^{\ord}$. For $m\geq i$ let $(\kappa_{m,i},\kappa'_{m,i}):((\Z/p^m)^{\times})^{n+1}\to {\Lambda}_i^{\times}$ the characters similarly as in \eqref{eq:MIgusasheaf2} and factoring through $((\Z/p^i\Z)^{\times})^{n+1}$. We let 
\[
\Omega^{(\kappa_1,\kappa_2)}_{m,i}= \left(((\pi_{m,i})_{\star}\Oo_{\mathfrak{IG}^{m,i}}\otimes_{\Oo_{{\mathcal{M}}_i^{\ord}}}(\pi'_{m,i})_{\star}\Oo_{\mathfrak{IG}^{m,i}\times (\mathfrak{IG}^{\vee})^{m,i}}) \otimes {\Lambda}_i\right)[(\kappa_{m,i},\kappa'_{m,i})].
\]
This is a sheaf of ${\Lambda}_i\otimes \Oo_{\M^{\ord}_i}$-modules. Let us denote by ${\m}_i$ the kernel of the map ${\Lambda}\to {\Lambda}_{i}$. We have natural inclusions ${\m}_{i}\subset {\m}_{i-1}$ and the kernel of the natural map ${\Lambda}_i\to {\Lambda}_{i-1}$ can be identified with ${\m}_{i-1}/{\m}_{i}$. These maps induce the horizontal maps in the diagram
\[
\begin{tikzcd}
\Omega^{(\kappa_1,\kappa_2)}_{m,i} \arrow{r} &\Omega^{(\kappa_1,\kappa_2)}_{m,i-1}\\
\Omega^{(\kappa_1,\kappa_2)}_{m-1,i}\arrow{u}\arrow{r} &\Omega^{(\kappa_1,\kappa_2)}_{m-1,i-1}\arrow{u},
\end{tikzcd}
\]
where on the other hand the vertical maps are induced by the natural maps between the Igusa towers. Let $\Omega^{(\kappa_1,\kappa_2)}_{\infty,i}=\ccolim_m \Omega^{(\kappa_1,\kappa_2)}_{m,i}$. We then have
$\Omega^{(\kappa_1,\kappa_2)}=\llimm_{i}\Omega^{(\kappa_1,\kappa_2)}_{\infty,i}.$
We can identify $\Omega^{(\kappa_1,\kappa_2)}_{\infty,i}=\Omega^{(\kappa_1,\kappa_2)}/\m_i$, which is a quasi-coherent sheaf over $\mathcal{M}^{ord}_i$.

From the above construction and the definition of the descent datum in $\S$ \ref{secdescentigusa}, we obtain that the above description remains valid when we descend $\Omega^{(\kappa_1,\kappa_2)}$ and $\Omega^{(\kappa_1,\kappa_2)}/\m_i$ to sheaves over $\X^{\ord}$ and $X_i^{\ord}$ respectively. 

Finally, we denote by $\Omega^{(\kappa_1,\kappa_2)}/\m_i(-D)$ the $\Oo_{X_i^{\ord}}\otimes \Lambda$-modules $(\Omega^{(\kappa_1,\kappa_2)}\otimes_{\Oo_{\X^{\ord}}}\Oo_{\X^{\ord}}(-D))\otimes_{\Lambda}\Lambda/\m_i$, where $D$ is the cuspidal divisor in $\X$. Let 
$$\mathcal{F}_i:=\Omega^{(\kappa_1,\kappa_2)}/\m_i \ \text{and} \ \mathcal{F}_i(-D):=\Omega^{(\kappa_1,\kappa_2)}/\m_i(-D).$$ 
Note that by taking the reflexive hull of $j_\star \mathcal{F}_i$ (or $j_\star \mathcal{F}_i(-D)$) we obtain a finite rank extension of $\mathcal{F}_i$ and $\mathcal{F}_i(-D)$ over $X_i$ which is torsion free for the sheaf of ideals corresponding to $D_i$.

{\begin{prop}\label{megaprop} For any $i\geq 1$, fix $\bar{\F}_i$ a coherent sheaf of $\Lambda/\m_i$-modules extending $\F_i$ over $X_i$. Assume it is of finite rank and $\I$-torsion free (where $\I$ is the sheaf of ideals corresponding to $D_i$). Consider the natural maps $\mathcal{F}_{i+1}\to \mathcal{F}_i$.
\begin{itemize}
\item[(i)] The natural map $X_i^{\ord}\to X^{\ord}_{i+1}$ induces a well defined map $\RG^{J}_{i+1}(\bar{\F}_{i+1})\to \RG^{J}_i(\bar{\F}_i)$ and, similarly, a well-defined map $\RG^{J}_{i+1}(\bar{\F}_{i+1}(-D))\to \RG^{J}_i(\bar{\F}_i(-D))$.
\item[(ii)] The operator $T_J$ acts on $\RG_i(\bar{\F})_{\underline{m},\underline{n}}$ and it induces an action on $\RG^{J}_i(\bar{\F}_i)$ compatibly with respect to the maps $\RG^{J}_{i+1}(\bar{\F}_{i+1})\to \RG^{J}_i(\bar{\F}_i)$, defining an endomorphism of $\RG^{J}(\bar{\F})=\limm_{i}\RG^{J}_i(\bar{\F}_i)$. Moreover $(T_J^{n!})_n$ converges to an idempotent $e(T_J)$ on $\RG^{J}(\bar{\F})$.
  The same statements hold replacing $\bar{\F}_i,\bar{\F}$ by $\bar{\F}_i(-D),\bar{\F}(-D)$.
\item[(iii)] 
The complex $e(T_J)\RG^{J}_1(\bar{\F})$ is concentrated in degrees $[\# J, n]$ and $e(T_J)\RG^{J}_1(\bar{\F}(-D))$ is concentrated in degrees $[0, \# J]$.
\item[(iv)] The cohomologies of the complexes $e(T_J)\RG^{J}(\bar{\F})$ and $e(T_J)\RG^{J}(\bar{\F}(-D))$ are independent on the choice of the extensions $\bar{\F}_i$ of ${\F}_i$.
\end{itemize}
\end{prop}}

\begin{proof}

We start by proving (i). Recall that, as explained in \ref{liftshasse}, the natural map $X_i\to X_{i+1}$ maps $p\cdot D_{\tau,i}$ to the divisor $D_{\tau,i+1}$. Hence the map induces $\RG_{i+1}(\bar{\F}_{i+1})_{\underline{m},\underline{n}}\to \RG_i(\bar{\F}_i)_{p\underline{m},p\underline{n}}$, where $(p\underline{m})_{\tau}=pm_{\tau}$ and similarly for $p\underline{n}$, and we obtain the desired map passing to the limits-colimits.

To prove (ii), we exhibit a continuous action of $T_J$ on $ \RG^{J}_i(\bar{\F}_i)$ compatible for all $i$'s.
For every $\p$, we can decompose $X_0(\p)_i^{\ord}=X_0(\p)^{\et}_i \sqcup X_0(\p)^{m}_i$, where $X_0(\p)^{\et}_i$ and $X_0(\p)^{m}_i$ are the components where the universal $\p$-isogeny has étale and connected kernel. Using the diagram \eqref{diagramsp1p2}, we can think of $F_{\p}: F_{\p}^{\star}\Omega^{(\kappa_1,\kappa_2)} \to \Omega^{(\kappa_1,\kappa_2)}$ (respectively $U_{\p}:(F_{\p})_{\star}\Omega^{(\kappa_1,\kappa_2)} \to \Omega^{(\kappa_1,\kappa_2)}$) as a cohomological correspondence $p_2^{\star}(\Omega^{(\kappa_1,\kappa_2)}/\m_i)\to p_1^!(\Omega^{(\kappa_1,\kappa_2)}/\m_i)$ on $X_0(\p)_i^{\ord}$ given by $F_{\p}$ on $X_0(\p)^{m}_i$ and by zero on the other component (respectively by $U_{\p}$ on $X_0(\p)^{\et}_i$ and by zero on $X_0(\p)^{m}_i$). From (\ref{j_*}) we get that there exists $\ell_{\p}, \ell_{\p}'$ such that $F_{\p}$ (respectively $U_{\p}$) induces
\begin{displaymath}
F_{\p}:p_2^{\star}\bar{\F}_i\to p_1^!(\I^{-\ell_{\p}}\bar{\F}_i), \ \ \ \ \ (\text{resp.  } U_{\p}:p_2^{\star}\bar{\F}_i\to p_1^!(\I^{-\ell'_{\p}}\bar{\F}_i) \ )
\end{displaymath} 
Moreover, when restricting to $X_0(\p)^{m}_i$ (respectively $X_0(\p)^{\et}_i$), there exists $h_{\p},h'_{\p}\geq 1$ such that $p_2^{\star}(\mathcal{I}_\p^{h_{\p}})\subset p_1^{\star}(\mathcal{I}_{\p})$ (resp. $p_2^{\star}(\mathcal{I}_\p)\subset p_1^{\star}(\mathcal{I}_{\p}^{h'_{\p}})$) and $p_2^{\star}(\mathcal{I}_\q)= p_1^{\star}(\mathcal{I}_{\q})$ for every $\q\neq \p$. This gives us maps
\begin{displaymath}
F_{\p}: \RG(X_i,I_{\p}^{h_{\p}n_{\p}} \prod_{\q\not\in J}\I_\q^{-m_{\q}}\prod_{\p\neq\q\in J}\I_{\q}^{n_{\q}}\bar{\F}_i)\to \RG(X_i, \I_\p^{ n_\p-\ell_\p}\prod_{\q\not\in J}\I_\q^{-m_{\q}-\ell_\p}\prod_{\p\neq \q\in J}\I_{\q}^{n_{\q}-\ell_\p}\bar{\F}_i),
\end{displaymath}
\begin{displaymath}
U_{\p}: \RG(X_i, \prod_{\q\not\in J}\I_\q^{-m_{\q}}\prod_{\q\in J}\I_{\q}^{n_{\q}}\bar{\F}_i)\to \RG(X_i, \I_\p^{-h'_{\p}m_\p-\ell'_\p}\prod_{\p\neq\q\not\in J}\I_\q^{-m_{\q}-\ell'_\p}\prod_{\q\in J}\I_{\q}^{n_{\q}-\ell'_\p}\bar{\F}_i),
\end{displaymath}
where $\p\in J$ and $\p\not\in J$ respectively.
We can deduce that there exist $\underline{h},\underline{h}'\in (\Z_{\geq 1})^n$, $\underline{L},\underline{L}'\in (\Z_{\geq 0})^n$ such that, for every  $\underline{m}\in (\Z_{\geq 0})^{n-\#J}, \underline{n}\in(\Z_{\geq 0})^{\#J}$, $T_J$ gives a map
\begin{equation}\label{eq:TJbeforelim}
T_J:\RG_i(\bar{\F}_{i})_{(\underline{m},\underline{h}\cdot\underline{n})} \to \RG_i(\bar{\F}_{i})_{(\underline{h}'\cdot\underline{m}+\underline{L}',\underline{n}-\underline{L})}.
\end{equation}
One can take $\underline{h}=(h_{\p})_\p\in J,\underline{h}'=(h'_{\p})_\p\not\in J$ and we can replace the $\ell_{\p}$'s, $\ell'_{\p}$'s to be big enough such that $\prod_{\q\in J}h_{\q}\mid \ell_\p,\ell'_\p$ for every $\p$ and, choosing an order $\p_1,\dots,\p_s \in J, \p'_1,\dots,\p'_r \not \in J$, we can take
$$L_j=\sum_{k=1}^r\ell'_{\p'_k} + (\sum_{u=1}^{j-1}\ell_{\p_u})/h_{\p_j} +\sum_{u=j}^s\ell_{\p_u}; \ \ \ \ L'_j=h'_{\p'_j}(\sum_{u=1}^s\ell_{\p_u}+\sum_{k=1}^{j-1}\ell'_{\p'_k})+\ell'_{\p'_j}.$$
We can then rewrite \eqref{eq:TJbeforelim} as: for every $\underline{m}\in (\Z_{\geq 0})^{n-\#J}, \underline{n}\in(\Z_{\geq 0})^{\#J}$, 
there exists $\underline{\tilde{L}}'\in (\Z_{\geq 0})^{n-\#J}, \underline{\tilde{L}}\in(\Z_{\geq 0})^{\#J}$ 
independent on $(\underline{m},\underline{n})$ such that $T_J$ gives a map: 
$T_J:\RG_i(\bar{\F}_{i})_{(\underline{m},\underline{n}+\underline{\tilde{L}})} \to \RG_i(\bar{\F}_{i})_{(\underline{m}+\underline{\tilde{L}}',\underline{n})}.$ 
Hence $T_J$ commutes with all transition maps in the direct and inverse limit.  
Moreover, because the shifts are independent on $(\underline{m},\underline{n})$, 
this also establishes the continuity of \(T_J\) with respect to the limit–colimit topology
(a basis of opens can be taken to be $U_{\underline n_0} =
\ker\Bigl(\RG^J_i(\bar\F_i)\to\colim_{\underline m}\RG_i(\bar\F_i)_{\underline m,\underline n_0}\Bigr)$, then the preimage of $U_{\underline n_0}$ contains
$U_{\underline{n}_0+\underline{\tilde{L}}}$). 
It is not hard
to show that the maps (\ref{eq:TJbeforelim}) can be chosen to be compatible with respect to the map $X_i\to X_{i+1}$.
Therefore $T_J$ defines an endomorphism of $\RG^J(\bar\F)$
and for each $i$ its reduction
$T_{J,i}=T_J\mod \m_i$,
agrees with the correspondence constructed above.

Note that for every $i,\underline{m}, \underline{n}$, there exists $K$ such that for $k,k'\geq K$ the difference $T_J^{k!}-T_J^{k'!}$ belongs to the kernel of the map $\pi_{i,\underline{m},\underline{n}}$ from the endomorphisms of $\RG^J(\bar\F)$ to endomorphisms of $\RG_i(\bar\F_i)_{\underline m,\underline n}$ (because the latter is a perfect complex over $R/\wp^i$ and its endomorphism group is finite). 
Hence the sequence $(T_J^{k!})_{k\ge1}$ is Cauchy in the profinite topology and it converges to a limit
$e=\lim_{k\to\infty}T_J^{k!}\in\End(\RG^{J}(\bar{\mathcal F})).$ Note that $\m_i$ need not equal the $i$-th power $\m^i$ of the augmentation ideal of $\Lambda$; however $\m_i$ is open in the $\m$-adic topology, so for each $i$ there exists $N(i)$ with $\m^{N(i)}\subset\m_i$. 
Moreover, for each finite quotient the stable value $\pi_{i,\underline{m},\underline{n}}(T_J^{K!})$ is idempotent, because
$\pi_{i,\underline{m},\underline{n}}\big((T_J^{K!})^2\big)=\pi_{i,\underline{m},\underline{n}}(T_J^{2K!})=\pi_{i,\underline{m},\underline{n}}(T_J^{K!}),$
hence every finite-level reduction of $e$ is idempotent and consequently $e^2=e$.



We now prove (iii).
By definition, we find that $\Omega^{(\kappa_1,\kappa_2)}/\m_{1}$ is isomorphic to the sheaf 
\begin{equation}\label{eq:isoF1}
\bigoplus_{\p\mid p}\left(\bigoplus_{k_{\p}\in I_{\p}}\bigoplus_{w\in I}{\myom}^{(2\underline{k},2w)}\right),
\end{equation}
where we can choose any set $I_{\p},I$ of characters $\Z_p^{\times}\to \mathbb{F}_p^{\times}$, which in turns means making a choice of representatives of $\Z/(p-1)$ in $\Z$. In particular, we can and do choose $I_{\p}=\{-p+2,\dots,0\}$ when $\p\in J$, $I_{\p}=\{2,\dots,p\}$ when $\p\not\in J$ and $I=\{2,\dots,p\}$. Moreover, by Lemma \ref{lemmareltp}, the operator $T_J$ specialized at such choices of weights is the same as $T_p=\prod_{\p}T_{\p}$, since we are working modulo $\wp$. Hence the reduction of $e(T_J)$ agrees with the idempotent $e(T_p)$. We can then apply Theorem \ref{modpclassical} to deduce that $e(T_J)\RG^{J}_1(\bar{\F}_{1})$ is concentrated in degrees $[\#J,n]$. A similar proof applies for the cuspidal complex.

We finally prove (iv). We now fix $i$ and to ease the notation we write $\F$ for the sheaf $\F_i$. If we have two sheaves $\bar{\F},\bar{\F}'$ extending $\F$, the sheaf $\bar{\F}\cap\bar{\F}'$ also extends $\F$, hence we can reduce to prove that if $\bar{\F}'\subset\bar{\F}$ are two sheaves extending $\F$ to $X_i$, then $\RG^J_i(\bar{\F})\simeq \RG^J_i(\bar{\F}')$.

In particular, under this assumption, we have that the sheaf $\bar{\F}/\bar{\F}'$ is supported on a subset of $\cup_{\p\mid p} D_{\p,i}$, hence we find that there exists $t\geq 0$ such that $\I^t\bar{\F}\subset \bar{\F}'$ and therefore for any $\underline{m}\in (\Z_{\geq 0})^{n-\#J}, \underline{n}\in(\Z_{\geq 0})^{\#J}$ (with $m_{\tau}\gneq t$) we find maps 
\begin{equation}\label{maps}
\RG_i(\bar{\F})_{\underline{m}-t,\underline{n}+t}\to \RG_i(\bar{\F}')_{\underline{m},\underline{n}}\to \RG_i(\bar{\F})_{\underline{m},\underline{n}},
\end{equation}
where $(\underline{m}-t)_{\tau}=m_{\tau}-t,(\underline{n}+t)_{\tau}=n_{\tau}+t$.
Consider the cohomological correspondence $T_J$. As discussed above, from (\ref{j_*}) we get that there exist $\underline{L},\underline{L}', \underline{L}_1,\underline{L}'_1$ such that $T_J$ induces
\begin{displaymath}
T_J:\RG^J_i(\bar{\F})_{(\underline{m},\underline{h}\cdot\underline{n})} \to \RG^J_i(\bar{\F})_{(\underline{h}'\cdot\underline{m}+\underline{L}',\underline{n}-\underline{L})}, \ \ \ T_J':\RG^J_i(\bar{\F}')_{(\underline{m},\underline{h}\cdot\underline{n})} \to \RG^J_i(\bar{\F}')_{(\underline{h}'\cdot\underline{m}+\underline{L}'_1,\cdot\underline{n}-\underline{L}_1)}.
\end{displaymath} 
Assume without loss of generality that $\underline{L}\geq \underline{L}_1,\underline{L}'\geq \underline{L}'_1$, so that we can write $T_J':\RG^J_i(\bar{\F}'_{i})_{(\underline{m},\underline{h}\cdot\underline{n})} \to \RG^J_i(\bar{\F}'_{i})_{(\underline{h}'\cdot\underline{m}+\underline{L}',\underline{n}-\underline{L})}$. Moreover we can replace $t$ by a bigger integer and we can therefore assume that $(\prod_{\p\in J}h_{\p})\mid t$ and we have maps as in (\ref{maps}). We find the following diagram 
\begin{center}
\begin{tikzpicture}[scale=1.5]
\node (A) at (0,3) {$\RG_i(\bar{\F}')_{\underline{m},\underline{h}\underline{n}}$};
\node (B) at (7,3) {$\RG_i(\bar{\F})_{\underline{m},\underline{h}\underline{n}}$};
\node (C) at (2,2) {$\RG_i(\bar{\F}')_{\underline{m}+t,\underline{h}\underline{n}-t}$};
\node (C') at (5,2) {$\RG_i(\bar{\F})_{\underline{m}+t,\underline{h}\underline{n}-t}$};
\node (CC) at (2,1) {$\RG_i(\bar{\F}')_{\underline{h}'(\underline{m}+{t})+\underline{L}',\underline{n}-\tfrac{t}{\underline{h}}-\underline{L}}$};
\node (DD) at (5,1) {$\RG_i(\bar{\F})_{\underline{h}'(\underline{m}+t)+\underline{L}',\underline{n}-\tfrac{t}{\underline{h}}-\underline{L}}$};
\node (AA) at (0,0) {$\RG_i(\bar{\F}')_{\underline{h}'\underline{m}+\underline{L}',\underline{n}-\underline{L}}$};
\node (BB) at (7,0) {$\RG_i(\bar{\F})_{\underline{h}'\underline{m}+\underline{L}',\underline{n}-\underline{L}}$};
\path[->,font=\scriptsize]
(A) edge node[above]{} (B)
(A) edge node[left]{$T'_J$} (AA)
(B) edge node[right]{$T_J$} (BB)
(AA) edge node[above]{} (BB)
(B) edge node[above]{$f$} (C)
(C) edge node[left]{$T'_J$} (CC)
(CC) edge node[above]{} (DD)
(AA) edge node{} (CC)
(C) edge node[above]{} (C')
(C') edge node[right]{$T_J$} (DD)
(A) edge node[above]{} (C)
(B) edge node[above]{} (C')
(BB) edge node{} (DD);
\end{tikzpicture}
\end{center}
where $f$ is given by the first map in (\ref{maps}), the horizontal arrows are given by $\bar{\F}'\subset\bar{\F}$ and the maps from the bigger square to the smaller one are given by the connecting maps $\RG_i(\mathcal{G})_{\underline{m},\underline{n}}\to \RG_i(\mathcal{G})_{\underline{m}',\underline{n}'}$ for $m_{\tau}\leq m_{\tau}'$ for every $\tau \not \in J$, $n_{\tau}\geq n'_{\tau}$ for every $\tau \in J$.

Taking the limits-colimits and the inverse limit over $i$, we obtain a commutative diagram
\[
\begin{tikzcd}
\RG^J(\bar{\F}')\arrow{r}\arrow{d}{T'_J} &\RG^J(\bar{\F})\arrow{d}{T_J}\arrow[dashed]{dl}{T'_J}\\
\RG^J(\bar{\F}')\arrow{r} &\RG^J(\bar{\F}),
\end{tikzcd}
\]
which implies that the map $e(T_J')\RG^J(\bar{\F}')\to e(T_J)\RG^J(\bar{\F})$ is a quasi-isomorphism. 
\end{proof}

We now define the following $\Lambda$-complexes
\begin{displaymath}
\RG_J(\Omega^{(\kappa_1,\kappa_2)}):=\limm_{i}\RG^{J}_i(\bar{\F}_i), \ \ \ ,M_J(\Omega^{(\kappa_1,\kappa_2)}):=e(T_J)\limm_{i}\RG^{J}_i(\bar{\F}_i)
\end{displaymath}
where the limit is taken with respect to the maps of Proposition \ref{megaprop}(i).

\begin{rmk} Note that a priori the $\RG_J(\Omega^{(\kappa_1,\kappa_2)})$ may depend on the chosen extension of the sheaves $\Omega^{(\kappa_1,\kappa_2)}/\m_i$ from $X_i^{\ord}$ to $X_i$, but the definition of $M_J(\Omega^{(\kappa_1,\kappa_2)})$ is independent on such choice by Proposition \ref{megaprop}(iv). Moreover, the definition of $M_J(\Omega^{(\kappa_1,\kappa_2)})$ is also independent on the order in which we take the limits and colimits in the definitions of $\RG^{J}_i(\bar{\F}_i)$. More precisely, writing $J=J_1\sqcup J_3$, $J^c=J_2\sqcup J_4$, for $J_j\subset \Sigma_{\infty}$, one could define $\tilde{M}_J(\Omega^{(\kappa_1,\kappa_2)})=\limm_{i}e(T_J)\tilde{\RG}^{J}_i(\bar{\F}_i)$, where 
\[
\tilde{\RG}^{J}_i(\bar{\F}_i)= \left(\limm_{n_\tau}\right)_{\tau \in J_1}\left(\colim_{m_\tau}\right)_{\tau \in J_2}\left(\limm_{n_\tau}\right)_{\tau \in J_3}\left(\colim_{m_\tau}\right)_{\tau \in J_4}\RG_i(\bar{\F}_i)_{\underline{m},\underline{n}}.
\]
The definition of ${M}_J(\Omega^{(\kappa_1,\kappa_2)})$ is the one for the choice $J_3=J_2=\emptyset$. We have natural maps between the $\tilde{M}_J(\Omega^{(\kappa_1,\kappa_2)})$ for different choices of $J_j$ and for different ordering in each subset $J_j$. Using the isomorphism \eqref{eq:isoF1}, Corollary \ref{corprojector} and Lemma \ref{lemmareltp} we obtain that these maps are quasi-isomorphisms modulo the maximal ideal $\m$. Then \cite[Proposition 2.2.2]{pilloni} implies that the $\tilde{M}_J(\Omega^{(\kappa_1,\kappa_2)})$ are quasi-isomorphic $\Lambda$-complexes.
\end{rmk}

In order to state the main theorem, we consider the natural maps, which are the characteristic zero analogues of \eqref{eq:nat}-\eqref{eq:nat2} in Remark \ref{rmknaturalmapsmodp}:
\begin{align*}
\RG_J(\Omega^{(\kappa_1,\kappa_2)}) \to \limm_i(\colim_{m_\tau})_{\tau\not\in J} \RG(X_i, \prod_{\tau\not\in J}\I_\tau^{-m_{\tau}}\bar{\F}_i) \leftarrow \limm_i \RG(X_i,\bar{\F}_i),\\
\RG_J(\Omega^{(\kappa_1,\kappa_2)}) \leftarrow  \limm_i(\limm_{n_\tau})_{\tau\in J} \RG(X_i, \prod_{\tau\in J}\I_\tau^{n_{\tau}}\bar{\F}_i) \to  \limm_i \RG(X_i,\bar{\F}_i).
\end{align*}
Moreover, using the isomorphism \eqref{eq:specialiso} for any $(\underline{k},w)\in \Z^{n+1}$, we obtain 
\begin{align}\label{eq:map}
\RG_J(\Omega^{(\kappa_1,\kappa_2)})\otimes_{\Lambda, (\underline{k},w)} R\to \limm_i(\colim_{m_\tau})_{\tau\not\in J} \RG(X_i, \prod_{\tau\not\in J}\I_\tau^{-m_{\tau}}\myom^{(2\underline{k},2w)}) \leftarrow \RG(X,\myom^{(2\underline{k},2w)}),\\
\label{eq:map2}
\RG_J(\Omega^{(\kappa_1,\kappa_2)}) \otimes_{\Lambda, (2\underline{k},2w)}R \leftarrow  \limm_i(\limm_{n_\tau})_{\tau\in J} \RG(X_i, \prod_{\tau\in J}\I_\tau^{n_{\tau}}\myom^{(2\underline{k},2w)}) \to  \RG(X,\myom^{(2\underline{k},2w)}).
\end{align}
\begin{thm}\label{mainthm}
The $\Lambda$-complex $M_J(\Omega^{(\kappa_1,\kappa_2)})$ is a perfect complex and is concentrated in degrees $[\# J,n]$. Let $\underline{k}\in \Z^n$, $w\in \Z$ such that $2k_{\p}\leq -1$ for $\p \in J$, $2k_{\p}\geq 3$ for $\p\not\in J$. Then
\[
M_J(\Omega^{(\kappa_1,\kappa_2)})\otimes_{\Lambda,(\underline{k},w)}R \simeq e(T_p)\RG(X,\myom^{(2\underline{k},2w)}).
\]
\end{thm}
\begin{proof}
The vanishing result follows from \cite[Lemma  2.6.7]{highersiegel}, once one shows that the complex is perfect, we postpone the proof of this fact to $\S$\ref{sec:perfectness}. Assuming that, we can apply the fact that the sheaf modulo $\m$, the maximal ideal of $\Lambda$, is isomorphic to the sheaf \eqref{eq:isoF1}, together with Theorem \ref{modpclassical} to show that the complex modulo $\m$ is perfect and concentrated in the right degrees. In order to prove the classicality result, 
recall that, by \eqref{eq:specialiso}, we obtain an isomorphism
\begin{displaymath}
M_J(\Omega^{(\kappa_1,\kappa_2)})\otimes_{\Lambda, (\underline{k},w)} R \simeq \RG_J((\myom^{(2\underline{k},2w)})_{|\X^{\ord}_R}),
\end{displaymath}
where the right hand side is the image of $e(T_J)$ of the limit over $i$ of $$(\limm_{n_\p})_{\p\in J} (\colim_{n_\p})_{\p\not\in J} \RG(X_i, \prod_{\p\not\in J}I_{\p}^{-n_\p}\prod_{\p\in J}I_{\p}^{n_\p}{\myom}^{(2\underline{k},2w)}).$$ 
Hence applying the projectors, \eqref{eq:map} and \eqref{eq:map2} give us maps 
\begin{align*}
M_J(\Omega^{(\kappa_1,\kappa_2)})\otimes_{\Lambda, (\underline{k},w)} R\to e(T_J)\limm_i(\colim_{m_\tau})_{\tau\not\in J} \RG(X_i, \prod_{\tau\not\in J}\I_\tau^{-m_{\tau}}\myom^{(2\underline{k},2w)}) \leftarrow e(T_p)\RG(X,\myom^{(2\underline{k},2w)}),\\
M_J(\Omega^{(\kappa_1,\kappa_2)}) \otimes_{\Lambda, (\underline{k},w)}R \leftarrow  e(T_J)\limm_i (\limm_{n_\tau})_{\tau\in J} \RG(X_i, \prod_{\tau\in J}\I_\tau^{n_{\tau}}\myom^{(2\underline{k},2w)}) \to  e(T_p)H^{*}(X,\myom^{(2\underline{k},2w)}),
\end{align*}
where we used that, by our assumptions on the weights and Lemma \ref{lemmareltp}, the projectors $e(T_J)$ and $e(T_p)$ are the same.
The classicality result modulo $\wp$ (Theorem \ref{modpclassical}, combined with Remark \ref{rmknaturalmapsmodp}) and Lemma \ref{lemmareltp} imply that these maps are isomorphisms modulo $\wp$. By \cite[Proposition 2.2.2]{pilloni} we deduce that they are isomorphisms over $R$.
\end{proof}

\begin{rmk}\label{rmkcuspidal}
We obtain the analogous result if we consider cuspidal cohomology, i.e. $M_J(\Omega^{(\kappa_1,\kappa_2)}(-D))$. In this case, the complex is concentrated in degrees $[0,\# J]$.
\end{rmk}

\subsubsection{Perfectness}\label{sec:perfectness} In order to show that the complex of $\Lambda$-modules $M_J(\Omega^{(\kappa_1,\kappa_2)})$ is perfect, so that we can apply Nakayama's lemma in the form of \cite[Lemma  2.6.7]{highersiegel}, we will show that for every $i$ the $\Lambda_i$-complex $e(T_J)\RG^{J}_i(\bar{\F}_i)$ is a perfect complex concentrated in bounded degrees independent on $i$.

Recall that we have fixed a mod $\wp^i$ lift of the partial Hasse invariant: 
$\tilde{h}_{\tau,i} \in H^0(X_i, \omega_{\tau}^{\otimes p^{i-1}(p-1)}).$
which is a lift of the $p^{i-1}$-th power of $h_{\tau}\in H^0(X_{1},\omega_{\mathbb{F},\tau}^{\otimes (p-1)})$ (unique up to units). We let
$D_{\tau,i}:=Va(\tilde{h}_{\tau,i})$
be the divisor on $X_i$ given by the vanishing locus of $\tilde{h}_{\tau,i}$. Under the natural map $X_1\to X_i$, the divisor $p^{i-1}\cdot D_\tau$ on $X_1$ is mapped to $D_{\tau,i}$. In other words, as subschemes of $X_i$, we have 
\begin{equation}\label{eq:red}
(D_{\tau,i})_{\rm red}= D_{\tau,1}
\end{equation}

Let $\p$ the prime corresponding to $\tau$. 
Let $(C_i^{\bullet},X_i^{\ord}, p_1,p_2)$ 
be the reduction mod $\wp^i$ of the 
correspondence $(\mathfrak{X}_0(\p)^{\bullet},p_1,p_2)$ over 
$\mathfrak{X}^{\ord}$, where $\bullet\in \{\rm m, \et\}$:

\begin{equation}\label{diagramsp1p2}
\begin{tikzcd}[column sep=1.1em, row sep=1em]
&\X_0(\p)^{\et}\arrow{dr}{p_2}[swap]{\simeq} \arrow{dl}[left]{p_1=F_{\p}}\\
\X^{\ord} &&\X^{\ord}
\end{tikzcd}
\begin{tikzcd}[column sep=1.1em, row sep=1em]
&\X_0(\p)^{m}\arrow{dr}[right]{p_2=F_{\p}} \arrow{dl}{\simeq}[swap]{p_1}\\
\X^{\ord} &&\X^{\ord}.
\end{tikzcd}
\end{equation}

As in \cite[Prop.4.4.4]{highersiegel} in the Siegel case, we can compactify this correspondence over $X_i^{\ord}$ by taking the closure of $X_i^{\ord}$ in $X_i$ (which is $X_i$) and defining $\bar{C}_i^\bullet$ to be $p_1^\star X_i \cap p_2^\star X_i$. If $i=1$ this is the restriction of the correspondence 
\[
\begin{tikzcd}[column sep=1.1em, row sep=1em]
&X_0(\p)_1\arrow{dr}{p_2}[swap]{} \arrow{dl}[left]{p_1}\\
X_1 &&X_1
\end{tikzcd}
\]
to the complement of the étale/multiplicative locus respectively. 

\begin{lemma}
For the compactified correspondence $(\bar{C}_i^\bullet,X_i,p_1,p_2)$, there exists $s<1$ such that:
\begin{itemize}
\item $p_2^\star D_{\p,1} \leq s p_1^\star D_{\p,1}$ \footnote{We say that $D\leq s D'$ for closed subschemes $D=Va(I),D'=Va(I')$ and $s = \tfrac{p}{q}$, if $(I')^p \subseteq I^q$.} if $\bullet = {\et}$,
\item $p_1^\star D_{\p,1} \leq s p_2^\star D_{\p,1}$ if $\bullet = {\rm m}$.
\end{itemize}
\end{lemma}
\begin{proof}
This follows from \eqref{eq:red}, the fact that on $X_1 = (X_n)_{\rm red}$ we have $p_1^\star D_{\p,1} = p \cdot p_2^\star D_{\p,1}$ and $p_2^\star D_{\p,1} = p \cdot p_1^\star D_{\p,1}$ respectively, and the local computations in the proof of \cite[Lemma 2.1.35]{highersiegel}. In the terminology of \emph{op. cit.}, this says that $(\bar{C}_i^\bullet,X_i,p_1,p_2,D_{\p,1})$ is a strict dynamic compactification of $(C_i^{\bullet},X_i^{\ord}, p_1,p_2)$.
\end{proof} 
Before proceeding, we remark that we are encountering the problem of the closed subscheme $D_{\p,1}$ in $X_n$ being not locally principal. By \cite[Lemma 2.1.30]{highersiegel}, we can find by blow up a strict map of strict dynamic correspondences $(\tilde{C}_i^{\bullet},\tilde{X}_i, \tilde{D}_{\p,n}) \to (\bar{C}_i^\bullet,X_i, D_{\p,1})$ such that $\tilde{D}_{\p,n}$ is now a Cartier divisor. Proposition 2.2.11 in \emph{op. cit.} then guarantees that it induces a quasi-isomorphisms 
\[
\colim_m \RG(X_i, \mathcal{F}(mD_{\p,1})) \simeq \colim_m\RG (\tilde{X}_i, \mathcal{F}(m\tilde{D}_{\p,i}))
\]
\[
\limm_m \RG(X_i, \mathcal{F}(-nD_{\p,1})) \simeq \limm_m\RG (\tilde{X}_i, \mathcal{F}(-n\tilde{D}_{\p,i}))
\]
compatible with the action of the correspondences $\bar{C}_i^\bullet$ and $\tilde{C}_i^{\bullet}$. By abuse of notation we will denote by $(\bar{C}_i^\bullet,X_i,p_1,p_2 D_{\p,i})$ such dynamic compactification with $D_{\p,i}$ a Cartier divisor satisfying 
\begin{itemize}
\item $p_2^\star D_{\p,i} \leq s p_1^\star D_{\p,i}$ if $\bullet = {\et}$,
\item $p_1^\star D_{\p,i} \leq s p_2^\star D_{\p,i}$ if $\bullet = {\rm m}$.
\end{itemize}
for some $0<s<1$. 
\begin{lemma}
There exist  $0<a<b \in \Z_{\geq 0}$ and maps
\[
p_2^\star \Oo_{X_n}(bD_{\p,i}) \to  p_1^\star \Oo_{X_n}(aD_{\p,i}) \ \ \ \text{if }\bullet = {\et}\]\[
p_2^\star \Oo_{X_n}(-aD_{\p,i}) \to  p_1^\star \Oo_{X_n}(-bD_{\p,i}) \ \ \ \text{if }\bullet = {\rm m}
\]
\end{lemma}
\begin{proof}
This is proved in \cite[Lemma 2.1.17 (3)]{highersiegel}, requiring the closed subschemes to be locally principal.
\end{proof}

Fix a coherent sheaf $\mathcal{F}$ over $X^{\ord}_i$ and a coherent extension $\bar{\mathcal{F}}$ to $X_i$. As explained in Proposition \ref{megaprop} using \eqref{j_*} (or, alternatively in \cite[Proposition 2.3.13]{highersiegel}), there exists $
\ell=\ell_{\p}, \ell'=\ell_{\p}'$ such that $F_{\p}$ (respectively $U_{\p}$), defined from the correspondence $(C_i^{\rm m},X_i^{\ord}, p_1,p_2)$ (respectively $(C_i^{\et},X_i^{\ord}, p_1,p_2)$) induces
\begin{equation}\label{eq:elltwist}
F_{\p}:p_2^{\star}\bar{\mathcal{F}}\to p_1^!\bar{\mathcal{F}}(\ell D_{\p,i}), \ \ \ \ \ (\text{resp.  } U_{\p}:p_2^{\star}\bar{\mathcal{F}}\to p_1^!\bar{\mathcal{F}}(\ell' D_{\p,i}) \ )
\end{equation}


\begin{prop} For $m,n$ large enough, $F_{\p}$ and $U_\p$ extend to maps of sheaves over $X_i$
\[
F_{\p}:p_2^{\star}\bar{\mathcal{F}}(-nb D_{\p,i})\to p_1^!\bar{\mathcal{F}}(-nb D_{\p,i}), \ \ \ \ \  U_{\p}:p_2^{\star}\bar{\mathcal{F}}(mb D_{\p,i})\to p_1^!\bar{\mathcal{F}}(mb D_{\p,i}) \ 
\]
Moreover they factorize as follows: 
\[
\begin{tikzcd}
p_2^* \bar{\mathcal{F}}(-nb D_{\p,i}) \arrow[r, dashed] \arrow[d] & p_1^! \bar{\mathcal{F}}(-(n+1)b D_{\p,i}) \\
p_2^* \bar{\mathcal{F}}(-(n+1)b D_{\p,i}), \arrow[ur, "F_{\p}"]
\end{tikzcd}\]
\[
\begin{tikzcd}
p_2^* \bar{\mathcal{F}}((m+1)bD_{\p,i}) \arrow[r, dashed] \arrow[d] & p_1^! \bar{\mathcal{F}}(mb D_{\p,i}) \\
p_2^* \bar{\mathcal{F}}(mb D_{\p,i}) \arrow[ur, "U_{\p}"].
\end{tikzcd}
\]
\end{prop}
\begin{proof}
The proof follows as in \cite[Lemma 2.5.3]{highersiegel}. From \eqref{eq:elltwist} and the lemma above, we obtain maps
\[
F_{\p}:p_2^{\star}\bar{\mathcal{F}}(-(an+\ell) D_{\p,i})\to p_1^!\bar{\mathcal{F}}(-nb D_{\p,i}), \ \ \ \ \ U_{\p}:p_2^{\star}\bar{\mathcal{F}}(mb D_{\p,i})\to p_1^!\bar{\mathcal{F}}((ma+\ell') D_{\p,i}) .
\]
There exist $m_0,n_0$ such that for all $m\geq m_0$ we have $mb>\ell' +(m+1)a$ and for all $n>n_0$, we have $bn> \ell +(n+1)a$. We therefore obtain maps
\[
\begin{tikzcd}
p_2^* \bar{\mathcal{F}}(-nb D_{\p,i}) \arrow[d] & \\
p_2^* \bar{\mathcal{F}}(-(\ell +(n+1)a) D_{\p,i}) \arrow[r,"F_{\p}"] & p_1^! \bar{\mathcal{F}}(-(n+1)b D_{\p,i}) \arrow[d] \\ 
& p_1^! \bar{\mathcal{F}}(-nb D_{\p,i}),
\end{tikzcd}\]
where the vertical arrows are the natural maps induced by $\bar{\mathcal{F}}(-n'D_{\p,i})\to \bar{\mathcal{F}}(-nD_{\p,i})$ with $n'>n$ (and similarly for $U_{\p})$.
\end{proof}

This proposition can be thought of as a generalisation of Proposition \ref{prop:shiftTpmodp} (where $i=1$ and the shift is controlled more precisely since we were working with $\mathcal{F}$ an automorphic sheaf). It therefore implies, similarly as in Corollary \ref{corprojector}, a classicality mod $\wp^i$ for $i>1$. In particular it implies that the complexes
\begin{align} 
&e(F_{\p}) \limm_{n_{\p}} \RG(X_i, \bar{\mathcal{F}}(-n_{\p}D_{\p,i})) \label{eq:modpicomplex1} \\
&e(U_{\p})  \colim_{m_{\p}}\RG(X_i, \bar{\mathcal{F}}(m_{\p}D_{\p,i}))\label{eq:modpicomplex2} \\
&e(T_J)\RG^{J}_i(\bar{\F}_i) = e(T_J) \left(\limm_{n_{\p}}\right)_{\p\in J} \left(\colim_{m_{\p}}\right)_{\p\not\in J}\RG(X_i, \bar{\mathcal{F}}(\sum_{\p\not\in J}m_{\p}D_{\p,i}-\sum_{\p\in J}n_{\p}D_{\p,i})) \label{eq:modpicomplex3}
\end{align}
are isomorphic to the respective ordinary part of a complex $\RG(X_i,\mathcal{G})$, where $\mathcal{G}$ is the coherent sheaf $\mathcal{F}$ with fixed order of zeros/poles along the divisor(s) $D_{\p,i}$. In particular:  

\begin{cor}
Let $\mathcal{F}$ be a coherent sheaf of $R/\wp^i$-modules over $X^{\ord}_i$ and fix a coherent extension $\bar{\mathcal{F}}$ to $X_i$. The complexes \eqref{eq:modpicomplex1}, \eqref{eq:modpicomplex2}, \eqref{eq:modpicomplex3} are represented by bounded complexes of finite $R/\wp^i$-modules (and are therefore by \cite[\href{https://stacks.math.columbia.edu/tag/066E}{Tag 066E}]{stacks-project} pseudo-coherent complexes) with cohomology concentrated in bounded degrees independent on $i$.
\end{cor}

We also remark that the same result holds replacing $R/\wp^i$ by an Artinian ring (and taking the base change of $X_i$). Finally, we are interested in considering the sheaf
$$\mathcal{F}_i=\Omega^{(\kappa_1,\kappa_2)}/\m_i \simeq \Omega_{i,i}^{(\kappa_1,\kappa_2)}$$
(or its cuspidal version). Recall that this is obtained as pushforwards to $X_i^{\ord}$ of the structure sheaves of certain Igusa towers (of level $p^i$) and considering the natural action of $G_i:=(\Z/p^i)^{n+1}$ on it. It is a sheaf of $\Lambda_i=R/\wp^i[G_i]$-modules.

\begin{cor}
Let $\mathcal{F}$ be $\mathcal{F}_i$ in \eqref{eq:modpicomplex1}, \eqref{eq:modpicomplex2}, \eqref{eq:modpicomplex3}. Then such complexes are perfect complexes of $\Lambda_i$-modules. 
\end{cor}
\begin{proof}
By \cite[Lemma 2.6.11]{highersiegel}, the sheaves $\mathcal{F}_i$ are $\Lambda_i$-flat and we can therefore apply Theorem 2.6.10 in \emph{op. cit.}: since $\Lambda_i$ is product of local rings, and we have shown above that over such rings the complexes are pseudo-coherent, Nakayama's lemma in the form of   
\cite[Lemma 2.6.7]{highersiegel} tells us that it is enough to test perfectness modulo the maximal ideal and this reduces the statement to Theorem \ref{modpclassical}.
\end{proof}

This, in particular, implies by \cite[\href{https://stacks.math.columbia.edu/tag/0CQG}{Tag 0CQG}]{stacks-project} that $M_J(\Omega^{(\kappa_1,\kappa_2)})$ and $M_J(\Omega^{(\kappa_1,\kappa_2)}(-D))$ are perfect complexes of $\Lambda$-modules as required. 

\subsubsection{Localisation at a non-Eisenstein maximal ideal} Finally, we show that, after localising at a non-Eisenstein maximal ideal of the Hecke algebra, the complexes $M_J(\Omega^{(\kappa_1,\kappa_2)})$ and $M_J(\Omega^{(\kappa_1,\kappa_2)}(-D))$ give rise to a finite projective $\Lambda$-module interpolating the ordinary cohomology in degree $\# J$. More precisely, let us consider the Hecke sub-algebra $\mathbb{T}\subset R[T_{\mathfrak{a}},\mathfrak{a}\subset \Oo_F]$ generated by Hecke operators outside a finite set of places containing the ones dividing $\mathfrak{N}p$, where $\mathfrak{N}$ is such that $\Gamma(\N)\subset K$, the level of the Hilbert modular surface $X$.

We have the usual action of the Hecke algebra $\mathbb{T}$ on $\RG(X,\myom^{(\underline{k},w)})$ and $\RG(X,\myom^{(\underline{k},w)}(-D))$, for any weight $(\underline{k},w)$. We need to verify that $\mathbb{T}$ acts on $M_J(\Omega^{(\kappa_1,\kappa_2)})$ and $M_J(\Omega^{(\kappa_1,\kappa_2)}(-D))$, compatibly with the classicality isomorphisms of Theorem \ref{mainthm} and Remark \ref{rmkcuspidal}. Let $\mathfrak{q}$ be a prime ideal of $\Oo_F$ coprime to $p$ and $\N$. The action of $T_\q$ on classical coherent cohomology of $X$ is given in terms of a correspondence $(C,p_1,p_2)$ and we can make sense of such smooth correspondence also over $\X^{\ord}$ (see for example \cite[Proposition 15.1.1]{pilloni}, where this is done for the Siegel threefold). Since the correspondence parametrises $\q$-isogenies and $\q$ is coprime to $p$, we obtain isomorphisms at the level of the Igusa tower and in particular an isomorphism
\begin{displaymath}
p_2^{\star}(\Omega^{(\kappa_1,\kappa_2)})\to p_1^{\star}(\Omega^{(\kappa_1,\kappa_2)}).
\end{displaymath}
Using again the fact that the isogeny is coprime to $p$, we find analogous isomorphisms for the divisors given by the vanishing loci of the (lifts) of the partial Hasse invariants. We hence obtain an action of $T_\q$ on the complex $\RG_i(\bar{\F})_{\underline{m},\underline{n}}=\RG(X_i, \prod_{\tau\not\in J}\I_\tau^{-m_{\tau}}\prod_{\tau\in J}\I_{\tau}^{n_{\tau}}\bar{\F})$ for $\bar{\F}=\Omega^{(\kappa_1,\kappa_2)}/\m^i$ compatible with the natural maps we introduced when varying $n_\tau,m_\tau$ and $i$. Since this action commutes with respect to the projector $e(T_J)$, we have produced an action of $T_\q$ on $M_J(\Omega^{(\kappa_1,\kappa_2)})$ which is compatible by construction with the one on classical coherent cohomology.

Now let $\bar{\rho}: G_{F} \rightarrow \GL_2\left(\bar{\mathbb{F}}_{p}\right)$ be a Galois representation, unramified away from the primes not dividing $p\N$. We assume that $\bar{\rho}$ is absolutely irreducible. We let $\mathfrak{M}$ be the associated maximal ideal of the Hecke algebra $\mathbb{T}$ and $\Theta_{\mathfrak{m}}: \mathbb{T} \rightarrow \bar{\mathbb{F}}_{p}$ the corresponding morphism. Since $M_J(\Omega^{(\kappa_1,\kappa_2)})$ is a bounded above perfect complex, the subalgebra of its endomorphisms generated by $\mathbb{T}$ is a finite $\Lambda$-algebra, which admits a decomposition as a product of its localisations at the finitely many maximal ideals of $\mathbb{T}$. We can therefore consider $M_J(\Omega^{(\kappa_1,\kappa_2)})_{\mathfrak{M}}$ and, reasoning in a similar manner, $M_J(\Omega^{(\kappa_1,\kappa_2)}(-D))_{\mathfrak{M}}$.

{\begin{prop}\label{propisomodnoneis}
The natural map of $\Lambda$-complexes $M_J(\Omega^{(\kappa_1,\kappa_2)}(-D))\to M_J(\Omega^{(\kappa_1,\kappa_2)})$ becomes a quasi-isomorphism after localising at $\mathfrak{M}$. In particular, $H^{\# J}(M_J(\Omega^{(\kappa_1,\kappa_2)}(-D))_{\mathfrak{M}})\simeq H^{\# J}(M_J(\Omega^{(\kappa_1,\kappa_2)})_{\mathfrak{M}})$ is a finite projective $\Lambda$-module, which specialises, for classical weights as in Theorem \ref{mainthm}, to $e(T_p)H^{\# J}(X,\myom^{(2\underline{k},2w)})_{\mathfrak{M}}$.
\end{prop}}
\begin{proof}
We need to prove $M_J(\Omega^{(\kappa_1,\kappa_2)}(-D))_{\mathfrak{M}} \to M_J(\Omega^{(\kappa_1,\kappa_2)})_{\mathfrak{M}}$ is a quasi-isomorphism. The second part of the proposition follows then directly from Theorem \ref{mainthm}. In order to prove that the map above is a quasi-isomorphism we show that it is a quasi-isomorphism modulo $\m$. By \eqref{eq:isoF1} and Theorem \ref{modpclassical}, we find that studying $M_J(\Omega^{(\kappa_1,\kappa_2)}(-D))\to M_J(\Omega^{(\kappa_1,\kappa_2)})$ modulo $\m$ reduces to studying the map of complexes
\[
e(T_p)\RG(X_1,\myom^{(\underline{k},w)}(-D)) \to e(T_p)\RG(X_1,\myom^{(\underline{k},w)}),
\]
for $k_{\p},w$ even integers with $\tfrac{k_{\p}}{2}\in I_{\p},\tfrac{w}{2}\in I$ where $I_{\p},I$ are chosen as in \eqref{eq:isoF1}. We are reduced to prove that for every $0\leq i\leq n$, the natural maps
\[
H^i(X_1,\myom^{(\underline{k},w)}(-D)) \to H^i(X_1,\myom^{(\underline{k},w)})
\]
become isomorphisms after localising at $\mathfrak{M}$. This follows from the description of the boundary (coherent) cohomology given for example in \cite{boundary}. More precisely, we have a long exact sequence
\begin{align*}
0&\to H^0(X_1,\myom^{(\underline{k},w)}(-D)) \to H^0(X_1,\myom^{(\underline{k},w)}) \to H^0(X_1,\myom^{(\underline{k},w)}\otimes \Oo_D) \to \dots \\&\to H^i(X_1,\myom^{(\underline{k},w)}(-D)) \to H^i(X_1,\myom^{(\underline{k},w)})\to H^i(X_1,\myom^{(\underline{k},w)}\otimes \Oo_D)\dots
\end{align*}
and we want to prove $H^i(X_1,\myom^{(\underline{k},w)}\otimes \Oo_D)_\mathfrak{M} =0$ for every $i$.
It is shown in \cite[Corollary 3.7.8, Corollary 4.1.12]{boundary} that the cohomology of the toroidal boundary of a Shimura variety with coefficients in an automorphic vector bundle can be expressed in terms of the cohomology of the (restriction) of certain automorphic vector bundles over the Shimura varities whose union gives the Baily-Borel boundary. The results of \emph{op. cit.} are actually proved in characteristic zero and the methods used are not expected to work in general in positive characteristic. However, the situation is rather simple in the case of Hilbert modular varieties, as the Baily-Borel boundary is zero-dimensional, and similar techniques can be used. More precisely, the Leray spectral sequence for the (restriction to the boundaries of the) canonical map $\pi:X_1\to X_1^{BB}$ converges $H^j(\partial X^{BB}_1,\mathrm{R}^i\pi_{\star}\myom^{(\underline{k},w))}) \Rightarrow H^{i+j}(\partial X_1,\myom^{(\underline{k},w))})$ (one possibly needs to restrict to each component of the boundary). Hence the vanishing of $H^i(X_1,\myom^{(\underline{k},w)}\otimes \Oo_D)$ after localisation will follow from the same statement for $H^0(\partial X^{BB}_1,\mathrm{R}^i\pi_{\star}\myom^{(\underline{k},w)})$. 
One can compute the stalk at a cuspidal point of $\mathrm{R}^i\pi_{\star}\myom^{(\underline{k},w)}$ using the description of the components of the toroidal compactification over a point of $\partial X^{BB}_1$ (see for example \cite[$\S$5.3]{lan12}, \cite[$\S$6,7]{lan13}): cusps are indexed by rational parabolic subgroups of $G$ whose associated torus gives rise to the toroidal boundary. The restriction of $\myom^{(\underline{k},w)}$ to the component of the toridal boundary is a torus-equivariant line bundle $\mathcal{L}_\chi$ associated to a character $\chi$ of the torus depending on the weight $(\underline{k},w)$. 
One then has
$H^0(\{c\},R^i \pi_*(\myom^{(\underline{k}, w)}) = H^i\left( \pi^{-1}(c), \mathcal{L}_\chi \right)$. The systems of eigenvalues contributing to these cohomology group are the ones induced from a proper Levi subgroup and therefore Eisenstein and these, on the other hand, gives rise to a reducible Galois representation $\rho': G_F \to \GL_2(\bar{\mathbb{F}}_p)$. We have therefore shown that after localising at a non-Eisenstein maximal ideal $\mathfrak{M}$ we obtain $H^i(X_1,\myom^{(\underline{k},w)}(-D)) \simeq H^i(X_1,\myom^{(\underline{k},w)})$ for every $i$.
%
\end{proof}

%

\subsection{Duality}\label{secduality}
The goal of this section is to define a pairing
\begin{displaymath}
\langle -,-\rangle: H^{\#J}(M_J(\Omega^{(\kappa_1,\kappa_2)}))_{\mathfrak{M}}\times H^{\#J^c}(M_{J^c}(\Omega^{(2-\kappa_1,-1-\kappa_2)}(-D)))_{\mathfrak{M}} \to \Lambda
\end{displaymath}
interpolating in classical weights the Serre duality pairing. Let us fix $J\subset \Sigma_{\infty}$ and let $i_J=\# J$. For every $i$, consider the modules
\begin{displaymath}
A_{\underline{m},\underline{n}}=H^{i_J}(X_i, \prod_{\tau\not\in J}\I_\tau^{-m_{\tau}}\prod_{\tau\in J}\I_{\tau}^{n_{\tau}}\bar{\F}_i), 
\end{displaymath}
\begin{displaymath}
B_{\underline{m},\underline{n}}=H^{n-i_J}(X_i, \prod_{\tau\not\in J}\I_\tau^{m_{\tau}}\prod_{\tau\in J}\I_{\tau}^{-n_{\tau}}\check{\bar{\F}}_i\otimes (\myom^{(\underline{2},-\underline{1})}(-D)\otimes \Lambda_i) )
\end{displaymath}
which come with the Serre duality pairing, that we denote by
\[
\langle -,-\rangle_{\underline{m},\underline{n}}: A_{\underline{m},\underline{n}}\times B_{\underline{m},\underline{n}} \to \Lambda_i.
\]
We consider
$H^{i_J,J}_i(\bar{\F}_i)= \left(\limm_{n_{\tau}}\right)_{\tau\in J} \left(\colim_{m_{\tau}}\right)_{\tau\not\in J} A_{\underline{m},\underline{n}}$ and let
\begin{displaymath}
H^{n-i_J,J^c}_i(\check{\bar{\F}}_i)= \left(\colim_{n_{\tau}}\right)_{\tau\in J}\left(\limm_{m_{\tau}}\right)_{\tau\not \in J}  B_{\underline{m},\underline{n}}.
\end{displaymath} 
\begin{lemma}
The pairing $\langle -,-\rangle_{\underline{m},\underline{n}}$ induces a well-defined pairing
\begin{displaymath}
\langle -,-\rangle: H^{i_J,J}_i(\bar{\F}_i) \times H^{n-i_J,J^c}_i(\check{\bar{\F}}_i) \to \Lambda_i.
\end{displaymath}
\end{lemma}
\begin{proof}
For any $\q\not\in J, \p\in J$, let $1_\q\in \Z^{n-\#J},1_\p\in \Z^{\#J}$ be the vectors which are equal to zero everywhere but at the $\q$-th (respectively $\p$-th) place, where they are equal to $1$. Hence we have maps 
\begin{displaymath}
A_{\underline{m},\underline{n}} \xrightarrow{a_{m_{\q}}} A_{\underline{m}+1_\q,\underline{n}}, \ \ \ \ A_{\underline{m},\underline{n}} \xrightarrow{a_{n_{\p}}} A_{\underline{m},\underline{n}-1_\p}, \ \ \ \ 
B_{\underline{m},\underline{n}} \xrightarrow{b_{m_{\q}}} B_{\underline{m}-1_\q,\underline{n}}, \ \ \ \ B_{\underline{m},\underline{n}} \xrightarrow{b_{n_{\p}}} B_{\underline{m},\underline{n}+1_\p}.
\end{displaymath}
Since the pairings $\langle -,-\rangle_{\underline{m},\underline{n}}$ are just obtained by Serre duality, they are compatible with respect to these maps. Namely, the following diagram is commutative
\[
\begin{tikzcd}[column sep=1.5em]
A_{\underline{m}-{1_{\q}},\underline{n}} \arrow{d}{a_{m_{\q}-1}} &\times & B_{\underline{m}-{1_{\q}},\underline{n}}\arrow{dr}{\langle -,-\rangle_{\underline{m}-{1_{\q}},\underline{n}}} & \\
A_{\underline{m},\underline{n}}\arrow{d}{a_{n_{\p}}} &\times & B_{\underline{m},\underline{n}}\arrow{u}{b_{m_{\q}}}\arrow{r}{\langle -,-\rangle_{\underline{m},\underline{n}}} & \ \ \ \Lambda_i. \\
A_{\underline{m},\underline{n}-{1_{\p}}} &\times & B_{\underline{m},\underline{n}-{1_{\p}}} \arrow{u}{b_{n_{\p}-1}}\arrow{ur}[swap]{\langle -,-\rangle_{\underline{m},\underline{n}-{1_{\p}}}} &
\end{tikzcd}
\]
Therefore the pairings $\langle -,-\rangle_{\underline{m},\underline{n}}$ induce a well-defined pairing on the limits with respect to these maps.
\end{proof}
Let $\Omega^{(2-\kappa_1,-\kappa_2)}(-D)=\myom^{(\underline{2},-\underline{1})}(-D)\otimes \Hom(\Omega^{(\kappa_1,\kappa_2)},\Lambda\otimes \Oo_{\X^{\ord}})$. Taking the limit over $i$, we obtain from the previous lemma, a pairing
\begin{displaymath}
\langle -,-\rangle: H^{\#J}_J(\Omega^{(\kappa_1,\kappa_2)})\times H^{\#J^c}_{J^c}(\Omega^{(2-\kappa_1,-\kappa_2)}(-D)) \to \Lambda.
\end{displaymath}
Note that we have an isomorphism of $\Lambda \otimes \Oo_{\X^{\ord}}$-modules $\Omega^{(2-\kappa_1,-\kappa_2)}(-D)\simeq \Omega^{(\kappa_1,\kappa_2)}(-D)\otimes_{\phi,\Lambda} \Lambda$, where $\phi$ is the automorphism of $\Lambda$ induced by the character
\begin{align*}
\phi: &(\Z_p^\times)^{n+1} \to \Lambda^{\times}\\
& ((x_\p)_\p, y)\mapsto \prod_{\p}x_\p^2\cdot\kappa_1((x_\p)_\p)^{-1}\kappa_2(y)^{-1}.
\end{align*}
Note that, similarly as in the proof of Proposition \ref{megaprop}, we have a well-defined action of the operator $T_{J^c}$ on $H^{\#J^c}_{J^c}(\Omega^{(2-\kappa_1,-\kappa_2)}(-D))$. Moreover, the classicality result (Theorem \ref{mainthm}) for this module reads as follows: for $\underline{k}\in \Z^n$, $w\in \Z$ such that
$2k_{\p}\leq -1$ for $\p \in J$, $2k_{\p}\geq 3$ for $\p\not\in J$,
\[
N_{J^c}(\Omega^{(2-\kappa_1,-\kappa_2)}(-D))\otimes_{\Lambda,(\underline{k},w)}R \simeq e(T_p)\RG(X,\myom^{(2-2\underline{k},-2w)}(-D)),
\]
where $N_{J^c}(\Omega^{(2-\kappa_1,-\kappa_2)}(-D))$ is defined, similarly to $N_{J^c}(\Omega^{(\kappa_1,\kappa_2)}(-D))$, by applying $e(T_{J^c})$ to the complex $\limm_i\RG_i^J(\Omega^{(2-\kappa_1,-\kappa_2)}(-D))$. We now let $$M^{\#J}_J(\Omega^{(\kappa_1,\kappa_2)})=H^{\#J}(M_J(\Omega^{(\kappa_1,\kappa_2)}))_{\mathfrak{M}}, \ \  \ M^{\#J^c}_{J^c}(\Omega^{(2-\kappa_1,-\kappa_2)}(-D))=H^{\#J^c}(N_{J^c}(\Omega^{(2-\kappa_1,-\kappa_2)}(-D)))_{\mathfrak{M}}.$$

\begin{thm}\label{thmduality}
\begin{itemize}
\item[(i)] For any $(f,g)\in  H^{\#J}_J(\Omega^{(\kappa_1,\kappa_2)})\times H^{\#J^c}_{J^c}(\Omega^{(2-\kappa_1,-\kappa_2)}(-D))$, we have $$\langle \langle p \rangle^{-1}T_J f,g \rangle = \langle f,T_{J^c}g \rangle ,$$
and hence the pairing restricts to a pairing
\begin{displaymath}
\langle -,-\rangle: M^{\#J}_J(\Omega^{(\kappa_1,\kappa_2)})\times M^{\#J^c}_{J^c}(\Omega^{(2-\kappa_1,-\kappa_2)}(-D)) \to \Lambda.
\end{displaymath}
\item[(ii)] It is a perfect pairing compatible with Serre duality, namely, for any $J\subset \Sigma_{\infty}$ and classical weights $(\underline{k},w)$ as in Theorem \ref{mainthm}, the following diagram commutes
\[
\begin{tikzcd}
M^{\#J}_J(\Omega^{(\kappa_1,\kappa_2)})\otimes_{\Lambda, (\underline{k},w)} R\arrow{d}{\simeq} &\times &M^{\#J^c}_{J^c}(\Omega^{(2-\kappa_1,-\kappa_2)}(-D))\otimes_{\Lambda, (\underline{k},w)} R\arrow{d}{\simeq}\arrow{r} &R \\
e(T_p)H^{\#J}(X,\myom^{(2\underline{k},2w)})_{\mathfrak{M}} &\times &e(T_p)H^{n-\#J}(X,\myom^{(2-2\underline{k},-2w)}(-D))_{\mathfrak{M}}\arrow{ur}
\end{tikzcd}
\]
where the bottom pairing is the restriction of the classical Serre duality pairing on the ordinary part of the cohomology localised at $\mathfrak{M}$.
\end{itemize}
\end{thm}
\begin{proof}
Recall that $\Z^{n+1}\subset \Hom_{\operatorname{cont}}((\Z_p^\times)^{n+1},\Z_p^{\times})$ is dense, where the embedding is given by sending $(k_1,\dots,k_{n+1})$ to the character $(x_1,\dots, x_{n+1})\mapsto \prod x_i^{k_i}$. 
Hence for $M=H^{\#J}_J(\Omega^{(\kappa_1,\kappa_2)})$ or $M=H^{\#J^c}_{J^c}(\Omega^{(2-\kappa_1,-\kappa_2)}(-D))$, the map $M\to \prod_{(\underline{k},w)\in \Z^{n+1}}M\otimes_{(\underline{k},w)}R$ is injective and so is the map $\Lambda\to \prod_{(\underline{k},w)\in \Z^{n+1}}R$. Hence to prove the identity claimed in (i), it is enough to prove it for the pairing specialised in weight $(\underline{k},w)$ for every $(\underline{k},w)\in \Z^{n+1}$, which is a pairing 
\[
H^{\#J}_J(\myom^{(2\underline{k},2w)})\times H^{\#J^c}_{J^c}(\myom^{(2-2\underline{k},-2w)}(-D)) \to R.
\]
Then the statement follows using Lemma \ref{lemmaFpUpdual}.

In order to prove that this pairing is perfect, it is enough to prove the commutativity of the diagram in (ii), since the bottom pairing is perfect and the $\Lambda$-modules $M^{\#J}_J(\Omega^{(\kappa_1,\kappa_2)}), M^{\#J^c}_{J^c}(\Omega^{(2-\kappa_1,-\kappa_2)}(-D))$ are projective by Proposition \ref{propisomodnoneis}. By construction, we have a commutative diagram
\[
\begin{tikzcd}[column sep=1em]
H^{\#J}_J(\Omega^{(\kappa_1,\kappa_2)})\otimes_{\Lambda, (\underline{k},w)} R \arrow{d}{} &\times &H^{\#J^c}_{J^c}(\Omega^{(2-\kappa_1,-\kappa_2)}(-D))\otimes_{\Lambda, (\underline{k},w)} R\arrow{dr} \\
 \limm_i(\colim_{m_\tau})_{\tau\not\in J} H^{\#J}(X_i, \prod_{\tau\not\in J}\I_\tau^{-m_{\tau}}\myom^{(\underline{k},w)})&\times & \limm_i(\limm_{m_\tau})_{\tau\not\in J} H^{\#J^c}(X_i, \prod_{\tau\not\in J}\I_\tau^{m_{\tau}}\myom^{(\underline{k},w)})\arrow{u}{}\arrow{d}{j}\arrow{r} &R.\\
H^{\#J}(X,\myom^{(2\underline{k},2w)})\arrow{u}{i} &\times &H^{n-\#J}(X,\myom^{(2-2\underline{k},-2w)}(-D))\arrow{ur}
\end{tikzcd}
\]
where the vertical maps are the ones obtained in \eqref{eq:map} (for the left ones) and in \eqref{eq:map2} (for the right ones). As before, since the projectors $e(T_J)$ and $e(T_p)$ are the same for our choice of $(\underline{k},w)$ we can write analogous maps for the image of such projectors. We need to check the pairings commute. This follows from (i). Indeed the top square is commutative by construction. For the bottom one, if we take $f\in e(T_p)H^{\#J}(X,\myom^{(2\underline{k},2w)})$ and $g\in \limm_ie(T_J)(\limm_{m_\tau})_{\tau\not\in J} H^{\#J^c}(X_i, \prod_{\tau\not\in J}\I_\tau^{m_{\tau}}\myom^{(2\underline{k},2w)})$, we obtain
\[
\langle e(T_J)i(f),g\rangle \stackrel{(a)}{=} \langle i(f), e(T_{J^c}) g\rangle \stackrel{(b)}{=} \langle i(f),g\rangle = \langle f,j(g)\rangle \stackrel{(b')}{=} \langle e(T_p)f,j(g)\rangle \stackrel{(a')}{=} \langle f, e(T_p)j(g)\rangle,
\]
where for $(a)$ and $(a')$ we used part (i) of the theorem and Proposition \ref{dualtp} respectively and for $(b)$ and $(b')$ the fact that the projectors are idempotent and $g$ lies in the image of $e(T_p)$, $f$ lies in the image of $e(T_J)$ respectively. The remaining equality follows from the commutativity of the bottom part of the above diagram.
\end{proof}


\bibliographystyle{alpha}
\bibliography{biblio}

\Addresses

\end{document}